\documentclass[letterpaper, 11pt]{amsart}
\usepackage{mathtools}
\usepackage{amsmath}
\usepackage{amssymb}
\usepackage{yhmath}
\usepackage{graphicx}
\usepackage{mathrsfs}
\usepackage{bbm}
\usepackage{xcolor}
\usepackage{tikz-cd}
\usepackage{tikz}
\usetikzlibrary{patterns}
\usepackage{hyperref}

\usepackage{verbatim}
\usepackage{float}

\DeclareMathAlphabet{\mathpzc}{OT1}{pzc}{m}{it}

\usepackage{thmtools}
\usepackage{thm-restate}

\usepackage{caption}

\newtheorem{theorem}{Theorem}[section]

\newtheorem*{claim*}{Claim}

\newtheorem{lemma}[theorem]{Lemma}
\newtheorem{lem}[theorem]{Lemma}
\newtheorem{corollary}[theorem]{Corollary}
\newtheorem{Cor}[theorem]{Corollary}
\newtheorem{cor}[theorem]{Corollary}

\newtheorem{proposition}[theorem]{Proposition}

\newtheorem{prop}[theorem]{Proposition}

\theoremstyle{definition}
\newtheorem{definition}[theorem]{Definition}
\newtheorem{Def}[theorem]{Definition}
\newtheorem{example}[theorem]{Example}

\theoremstyle{remark}
\newtheorem{remark}[theorem]{Remark}

\newtheorem{Rmk}[theorem]{Remark}

\numberwithin{equation}{section}

\newcommand{\op}{\operatorname}

\newcommand{\be}{\begin{equation}}
\newcommand{\ee}{\end{equation}}
\newcommand{\Ga}{\Gamma}

\newcommand{\R}{\mathbb R}
\renewcommand{\H}{\mathbb H}
\newcommand{\Z}{\mathbb Z}
\newcommand{\N}{\mathbb N}
\newcommand{\ga}{\gamma}

\newcommand{\la}{\lambda}
\newcommand{\La}{\Lambda}
\newcommand{\inte}{\op{int}}
\newcommand{\ba}{\backslash}

\newcommand{\cal}{\mathcal}
\newcommand{\br}{\mathbb R}
\newcommand{\SO}{\op{SO}}
\newcommand{\Isom}{\op{Isom}}
\newcommand{\PSL}{\op{PSL}}
\newcommand{\F}{\cal F}

\newcommand{\bH}{\mathbb H}

\newcommand{\diam}{\op{diam}}
\newcommand{\vol}{\op{Vol}}

\newcommand{\G}{\Gamma}

\newcommand{\T}{\mathscr T}
\renewcommand{\frak}{\mathfrak}

\newcommand{\e}{\varepsilon}

\renewcommand{\L}{\mathcal L}
\newcommand{\fa}{\mathfrak a}

\renewcommand{\i}{\op{i}}

\renewcommand{\S}{\mathbb S}

\newcommand{\so}{\SO^\circ}

\newcommand{\C}{\cal C}

\newcommand{\id}{\op{id}}

\newcommand{\PGL}{\op{PGL}}
\newcommand{\GL}{\op{GL}}

\renewcommand{\P}{\mathbb{P}}
\newcommand{\ess}{\mathsf{E}}
\newcommand{\B}{\mathcal{B}}
\newcommand{\fg}{\frak g}

\newcommand{\lat}{\La_{\theta}}
\newcommand{\ft}{\F_\theta}

\renewcommand{\epsilon}{\e}

\renewcommand{\d}{\mathsf{d}}

\def\g{\gamma}

\newcommand{\TG}{\mathscr T_\Ga}
\newcommand{\fat}{\fa_\theta^*}
\newcommand{\Ltm}{\L_\theta-\{0\}}
\newcommand{\fh}{\mathfrak{h}}

\title[Ahlfors regularity of Patterson-Sullivan measures]{Ahlfors regularity of Patterson-Sullivan measures of Anosov~groups and applications}

\author{Subhadip Dey}
\address{Department of Mathematics, Yale University, New Haven, CT 06511\linebreak \indent
    {\rm\em Current address:} School of Mathematics, Tata Institute of Fundamental Research, Mumbai, India}
\email{subhadip@math.tifr.res.in}

\author{Dongryul M. Kim}
\address{Department of Mathematics, Yale University, New Haven, CT}
\email{dongryul.kim@yale.edu}

\author{Hee Oh}
\address{Department of Mathematics, Yale University, New Haven, CT and Korea Institute for Advanced Study, Seoul}
\email{hee.oh@yale.edu}

\thanks{Oh is partially supported by the NSF grant No. DMS-1900101 and 2450703.}

\begin{document}


\begin{abstract} 
For all Zarski dense Anosov subgroups of a semisimple real algebraic group, we prove that their limit sets are Ahlfors regular for intrinsic conformal premetrics. As a consequence, we obtain that
a Patterson-Sullivan measure is Ahlfors regular (and hence equal to the Hausdorff measure) if and only if the associated linear form is symmetric.
We also discuss several applications, including analyticity of $(p,q)$-Hausdorff dimensions on the Teichm\"uller spaces, new upper bounds on the growth indicator, and $L^2$-spectral properties of associated locally symmetric manifolds.

\end{abstract}

\maketitle

\section{Introduction}
Let $G$ be a connected semisimple real algebraic group. Let $\Gamma < G$ be a 
discrete subgroup. Patterson-Sullivan measures are certain families of Borel measures on a generalized flag variety, supported on the limit set of $\Gamma$. They  play a crucial role in the study of dynamics on the associated locally symmetric space, especially in the counting and equidistribution of $\Gamma$-orbits of various geometric objects. The original construction is due to Patterson and Sullivan for Kleinian groups  (\cite{Patterson1976limit}, \cite{Sullivan1979density}), which was generalized 
by Quint \cite{Quint2002Mesures} (see \cite{Albuquerque1999patterson} and \cite{Burger_manhattan} for earlier works).

Sullivan showed that for convex cocompact Kleinian groups of $\Isom^+(\bH_\br^n)$, Patterson-Sullivan measures are Ahlfors regular Hausdorff measures on the limit sets in $\S^{n-1}$ \cite[Theorem 8]{Sullivan1979density}. Since Patterson-Sullivan measures are constructed from the weighted Dirac measures on an orbit of $\Gamma$ in the symmetric space $\bH_\br^n$, it is remarkable that they can be given the geometric characterization purely in terms of the internal metric on the limit set of $\Gamma$ which is a subset of the boundary $\partial \bH_\br^n\simeq \S^{n-1}$.

In recent decades, Anosov subgroups have emerged as a higher rank generalization of convex cocompact Kleinian groups. Therefore it is natural to ask when the Patterson-Sullivan measures of Anosov subgroups arise as Ahlfors regular Hausdorff measures on the limit sets with respect to appropriate metrics. The main goal of this paper is to answer this question.

To state our results, 
fix a Cartan decomposition $G=K A^+ K$ where $K < G$ is a maximal compact subgroup and $A^+ \subset A$ is a positive Weyl chamber of a maximal real split torus $A < G$.  We denote by $X$ the associated Riemannian symmetric space $G/K$. Let $\fg$ and $\fa$ denote the Lie algebras of $G$ and $A$ respectively, and set $\fa^+=\log A^+$. Let
 $\Pi$ denote the set of all simple roots of $(\frak g, \frak a)$ with respect to the choice of $\fa^+$.

Fix a non-empty subset $\theta$  of $\Pi$. Let $P_\theta$ be the standard parabolic subgroup of $G$ associated with $\theta$.
The quotient space $$\F_\theta=G/P_\theta$$ is called the $\theta$-boundary of $X$, or a generalized flag variety.   We denote by $\La_\theta$ the limit set  of $\Ga$ in $\F_\theta$ \cite{Benoist1997proprietes}.  For $\theta=\Pi$, we omit the subscript $\theta$ from now on; so in particular, $P=P_\Pi$ is a minimal parabolic subgroup of $G$.  Set $\fa_\theta=\bigcap_{\alpha\in \Pi-\theta} \op{ker}\alpha$ and let $\fa_\theta^*$ denote the dual vector space of $\fa_\theta$. We may think of $\fa_\theta^*$ as a subspace of $\fa^*$ via  the canonical projection $p_\theta: \fa\to \fa_\theta$ \eqref{att}.
For $\psi\in \fa_\theta^*$,
a $(\Ga, \psi)$-{\em Patterson-Sullivan measure}  is a Borel probability measure $\nu$ on $\La_{\theta}$ such that for all $\ga \in \Ga$ and $\xi \in \La_{\theta}$, $$\frac{d \ga_* \nu}{d \nu}(\xi) = e^{\psi(\beta_{\xi}(e, \ga))}$$
where $\beta$ denotes the Busemann map (see \eqref{Bu}). 

A finitely generated subgroup $\Ga<G$ is called {\em $\theta$-Anosov} if
 there exists a constant $C > 1$ such that for all $\alpha \in \theta$, $$\alpha(\mu(\ga)) \ge C^{-1} |\ga| - C \quad \text{for all } \ga \in \Ga$$
where $|\cdot |$ is a word metric on $\Ga$ with respect to a fixed finite generating set and  $\mu:G\to \fa^+$ is the Cartan projection defined
by the condition that $g \in K ( \exp \mu(g)) K$ for all $g\in G$. 
See (\cite{Labourie2006anosov}, \cite{GW_anosov}, \cite{KLP_Anosov}, \cite{KLP_2018}, \cite{GGKW_gt}, \cite{BPS_anosov}, etc.) for other equivalent definitions of Anosov subgroups.

In the rest of the introduction, let $\Ga$ be a non-elementary\footnote{the limit set $\La_{\theta}$ has at least three points.} $\theta$-Anosov subgroup of $G$. We impose the non-elementary assumption on Anosov subgroups for the entire paper. The space of all Patterson-Sullivan measures of $\Ga$ is parameterized  by the set $\mathscr T_\Ga\subset \fa_\theta^*$ of all linear forms  tangent to the $\theta$-growth indicator $\psi_\Ga^\theta$ (Definition \ref{def.growthindicator}): $$ \mathscr T_\Ga=\{\psi\in \fa_\theta^*:\ \psi\ge \psi_\Ga^\theta, \ \psi(u)=\psi_\Ga^\theta(u) \quad \text{for some } u\in \fa_\theta-\{0\} \} .$$ More precisely, for any $\psi\in \mathscr T_\Ga$, there exists a unique $(\Ga, \psi)$-Patterson-Sullivan measure $$\nu_\psi$$ and every Patterson-Sullivan measure of $\Ga$ arises in this way (Theorem \ref{thm.uniquePS}).
Denote by $\L_\theta\subset \fa_{\theta}^+$ the $\theta$-limit cone of $\Ga$, which is the asymptotic cone of $p_{\theta}(\mu(\Ga))$. Then $\T_\G$ is  in bijection with the set $\{\psi\in \fa_\theta^*: \psi>0 \text{ on } \L_\theta-\{0\}\}/\sim$, where $\psi_1\sim \psi_2$ if and only if $\psi_1=c \cdot \psi_2$ for some $c>0$. When the limit cone $\L_\theta$ has non-empty interior (e.g., when $\G$ is Zariski dense in $G$), $\T_\G$ is homeomorphic to $\br^{\# \theta-1}$.
 
\subsection*{Ahlfors regularity and Hausdorff measures} 
Anosov subgroups of a rank one Lie group $G$ are precisely
 convex cocompact subgroups. 
 In general rank one groups, the unique Patterson-Sullivan measure of  $\Ga$ is Ahlfors regular and coincides with the Hausdorff measure on $\La$  with respect to a $K$-invariant {\it sub-Riemannian} metric on the boundary $\partial_\infty X $ which is defined in terms of the Gromov product 
 \cite[Theorem 5.4]{Corlette_Inv}.  Except for the case of $\SO(n,1)$,
 this sub-Riemannian metric is not a Riemannian metric. 

In this paper, we prove an analogous theorem for a general Anosov subgroup. Let $\psi\in \T_\G$.
 The $\theta$-Anosov property of $\Ga$ implies that
any two distinct points of $\La_\theta$ are in general position and hence
 the following defines a premetric\footnote{On a topological space $X$, a premetric $d$ is a positive definite continuous function $d: X \times X \to \R$ such that $d(x, x) = 0$ for all $x \in X$.} on $\La_\theta$: for $\xi, \eta\in \La_\theta$,
\be \label{dpsi}
d_{\psi}(\xi, \eta) =\begin{cases}
e^{-\psi( \cal G(\xi, \eta))} & \quad \text{if } \xi \neq \eta
\\ 
  0 &\quad\text {if } \xi = \eta\end{cases}  
\ee where $\cal G$ is the $\fa$-valued Gromov product (see Definition \ref{def.defgromovprod}). This premetric turns out to be a correct replacement of the sub-Riemannian metric of the rank one case.

For $s > 0$,
 we denote by  $\cal H_{\psi}^s$
the $s$-dimensional Hausdorff measure on $\La_\theta$ with respect to the premetric $d_\psi$, which is a Borel outer measure  \eqref{hhhjun19}.
We write $\cal H_\psi$ for $\cal H_\psi^1$. 
It turns out that the metric properties of the Patterson-Sullivan measure $\nu_\psi$  depend on the symmetricity of   $\psi\in \fa_\theta^*$: $\psi$ is called {\em symmetric} if $\psi$ is invariant under the opposition involution $\i$ of $\fa$ (see \eqref{oppo}).

Our main theorem is as follows:

\begin{theorem} \label{main1} Let $\Ga$ be a non-elementary $\theta$-Anosov subgroup  of $G$. Let $\psi\in \T_\Ga$ be a symmetric linear form. Then the 
Patterson-Sullivan measure $\nu_\psi$
is Ahlfors $1$-regular and equal to the one-dimensional Hausdorff measure  $\cal H_\psi$, up to a constant multiple.
\end{theorem}

The Ahlfors $1$-regularity of $\nu_\psi$ means that  there exists $C\ge 1$ such that for any $\xi\in \La_\theta$ and $0\le r<\op{diameter} (\La_{\theta}, d_\psi)$,
   \be\label{nucom} C^{-1} r\le \nu_\psi (B_{\psi}(\xi, r))\le C r \ee 
   where $B_{\psi}(\xi, r)=\{\eta\in \La_{\theta}: d_\psi(\xi, \eta)<r\}$.The premetric space $(\La_{\theta}, d_{\psi})$ is called Ahlfors $s$-regular for $s > 0$ if it admits an Ahlfors $s$-regular Borel measure (see Definition \ref{aaa}). Noting that  $\cal H_{s \psi}=\cal H_\psi^s$ for $s>0$, the reason that the Patterson-Sullivan measure is the {\it one-dimensional} Hausdorff measure in the above theorem is due to the normalization of $\psi$ made by the choice that $\psi$ is a tangent form, i.e., $\psi\in \T_\Ga$ (see Remark \ref{sss}).

\begin{remark} If $\psi$ has gradient in the interior of $\fa_{\theta}^+$,  then
$\psi$ can be used to define a Finsler metric on $X$ and Dey-Kapovich \cite[Theorem A]{DK_patterson} showed that $\nu_\psi$ is the Hausdorff measure,  without addressing the Ahlfors regularity  (see Remark \ref{rmk.DK}). Note that Hausdorff measures need not be Ahlfors-regular in general. Our approach 
 in this paper is different; indeed, we first establish the Ahlfors regularity of $\nu_\psi$  and deduce the rest as a consequence of this.
\end{remark}
The opposition involution $\i$ of $\fa$ is known to be trivial if and only if
$G$ does not have a simple factor of type $A_n$ ($n\ge 2$), $D_{2n+1}$ ($n\ge 2$) or $E_6$ \cite[1.5.1]{Tits_classification}.  When $\i$ is non-trivial,
the symmetricity hypothesis on $\psi$ cannot be removed. In fact, we prove the following (Theorem \ref{thm.nonsymmhmeas}, Theorem \ref{thm.Aregular}, see also Remark \ref{rmk.jan13}): 

\begin{theorem}\label{main0}
   Let $\Ga$ be a Zariski dense $\theta$-Anosov subgroup of $G$. For any non-symmetric $\psi\in \mathscr T_\Ga$,
 the premetric space $(\La_{\theta}, d_{\psi})$
  is Ahlfors $s$-regular for some $0<s<1$ but the Patterson-Sullivan measure
$\nu_\psi $ is not comparable\footnote{Two measures $\nu_1$ and $\nu_2$ are comparable
if $C^{-1} \nu_2 \le \nu_1\le C \nu_2 $ for some $C\ge1$} to any $\cal H_\psi^s$, $s > 0$.
\end{theorem}

\subsection*{Critical exponents and Hausdorff dimensions} 
Denote by $\L_\theta\subset \fa_{\theta}^+$ the $\theta$-limit cone of $\Ga$, which is the asymptotic cone of $p_{\theta}(\mu(\Ga))$.
For $\psi\in \fa_\theta^*$ which is positive on  $\L_{\theta}-\{0\}$, we set
 \be\label{dp00} \delta_\psi:=\delta_\psi(\Ga) = \limsup_{T \to \infty} \frac{\log \# \{ \ga \in \Ga : \psi(\mu(\ga)) \le T\}}{T}.\ee

The Hausdorff dimension of  $(\La_{\theta}, d_{\psi})$ is defined as:
$$\dim_{\psi} \La_{\theta}:=\inf\{s > 0: \cal H_{\psi}^{s}(\La_{\theta}) < \infty\}. $$
A natural question is whether $\dim_\psi \La_\theta$ is equal to $\delta_\psi$. 
\begin{theorem}\label{main2}
Let $\Ga$ be a non-elementary $\theta$-Anosov subgroup  of $G$.
   For any $\psi\in \fa_\theta^*$ which is positive on $\L_{\theta}-\{0\}$,
   we have 
   $$\dim_{\psi} \La_{\theta}= \delta_{\bar \psi} $$
    where $\bar \psi = \frac{\psi + \psi \circ \i}{2}$.
In particular, if $\psi$ is symmetric, then  $\dim_{\psi} \La_{\theta}= \delta_{\psi}$.
\end{theorem}

\begin{remark} \begin{itemize}
    \item We remark that for $\psi$ non-symmetric,
$ \dim_{\psi} \La_{\theta}$ is not equal to $\delta_\psi $ in general (Proposition \ref{prop.mincrit}).
\item As mentioned above, if $\psi$ is symmetric and
the gradient of $\psi$ belongs to the interior of $\fa_{\theta}^+$, then Theorem \ref{main2} is due to Dey-Kapovich \cite[Theorem A(v)]{DK_patterson}. When $G$ is of rank-one, Theorem \ref{main2} is due to Patterson, Sullivan (\cite{Patterson1976limit}, \cite{Sullivan1979density}) and Corlette \cite{Corlette_Inv}.
\end{itemize}
\end{remark}

Together with a work of Bridgeman-Canary-Labourie-Sambarino \cite[Proposition 8.1]{BCLS_gafa}, Theorem \ref{main2} implies that for any $\psi$ non-negative on $\fa_\theta^+$,
$\dim_\psi \La_\theta$ depends real-analytically on $\theta$-Anosov representations (Corollary \ref{analytic}). We describe one concrete example as follows.

\subsection*{$(p,q)$-Hausdorff dimension and Teichm\"uller space} Let $\Sigma$ be a torsion-free uniform lattice of $\PSL_2(\br)$, and let $\op{Teich}(\Sigma)$
be the Teichm\"uller space: $$\op{Teich}(\Sigma)=\left\{\sigma:\Sigma\to \PSL_2(\br): \begin{matrix} \text{a discrete faithful representation}
\end{matrix}\right\} / \sim$$
where the equivalence relation is given by conjugations by elements of $\op{PGL}_2(\br)$. 
It is well-known that $\op{Teich}(\Sigma)\simeq \br^{6g-6}$ where
 $g$ is the genus of the surface $\Sigma \ba \H_\br^2$.
For $\sigma \in \op{Teich}(\Sigma)$, denote by $\La_\sigma\subset \S^1\times \S^1$ the limit set of the self-joining subgroup
$(\id\times \sigma)(\Sigma)=\{(\ga, \sigma(\ga)): \ga \in \Sigma\}$, which is well-defined up to conjugations. The  Hausdorff dimension of $\La_\sigma$ with respect to a Riemannian metric on $\S^1\times \S^1$ is equal to $1$ for any $\sigma\in \op{Teich}(\Sigma)$ \cite[Theorem 1.1]{KMO_HD}.
For any pair $(p,q)$ of positive real numbers, consider the premetric on $\S^1\times \S^1$ given by
$$d_{p,q} ( \xi, \eta) = d_{\S^1}(\xi_1, \eta_1)^p d_{\S^1}(\xi_2, \eta_2)^q$$
for any $\xi=(\xi_1, \xi_2)$ and $\eta= (\eta_1, \eta_2)$ in $\S^1\times \S^1$,
where $d_{\S^1}$ is a Riemannian metric on $\S^1$.
For a subset $S \subset \S^1 \times \S^1$, denote by $\dim_{p,q} S $ the Hausdorff dimension of $S$ with respect to $d_{p, q}$. 
Note that on the diagonal of $\S^1\times \S^1$, $d_{p,q}$ is the $(p+q)$-th power of $d_{\S^1}$ and hence $$\dim_{p,q} \La_{\id}=\frac{1}{p + q}.$$

For each $\sigma \in \op{Teich}(\Sigma)$, denote by $\delta_{p, q}(\sigma)$ the critical exponent of the Poincar\'e series
$s\mapsto \sum_{\ga \in \Sigma} e ^{ - s(pd_{\bH_\br^2}(o, \ga o) + q d_{\bH_\br^2}(o, \sigma(\ga) o))}$. 

\begin{Cor} \label{teich} Let $p,q > 0$.
\begin{enumerate}
    \item 
For any $\sigma \in \op{Teich}(\Sigma)$,
we have $$\dim_{p, q}\La_{\sigma} =\delta_{p,q}(\sigma);$$
\item For any $\sigma \in \op{Teich}(\Sigma)$,
we have $$\dim_{p, q} \La_{\sigma} \le \frac{1}{p + q}$$ and the equality holds if and only if $\sigma = \id$;
\item The map $$\sigma \mapsto \dim_{p, q}\La_{\sigma} $$ is a real-analytic function on $\op{Teich}(\Sigma)$. 
\end{enumerate}
\end{Cor}

Part (2) is an immediate consequence of (1) by the rigidity theorem on
$\delta_{p,q}(\sigma)$, due to Bishop and Steger \cite[Theorem 2]{BS_rigidity} and to Burger \cite[Theorem 1(a)]{Burger_manhattan}.
See Corollary \ref{analytic23} for a more general version for
 convex cocompact representations. If we denote by $f=f_\sigma$ the $\sigma$-equivariant homeomorphism $\S^1\to \S^1$, then
$\La_\sigma=\{(x, f(x)):x\in \S^1\}$
and $\dim_{p,q} \La_\sigma$ can also be understood as the Hausdorff dimension of $\La_\Sigma=\S^1$ with respect to the premetric $d_{\sigma, p,q} (x,y)= d_{\S^1}(x,y)^p d_{\S^1}(f(x), f(y))^q$, $x,y\in \S^1$.

\subsection*{Hausdorff dimension of $\La_\theta$ with respect to a Riemannian metric} 
We denote by $\dim \La_\theta$ the Hausdorff dimension of $\La_\theta$ with respect to a {\it Riemannian metric} on $\F_\theta$; since all Riemannian metrics on $\F_\theta$ are Lipschitz equivalent to each other, this is well-defined. With the exception of $G=\so(n,1)$, $\dim \La$ is not in general equal to the critical exponent of $\Ga$ even in rank one case. See \cite{Dufloux_hausdorff}
for a discussion on this for the case of $G=\op{SU}(n,1)$.

From Theorem \ref{main2}, we derive an estimate on $\dim \La_\theta$ in terms of  critical exponents.
Let $\chi_{\alpha}$ denote the Tits weight of $G$ associated to $\alpha\in \Pi$ as given in \eqref{ttt2}. When $G$ is split over $\br$, $\chi_\alpha$ is simply the fundamental weight associated to $\alpha$. We prove: 

\begin{theorem}\label{m4}
For any $\theta$-Anosov subgroup $\Ga$ of $G$, we have
  $$  \max_{\alpha \in \theta} \delta_{\chi_{\alpha} + \chi_{\i(\alpha)}} \le \dim \La_{\theta} \le   \max_{\alpha \in \theta} \delta_{\alpha} .$$
  Moreover, both the upper and lower bounds are attained by some Anosov subgroups. 
  \end{theorem}

For $G=\PSL_n(\br)$, we have the set of simple roots given by
$$\alpha_{k}(\op{diag}(a_1, \cdots, a_{n}))=a_{k}-a_{k+1}, \quad 1\le k\le n-1.$$ 
When $G=\PSL_n(\br)$ and $\theta=\{\alpha_1\}$, the lower bound in Theorem \ref{m4}
 was  obtained by  Dey-Kapovich \cite{DK_patterson}, and the upper bound by
 Pozzetti-Sambarino-Wienhard \cite{PSW_Lipschitz} (see also \cite{CZZ_entropy}). For some  special class of Anosov subgroups, much sharper bounds are known, see 
(\cite{GMT_Hausdorff}, \cite{PSW_conformality}, \cite{PSW_Lipschitz}, \cite{KMO_HD}).
Recently, Li-Pan-Xu proved that for $G=\PSL_3(\br)$,
$\dim\La_{\alpha_1} $ coincides with the affinity exponent of $\Ga$ \cite{LPX_dimension}. See also (\cite{LO_dichotomy}, \cite{KOW_indicators}) which show that $\La_\theta$ has Lebesgue measure zero in higher rank and \cite{Ledrappier-Lessa}  which shows that $\dim \La_\theta$ has a positive co-dimension for all Zariski dense Anosov subgroups of 
$\PSL_n(\br)$, $n\ge 3$. 

The novelty of Theorem \ref{m4} is that it applies to {\it all} $\theta$-Anosov subgroups of any semisimple real algebraic group.
Since both upper and lower bounds are realized by some Anosov subgroups, Theorem \ref{m4} cannot be improved in this generality. 
A {\em Hitchin subgroup} of $\PSL_n(\br)$ is the image of a representation $\pi: \Sigma \to \PSL_n(\br)$ of a uniform lattice $\Sigma<\PSL_2(\br)$ belonging to the same connected component as $\iota|_\Sigma$ in the character variety $\op{Hom}(\Sigma, \PSL_n(\br))/\sim$ where 
$\iota$ is the irreducible representation of  $\PSL_2(\br)$ into $\PSL_n(\br)$ and the equivalence is given by conjugations. Hitchin subgroups are $\Pi$-Anosov, as was shown by Labourie \cite[Theorem 1.4]{Labourie2006anosov}.
For Hitchin subgroups of $\PSL_n(\br)$,  we have
$\dim \La_{\theta} = 1$ by (\cite{Labourie2006anosov}, \cite[Proposition 1.5]{CZZ_entropy}) and $\delta_{\alpha} = 1$ for all $\alpha \in \Pi$ by \cite[Theorem B]{PS_eigenvalues}. Hence 
$$\dim \La_{\theta} = 1 = \delta_{\alpha}\quad\text{ for all $\alpha \in \theta$.}$$

The upper bound in Theorem \ref{m4} is also obtained for Anosov subgroups of the product of $\so(n, 1)$'s \cite{KMO_HD}. For the lower bound, 
let $\G$ be the image of a uniform lattice $\Sigma$ of $\PSL_2(\br)$ under the embedding $\PSL_2(\R) \hookrightarrow \begin{pmatrix}
    \PSL_2(\R) & 0 \\
    0 & I_{n-2}
\end{pmatrix} < \PSL_n(\R)$ where $I_{n-2}$ is the $(n-2) \times (n-2)$ identity matrix.
Then $\Ga$ is $\{\alpha_1\}$-Anosov. On one hand, the limit set $\La_{\alpha_1}$ of $\Ga$ in $\F_{\alpha_1} = \P(\R^n)$ is the projective line, and hence $\dim \La_{\alpha_1} = 1$. On the other hand, since 
$(\chi_{\alpha_1} + \chi_{\i(\alpha_1)})(\op{diag}(a_1, \cdots, a_n)) = a_1 - a_n$, we have
$$\delta_{\chi_{\alpha_1} + \chi_{\i(\alpha_1)}} = \delta_{\Sigma} = 1=\dim \La_{\alpha_1}.$$ Therefore the lower bound in Theorem \ref{m4} is achieved for this example.

\subsection*{Growth indicator bounds and $L^2$-spectral properties} The growth indicator $\psi_{\Ga}=\psi_\Ga^{\Pi}:\fa\to \br \cup\{-\infty\}$
is a higher rank version of the critical exponent of $\Ga$ that captures the growth rate of $\mu(\Ga)$
in each direction of $\fa$ (Definition \ref{def.growthindicator}).  This was introduced by Quint \cite{Quint2002divergence}.
Denote by $\rho $  the half-sum of all positive roots of $(\fg, \fa)$ counted with multiplicity.
Then for any discrete subgroup $\Ga<G$,  $\psi_\Ga \le 2 \rho$, and if $G$ is simple and has higher rank and $\text{Vol}(\Ga\ba G)=\infty$, then Quint proved a gap theorem  that $\psi_\Ga \le 2\rho-\Theta$ where $\Theta$ denotes the half-sum of all roots
in a maximal strongly orthogonal system of $(\frak g, \frak a)$ (\cite{quint_ka}, \cite{oh}, \cite[Theorem 7.1]{LO_dichotomy}).
We obtain the following bound on $\psi_\Ga$ for Anosov subgroups:

\begin{corollary} \label{maincor5} For any $\theta$-Anosov subgroup $\Ga$ of $G$,
we have
\be\label{p100} 
    \psi_{\Ga} \le \dim \La_{\theta} \cdot  \min_{\alpha \in \theta} (\chi_{\alpha} + \chi_{\i(\alpha)}) \quad\text{on }\fa. 
\ee 

\end{corollary}

Recall that $X=G/K$ denotes the associated Riemannian symmetric space.
The size of $\psi_\Ga$ is closely related to the spectral properties of the locally symmetric space $\Ga\ba X$. Let $\lambda_0(\Ga\ba X)$ denote the bottom of the $L^2$-spectrum of $\Ga\ba X$ (see \eqref{ll}).
    As first introduced by Harish-Chandra \cite{Harish},  a unitary representation $(\pi, \cal H_{\pi})$ of $G$ is {\it tempered}
if all of its matrix coefficients belong to
$L^{2+\epsilon}(G)$ for all $\e>0$, or, equivalently, if
$\pi$ is weakly contained\footnote{$\pi$ is weakly contained in a unitary representation $\sigma$ of $G$ if any diagonal
matrix coefficients of $\pi$ can be approximated, uniformly on compact sets, by convex
combinations of diagonal matrix coefficients of $\sigma$.}
in the regular representation $L^2(G)$. Hence the temperedness of the quasi-regular representation $L^2(\Ga\ba G)$ means that $\Ga\ba G$ looks like $G$ from the $L^2$-viewpoints.
If a discrete subgroup $\Ga$ of $G$ satisfies that $\psi_\Ga\le \rho$, 
then  $L^2(\Ga\ba G)$ is tempered and $\lambda_0(\Ga\ba X)=\|\rho\|^2$ as shown in \cite[Theorem 1.6]{EO_temperedness} for $\Pi$-Anosov groups and in \cite{LWW} in general. Moreover, $\lambda_0(\Ga\ba X)$ is not an $L^2$-eigenvalue (\cite{EO_temperedness}, \cite{EFLO}).
However it is not easy to decide whether $\psi_\Ga \le \rho$ holds or not. We give a criterion for this in terms of $\dim \La_\theta$ using Corollary \ref{maincor5}.

Define $$\mathsf c_\theta:=\min\{c \ge  0:\sum_{\alpha\in \theta} (\chi_\alpha + \chi_{\i(\alpha)}) \le  c\cdot \rho \text{ on } \fa^+\}.$$
We set $\mathsf c_G := \mathsf c_\Pi$. Note that 
$0 < \mathsf c_\theta\le \mathsf c_{G}$ and moreover, if $\theta\cap \i(\theta)=\emptyset$, $\mathsf c_\theta \le \mathsf{c}_G/2$. 

If $G$ is $\br$-split, then $\sum_{\alpha\in \Pi} \chi_\theta= \rho $  \cite[Proposition 29]{Bourbaki}, and hence $\mathsf c_G=2$. In general, we have
\be\label{cg} 0<\mathsf c_G\le 2 \ee
by Lemma \ref{smi} due to Smilga.

\begin{cor} \label{cortempered} Let $\Ga$ be a  $\theta$-Anosov subgroup such that
 $$ \dim \La_{\theta}  \le \frac{\# \theta}{ {\mathsf c}_\theta}.$$
Then $L^2(\Ga \ba G)$ is tempered and $\la_0(\Ga\ba X) = \| \rho\|^2$.
    In particular, the conclusion holds for any $\Pi$-Anosov subgroup with $\dim \La \le
  \frac{\op{rank } G}{ {\mathsf c}_G}$.
   \end{cor}    
  
See Remark \ref{lpint} for a more general statement.
Corollary \ref{cortempered} recovers Sullivan's theorem \cite{Sullivan_Riemannian}
in rank one Lie groups (see Remark \ref{ffinal})
and immediately applies to many examples of Anosov subgroups with limit sets of low Hausdorff dimensions; for example to all $\Pi$-Anosov subgroups of higher rank Lie groups with $\dim \La\le 1$ such as Hitchin subgroups and the image of any positive representation into a real split group (\cite[Propositions 1.5 and 11.1]{CZZ_entropy}). 
Although the conclusion of Corollary \ref{cortempered}  was already known for Hitchin subgroups by \cite{KMO_tent} and \cite{EO_temperedness} relying on the work of \cite{PS_eigenvalues}, we obtain a completely different proof in this paper.
Another application is that
the image of a maximal representation of a surface group into $\op{Sp}_{2n}(\br)$ is a tempered subgroup of  $\op{Sp}_{2n}(\br)$  for $n\ge 3$. Such an image is an $\{\alpha_n\}$-Anosov subgroup of $\op{Sp}_{2n}(\br)$ where $\alpha_n$ is the long simple root of  $\op{Sp}_{2n}(\br)$, 
$\dim \La_{\alpha_n}=1$ by \cite{BI}, and we can directly compute $c_{\alpha_n}\le 1$ for $n\ge 3$.

Since the opposition involution $\i$ of $\PSL_n(\br)$ sends the simple root $\alpha_i$ to $\alpha_{n-i}$ for $1 \le i \le n -1$, we also deduce:
\begin{cor}
   Let $n\ge 3$. 
   If $\Ga<\PSL_n(\br)$ is
   $\{\alpha_i\}$-Anosov with
    $\dim \La_{\alpha_i}\le~1$ for some 
    $i\ne \frac{n}{2}$, 
    then
     $L^2(\Ga \ba \PSL_n(\br))$ is tempered and $\la_0(\Ga\ba X ) = \| \rho\|^2$.
\end{cor}
This corollary applies to any $(1,1,2)$-hyperconvex subgroup whose Gromov boundary is homeomorphic to a circle, since such a subgroup is $\{\alpha_1\}$-Anosov with $\dim \La_{\alpha_1} = 1$
by Pozzetti-Sambarino-Wienhard \cite{PSW_conformality}.
It also applies to  the image of a  purely hyperbolic Schottky representation of the free group $F_k$ on $k$-generators  in $\PSL_n(\br)$ in the sense of Burelle-Treib \cite{burelle2022schottky} by \cite[Proposition 11.1]{CZZ_entropy}.

\subsection*{On the proof of Theorem \ref{main1}}  
The key step is to prove that for a symmetric $\psi\in \T_\Ga$, the Patterson-Sullivan measure $\nu_\psi$ 
is Ahlfors one-regular.
Fix $o=[K]\in X$.  The $\theta$-Anosov property of $\Ga$ implies that
$\Ga$ is a hyperbolic group and that the orbit map $\ga \mapsto \ga o$ is a quasi-isometric embedding that continuously  extends to a $\G$-equivariant homeomorphism between the Gromov boundary $\partial \Ga$ and limit set $\La_\theta $. 
One key feature of a Gromov hyperbolic space is that the Gromov product measures the distance between a fixed point and a geodesic, up to an additive error. The main philosophy of our proof is to establish an analogue of this property, by showing that there is a metric-like function $\d_\psi$ on $\Ga o$ that is closely related to the $\psi$-Gromov product $\psi \circ \cal G$ on the limit set $\La_\theta$. For $\ga_1, \ga_2 \in \Ga$, set \be \label{eqn.dpsiorbitintro}
\d_{\psi}(\ga_1 o, \ga_2 o) = \psi(\mu(\ga_1^{-1} \ga_2)).
\ee
We prove that $\d_\psi$ satisfies the coarse triangle inequality (Theorem \ref{thm.triangle}), using a higher rank Morse lemma due to Kapovich-Leeb-Porti \cite{KLP_2018}: there exists $D>0$ such that for any $ \ga_1, \ga_2, \ga_3 \in \Ga$,
 \be \label{eqn.coarsetriangleintro}
\d_{\psi}(\ga_1 o, \ga_3 o) \le \d_{\psi}(\ga_1 o, \ga_2 o) + \d_{\psi}(\ga_2 o, \ga_3 o) + D.
\ee
This allows us to treat $\d_{\psi}$ as a ``metric" on $\Ga o$.
Moreover $(\Ga o, \d_{\psi})$ has a uniform thin-triangle property. That is, there exists $\delta > 0$ such that for any $\xi_1, \xi_2, \xi_3 \in \Ga \cup \partial \Ga$, the image of the geodesic triangle $[\xi_1, \xi_2] \cup [\xi_2, \xi_3] \cup [\xi_3, \xi_1]$ under the orbit map is $\delta$-thin in the $\d_{\psi}$-metric. On the other hand, since $(\Ga o, \d_{\psi})$ is not a geodesic space in general, the thin-triangle property does not imply that $(\Ga o, \d_{\psi})$ is a Gromov hyperbolic space. 
Nevertheless, investigating fine geometric properties of thin-triangles in $(\Ga o, \d_{\psi})$ leads us to proving that the $\psi$-Gromov product measures the $\d_{\psi}$-distance between $o$ and a geodesic (Proposition \ref{ss}). That is, for $\xi \neq \eta \in \La_{\theta} \simeq \partial \Ga$, \be \label{eqn.gromovdistintro}
\psi(\cal G(\xi, \eta)) = \d_{\psi}(o, [\xi, \eta]o)+O(1)
\ee where $[\xi, \eta]o$ is the image of a bi-infinite geodesic $[\xi, \eta]$ in $\Ga$ connecting $\xi$ and $\eta$ under the orbit map. 
 We also prove that shadows on the Gromov boundary $\partial \Ga$ are comparable to shadows on $\La_\theta$ (Proposition \ref{prop.shadowgromovbooundary}) and use it
 to establish the comparability of the $d_{\psi}$-balls and shadows in $\La_{\theta}$ (Theorem \ref{thm.ballshadowball}): for all large $R > 1$ there exists $c \ge 1$ such that for any $\xi\in \La_\theta$ and $\ga \in \Ga$ on a geodesic ray in $\Ga$ toward $\xi\in \La_\theta\simeq \partial \Ga$ from the identity $e \in \Ga$, we have 
\be\label{cshadow} B_{\psi}(\xi, c^{-1} e^{-\psi(\mu(\ga))}) \subset O_R (o, \ga o) \cap \La_{\theta} \subset B_{\psi}(\xi, c  e^{- \psi(\mu(\ga))})\ee 
where the shadow $O_R(o, \ga o)$ is the set of endpoints of all positive Weyl chambers  based at $o$ passing through the Riemannian ball in $X$ of radius $R > 0$ with center $\ga o$ in $X$. Then the Ahlfors one-regularity of $\nu_\psi$ is deduced by applying the higher rank version of Sullivan's shadow lemma (Lemma \ref{lem.shadow}). 
While positivity of $\cal H_{\psi}(\La_{\theta})$ is a standard consequence of the Ahlfors $1$-regularity,  finiteness of $\cal H_{\psi}(\La_{\theta})$ is not immediate since $d_\psi$ is not a genuine metric. We rely on the Vitali covering type lemma for the conformal premetric $d_{\psi}$ on $\La_{\theta}$ (Lemma \ref{lem.lovitali}).

\subsection*{Organization} \begin{itemize}
    \item In section \ref{sec.prelim}, we review some basic structures of Lie groups and $\theta$-boundaries. The notations set up in this section will be used throughout the paper.
    
    \item In section \ref{sec.PSmeas}, we recall the classification of Patterson-Sullivan measures of Anosov subgroups using tangent forms and some basic properties of Anosov subgroups.

    \item In section \ref{sec.metriclike}, we show  that  for each $\psi\in \fat$ positive on $\L_\theta-\{0\}$, the composition $\psi \circ \mu$
 defines a metric-like function $\d_\psi$ on  the $\Ga$-orbit $\Ga o$. The
coarse triangle inequality of $\d_\psi$ (Theorem \ref{thm.triangle}) is a crucial ingredient of this paper. Its proof makes a heavy use of the notion of diamonds and the Morse lemma due to Kapovich-Leeb-Porti (Theorem \ref{thm:ML}).

    \item In section \ref{sec.confmetric}, we define a conformal premetric $d_\psi$ on the limit set $\La_\theta$  and discuss its basic properties.

    \item Sections \ref{sec.ballinshadow} and \ref{sec.shadowinball} are devoted to the proof of
    the compatibility between  shadows and $d_\psi$-balls in the limit set $\La_\theta$ as in \eqref{cshadow}.

    \item In sections \ref{sec.localsize} and \ref{sec.hausmeas}, we prove
     Theorem \ref{main1}. In section \ref{sec.localsize}, we  prove that for symmetric $\psi \in \fa_{\theta}^*$, the $(\Ga, \psi)$-Patterson-Sullivan measure is Ahlfors one-regular. In
      section \ref{sec.hausmeas}, we prove that Patterson-Sullivan measures for symmetric linear forms are Hausdorff measures on the limit set, up to a constant multiple. We also prove Theorem \ref{main2}.

    \item In section \ref{sec.riem}, we prove Theorem \ref{m4} on the estimate of the Hausdorff dimension of $\La_\theta$ with respect to a Riemannian metric. 
    \item In section \ref{fff}, we obtain an upper bound on the growth indicator and discuss its implications on the temperedness of $L^2(\Ga \ba G)$.
\end{itemize}

\subsection*{Acknowledgement} We would like to thank Ilia Smilga for providing the proof of Lemma \ref{smi}.

\section{Basic structure theory of Lie groups and $\theta$-boundaries} \label{sec.prelim}
Throughout the paper, let $G$ be a connected semisimple real algebraic group, more precisely, $G$ is the identity component $\mathbf G(\br)^\circ$ of the group of real points of a semisimple algebraic group $\mathbf G$ defined over $\br$. In this section, we review some basic facts about the Lie group structure of $G$.
Let $A$ be a maximal real split torus of $G$.
Let $\fg$ and $\fa$ respectively denote the Lie algebras of $G$
and $A$. Fix a positive Weyl chamber $\fa^+ \subset \fa$ and set $A^+=\exp \fa^+$, and a maximal compact subgroup $K< G$ such that the Cartan decomposition $G=K A^+ K$ holds. Let $\Phi=\Phi(\fg, \fa)$ denote the set of all roots and $\Pi$  the set of all  simple roots given by the choice of $\fa^+$. Denote by $N_K(A)$ and $C_K(A)$ the normalizer and centralizer of $A$ in $K$ respectively. The Weyl group $\cal W$ is given by $N_K(A)/C_K(A)$.  
Consider the real vector space $\mathsf E^*=\mathsf X(A)\otimes_\Z \br$
where $\mathsf X(A)$ is the group of all real characters of $A$. and let $\mathsf E$ be its dual.
Denote by $(, )$ a $\cal W$-invariant inner product on $\mathsf E$.
 We
denote by $\{\omega_{\alpha}:\alpha\in \Pi\}$ the (restricted) fundamental weights of $\Phi$ defined by \be\label{fw} 2\frac{(\omega_\alpha, \beta)}{(\beta,\beta)} = c_\alpha \delta_{\alpha, \beta}\ee  where $c_\alpha=1$ if $2\alpha\notin \Phi$ and $c_\alpha=2$
otherwise.

Fix an element $w_0\in N_K(A) $ of order $2$  representing the longest Weyl element so that $\op{Ad}_{w_0}\mathfrak a^+= -\mathfrak a^+$. 
The map \be\label{oppo} \i= -\op{Ad}_{w_0}:\fa\to \fa\ee  is called the opposition involution.
It induces an involution of $ \Phi$ preserving $\Pi$, for which we use the same notation $\i$, so that $\i (\alpha )  = \alpha \circ \i$ for all $\alpha\in \Phi$.

Henceforth, we fix a  non-empty subset $\theta$ of $ \Pi$. 
Let 
\begin{equation*}
\mathfrak{a}_\theta =\bigcap_{\alpha \in \Pi-\theta} \ker \alpha, \quad  \quad \fa_\theta^+  =\fa_\theta\cap \fa^+, \end{equation*}
\begin{equation*} A_{\theta}  = \exp \fa_{\theta}, \quad \text{and} \quad     A_{\theta}^+  = \exp \fa_{\theta}^+. \end{equation*}
 Let $$ p_\theta:\mathfrak{a}\to\mathfrak{a}_\theta$$ denote  the projection invariant under all $w\in \cal W$ fixing $\fa_\theta$ pointwise.

Let $ P_\theta$ denote a standard parabolic subgroup of $G$ corresponding to $\theta$; that is,
$P_{\theta}=L_\theta N_\theta$ where $L_\theta$ is the centralizer of $A_\theta$ and $N_\theta$
is the unipotent radical of $P_\theta$ such that $\log N_\theta$ is generated by
root subgroups associated to all positive roots which are not $\mathbb Z$-linear combinations of $\Pi-\theta$. 

 We set $M_{\theta} = K \cap P_{\theta}=C_K(A_\theta)$.
The Levi subgroup $L_\theta$ can be written as  $L_{\theta} = A_{\theta}S_{\theta}$ where $S_{\theta}$ is an almost direct product of
 a connected semisimple real algebraic subgroup and a compact center.
 Letting $B_\theta=S_\theta \cap A$ and $B_\theta^+=\{b\in B_\theta: \alpha (\log b)\ge 0
 \text{ for all $\alpha\in \Pi-\theta$}\},$
 we have the Cartan decomposition of $S_\theta$: 
$$S_\theta = M_{\theta} B_\theta^+ M_{\theta}.$$
Note that  $A=A_\theta B_\theta$ and $ A^+\subset A_\theta^+ B_\theta^+.$
The space $\fa_\theta^*=\op{Hom}(\fa_\theta, \br)$ can be identified with the subspace of $\fa^*$ consisting of $p_\theta$-invariant linear forms: \be\label{att} \fa_\theta^*=\{\psi\in \fa^*: \psi\circ p_\theta=\psi\}.\ee 
Hence for $\theta_1\subset \theta_2$, we have 
\be\label{inc} \fa_{\theta_1}^*\subset \fa_{\theta_2}^*.\ee

When $\theta=\Pi$, we will omit the subscript. So
 $P=P_\Pi$ is a minimal parabolic subgroup and $P=MAN$.

\subsection*{Cartan projection} 
Recall the Cartan decomposition $G = KA^+ K$, which means that for every $g\in G$, there exists a unique element $\mu(g)\in \fa^+$ such that $g\in K \exp \mu(g) K$.
The map $G\to \fa^+$ given by $g\mapsto \mu(g)$ is called the Cartan projection.
 We have 
\be\label{inverse} \mu(g^{-1})=\i (\mu(g))\quad\text{ for all $g\in G$. }
\ee

 Let $X = G/K$ be the associated Riemannian symmetric space, and set $o = [K] \in X$.  Fix a $K$-invariant norm $\| \cdot \|$ on $\fg$  and a Riemannian distance $d$ on $X$, induced from the Killing form on $\fg$; so that
 $$d(go, ho)=\|\mu(g^{-1}h)\|$$ for any $g, h \in G$.
 For $p \in X$ and $R > 0$, let $B(p, R) $ denote the metric ball $ \{ x \in X : d(x, p) < R\}$.

\begin{lemma} \cite[Lemma 4.6]{Benoist1997proprietes} \label{lem.cptcartan}
For any compact subset $Q \subset G$, there exists a constant $C=C_Q>0$ such that for all $g \in G$, $$\sup_{q_1, q_2\in Q} \| \mu(q_1gq_2) -\mu(g)\| < C .$$  
\end{lemma}

 We then write $$\mu_{\theta} := p_{\theta} \circ \mu :G\to \fa_\theta^+.$$
In view of \eqref{att}, we have $\psi\circ \mu_\theta=\psi \circ \mu$ for all $\psi\in \fa_\theta^*$.

\subsection*{The $\theta$-boundary $\F_\theta$} 
We set $$\F_\theta=G/P_{\theta} \quad\text{and}\quad \F=G/P.$$
Let $$ \pi_\theta:\F\to \F_\theta$$ denote
 the canonical projection map given by $gP\mapsto gP_\theta$, $g\in G$. 
 We set \be\label{xit} \xi_\theta=[P_\theta] \in \F_{\theta}.\ee 
By the Iwasawa decomposition $G=KP=KAN$, the subgroup $K$ acts transitively on $\F_\theta$, and hence
 $\F_\theta\simeq K/ M_\theta.$

We consider the following notion of convergence of a sequence in $G$ (or in $X$) to an element of $\F_\theta$. For a sequence $g_i \in G$, we say $g_i \to \infty$ $\theta$-regularly if $\min_{\alpha\in \theta} \alpha(\mu(g_i)) \to \infty$ as $i \to \infty$.

\begin{definition} \label{fc} For a sequence $g_i\in G$  and $\xi\in \ft$, we write $\lim_{i\to \infty} g_i=\lim g_i o =\xi$ and
 say $g_i $ (or $g_io \in X$) converges to $\xi$ if \begin{itemize}
     \item $g_i \to \infty$ $\theta$-regularly; and
\item $\lim_{i\to\infty} \kappa_{i}\xi_\theta= \xi$ in $\F_\theta$ for some $\kappa_{i}\in K$ such that $g_i\in \kappa_{i} A^+ K$.
 \end{itemize}  
\end{definition}

\subsection*{Points in general position} Let $P_\theta^+$ be the
standard parabolic subgroup of $G$ opposite to $P_\theta$ such that $P_\theta\cap P_\theta^+=L_\theta$. We have $P_\theta^+ =w_0 P_{\i(\theta)}w_0^{-1}$ and hence 
$$\F_{\i(\theta)}=G/P_\theta^+.$$

For $g\in G$, we set
$$g_\theta^+ := gP_{\theta}\quad \text{ and }\quad  g_\theta^- := g w_0 P_{\i(\theta)};$$
as we fix $\theta$ in the entire paper, we write $g^{\pm}=g_\theta^{\pm}$ for simplicity when there is no room for confusion. Hence for the identity $e\in G$,
$(e^+, e^-)=(P_\theta, P_\theta^+)=(\xi_\theta, w_0\xi_{\i(\theta)})$.
The $G$-orbit
of $(e^+, e^-)$ is the unique open $G$-orbit
in $G/P_\theta\times G/P_\theta^+$ under the diagonal $G$-action. 
 We set
\be\label{f2} \F_{\theta}^{(2)}= \{(g_\theta^+, g_\theta^-): g\in G\}.\ee 
 Two elements
$\xi\in \F_\theta$ and $\eta\in \F_{\i(\theta)}$ are said to be in general position (or antipodal) if $(\xi, \eta)\in \F_\theta^{(2)}$.

\subsection*{Busemann maps and Gromov products}
The $\frak a$-valued Busemann map $\beta: \cal F\times G \times G \to\frak a $ is defined as follows: for $\xi\in \cal F$ and $g, h\in G$,
$$  \beta_\xi ( g, h):=\sigma (g^{-1}, \xi)-\sigma(h^{-1}, \xi)$$
where  $\sigma(g^{-1},\xi)\in \fa$ 
is the unique element such that $g^{-1}k \in K \exp (\sigma(g^{-1}, \xi)) N$ for any $k\in K$ with $\xi=kP$.
For $(\xi,g,h)\in \cal F_\theta\times G\times G$, we define
 \be\label{Bu} 
 \beta_{\xi}^\theta (g, h): = 
p_\theta ( \beta_{\xi_0} (g, h)) \quad\text{for $\xi_0\in \pi_\theta^{-1}(\xi)$};
\ee 
this is well-defined independent of the choice of $\xi_0$ 
\cite[Lemma 6.1]{Quint2002Mesures}. 
We also have $\|\beta_\xi^{\theta}(g, h)\|\le d(go, ho)$ for all $g, h\in G$ \cite[Lemma 8.9]{Quint2002Mesures}.
The Busemann map has the following properties: for all $\xi \in \F_{\theta}$ and $g_1, g_2, g_3\in G$,
$$\begin{aligned}
    \text{(Invariance)} \quad & \beta_{\xi}^{\theta}(g_1,  g_2) = \beta_{g_3 \xi}^{\theta}(g_3 g_1, g_3 g_2);\\
    \text{(Cocycle property)} \quad & \beta_{\xi}^{\theta}(g_1, g_2) = \beta_{\xi}^{\theta}(g_1, g_3) + \beta_{\xi}^{\theta}(g_3, g_2).
\end{aligned}$$
For $p, q \in X$ and $\xi \in \F_{\theta}$,  we set $\beta_{\xi}^{\theta}(p, q) := \beta_{\xi}^{\theta}(g, h)$ where $g, h \in G$ satisfies $go = p$ and $ho=q$. It is easy to check that this is well-defined.

\begin{Def}\label{def.defgromovprod}
    For $(\xi, \eta) \in \F_{\theta}^{(2)}$, we define the $\theta$-Gromov product as 
    $$
\cal G^{\theta}(\xi, \eta) = \frac{1}{2} ( \beta_{\xi}^{\theta}(e, g) + \i (\beta_{\eta}^{\i(\theta)}(e, g)))
$$ where $g \in G$ satisfies $(g^+, g^-) = (\xi, \eta)$. This does not depend on the choice of $g$ \cite[Lemma 9.11]{KOW_indicators}.
\end{Def}

\section{Classification of Patterson-Sullivan measures by tangent forms} \label{sec.PSmeas}
Let $G$ be a connected semisimple real algebraic group. We fix a non-empty subset $\theta$ of the set $\Pi$
of all simple roots. Throughout this section, let $\Ga$ be a discrete subgroup of $G$.
When $\Ga$ is $\theta$-Anosov, we have a complete classification of all linear forms $\psi\in \fa_\theta^*$ admitting a $(\Ga, \psi)$-Patterson-Sullivan measure (\cite{LO_invariant} \cite{sambarino2022report}, \cite{KOW_indicators}). The goal of this section is to review this classification, in addition to
recalling some basic notions such as the limit cone and the growth indicator of $\Ga$. We refer to \cite{KOW_indicators} for a more detailed discussion of the material of  this section.

The $\theta$-limit set of  $\Ga$ is defined as follows:
$$\lat=\lat(\Ga):=\{\lim {\ga}_i\in \F_\theta:\ {\ga}_i\in \Ga  \}$$ where $\lim \ga_i$ is defined as in Definition \ref{fc}.  If $\Ga < G$ is Zariski dense, then the limit set $\lat$ is the unique $\Ga$-minimal subset of $\F_{\theta}$ (\cite[Section 3.6]{Benoist1997proprietes}, \cite[Theorem 7.2]{Quint2002Mesures}). Furthermore, if we set $\La=\La_\Pi$,
then $ \pi_{\theta}(\La) = \La_{\theta}$.
For $\psi \in \fa_{\theta}^*$, a Borel probability measure $\nu$ on $\F_{\theta}$ is called a $(\Ga, \psi)$-conformal measure if for all $\ga \in \Ga$ and $\xi \in \F_{\theta}$, $$\frac{d \ga_* \nu}{d \nu}(\xi) = e^{\psi(\beta_{\xi}^{\theta}(e, \ga))}$$ where $\ga_*\nu(B) = \nu(\ga^{-1} B)$ for any Borel $B \subset \F_{\theta}$. A $(\Ga, \psi)$-conformal measure is called a $(\Ga, \psi)$-Patterson-Sullivan measure if it is supported on $\La_{\theta}$. 

In order to discuss which linear forms $\psi$ admit a Patterson-Sullivan measure, we need the definitions of the $\theta$-limit cones and growth indicators.

 The $\theta$-limit cone $\L_{\theta}=\L_\theta(\Ga)$ of $\Ga$ is defined as the asymptotic cone of $\mu_{\theta}(\Ga)$ in $\fa_{\theta}$, that is, $u\in \L_\theta$ if and only if $u=\lim t_i \mu_\theta(\ga_i)$ for some sequences $t_i\to 0$ and $\ga_i\in \Ga$.  If $\Ga$ is Zariski dense, $\L_\theta$ is a convex cone  with non-empty interior by \cite[Section 1.2]{Benoist1997proprietes}. Recalling the convention of dropping the subscript $\theta$ when $\theta=\Pi$, we write $\L=\L_\Pi$. We then have
$p_\theta(\L)=\L_\theta$.

\subsection*{Growth indicators}
We say that $\Ga$ is $\theta$-discrete 
if  the restriction $\mu_\theta|_{\Ga}:\Ga \to \fa_\theta^+$ is proper.
The $\theta$-discreteness of $\Ga$ implies that $\mu_\theta(\Ga)$ is a closed discrete subset of $\fa_\theta^+$. Indeed, $\Ga$ is $\theta$-discrete if and only if the counting measure on $\mu_\theta(\Gamma)$ weighted with multiplicity is a Radon measure on $\fa_\theta^+$.

\begin{Def}[$\theta$-growth indicator (\cite{Quint2002divergence}, \cite{KOW_indicators})] \label{def.growthindicator} For a $\theta$-discrete subgroup $\Ga<G$,
 the $\theta$-growth indicator $\psi_\Ga^{\theta}:\fa_\theta\to [-\infty, \infty] $ is defined as follows: if $u \in \fa_\theta$ is non-zero,
\be\label{gi2} \psi_\Ga^{\theta}(u)=\|u\| \inf_{u\in \cal C}
\tau^\theta_{\mathcal C}\ee 
where $\cal C\subset \fa_\theta$ ranges over all open cones containing $u$, and $\psi_{\Ga}^{\theta}(0) = 0$.
Here $-\infty\le \tau^{\theta}_{\cal C}\le \infty$ denotes the abscissa of convergence of the series ${\mathcal P}^{\theta}_{\cal C}(s)=\sum_{\ga\in \Ga, \mu_\theta(\ga)\in \mathcal C} e^{-s\|\mu_\theta(\ga)\|}$. As mentioned, we simply write $\psi_{\Ga} := \psi_{\Ga}^{\Pi}$.
\end{Def}
This definition is independent of the choice of a norm on $\fa_\theta$. It was proved in (\cite[Theorem 1.1.1]{Quint2002divergence}, \cite[Theorem 3.3]{KOW_indicators}) that 
$$\psi_{\Ga}^{\theta} < \infty, \quad \L_{\theta} = \{\psi_{\Ga}^{\theta} \ge 0\}\quad \text{ and } \quad \psi_{\Ga}^{\theta} > 0 \text{ on }\inte \L_{\theta}$$
where $\inte\L_\theta$ denotes the interior of $\L_\theta$ in the relative topology of $\fa_{\theta}$.
Moreover, $\psi_\Ga^\theta$ is upper semi-continuous and concave.
When $\theta=\i(\theta)$, it follows from \eqref{inverse} that
$\psi_\Ga^\theta$ is $\i$-invariant.

We say a linear form $\psi$ is tangent to $\psi_\Ga^\theta$ (at $u\in \fa_\theta-\{0\}$)
if $\psi\ge \psi_\Ga^\theta$ and $\psi(u)=\psi_\Ga^\theta(u)$.
For any  $u\in \inte \L_\theta$, there exists $\psi\in \fa_\theta^*$ tangent to $\psi_\Ga^\theta$ at $u$. Moreover, for any $\psi\in \fa_\theta^*$ tangent to $\psi_\Ga^\theta$
at an interior direction of $\fa_\theta^+$, there exists a $(\Ga, \psi)$-Patterson-Sullivan measure (\cite[Theorem 8.4]{Quint2002Mesures}, \cite[Proposition 5.9]{KOW_indicators}).

For $\theta$-Anosov subgroups, we have a more precise classification of Patterson-Sullivan measures in terms of tangent forms.

\begin{Def}\label{ta}
A finitely generated subgroup $\Ga<G$ is $\theta$-Anosov if there exists a constant $C > 1$ such that for all $\alpha \in \theta$ and $\ga\in \Ga$, we have \be \label{eqn.defanosov}
\alpha(\mu(\ga)) \ge C^{-1} |\ga| - C 
\ee where $|\cdot|$ denotes a fixed word metric on $\Ga$.
\end{Def}
We recall that all $\theta$-Anosov subgroups are assumed to be non-elementary in this paper. Define
 \be\label{ttt} \mathscr T_\Ga=\{\psi\in \fa_\theta^*:\text{$\psi$ is tangent to $\psi_\Ga^\theta$} \} .\ee 
 By the duality lemma (\cite[Section 4]{Quint_indicator}, \cite[Lemma 4.8]{samb_hyper}), the following can be deduced from \cite[Theorem A]{sambarino2022report}  (see \cite[Theorem 13.2]{KOW_indicators}):

\begin{theorem}\label{strict} Let $\Ga$ be a $\theta$-Anosov subgroup.
 Then
 \begin{enumerate}
     \item $\psi_\Ga^\theta$ is analytic on $\inte \L_{\theta}$, strictly concave on $\L_{\theta}$, and vertically tangent\footnote{It means that there is no linear form tangent to $\psi_\Ga^\theta$ at some $u\in \partial \L_\theta$.};
\item For any $\psi \in \mathscr T_\Ga$, there exists a unique unit vector $u=u_\psi\in \inte \L_\theta$ such that $\psi(u)=\psi_\Ga^\theta(u)$. If $\Ga$ is Zariski dense, the map
$\psi\mapsto u_\psi$ is a bijection between $\mathscr T_\Ga$ and
 $\{u\in \inte \L_{\theta}:\|u\|=1\}$. 
  \end{enumerate}
\end{theorem}

The following theorem was proved  in \cite[Theorem 1.3]{LO_invariant} for $\theta=\Pi$ and $\Ga$ Zariski dense. The general case follows from  \cite[Theorem A]{sambarino2022report} (see \cite[Corollary 1.13]{KOW_indicators} for a more general discussion on conformal measures).

 \begin{theorem}\label{thm.uniquePS} 
 Let $\Ga$ be a  $\theta$-Anosov subgroup.   For any $\psi\in \mathscr T_\Ga$, there exists a unique $(\Ga, \psi)$-Patterson-Sullivan measure on $\La_\theta$
 which we denote by $\nu_\psi=\nu_{\psi, \theta}$. 
 The map $$\psi\mapsto \nu_\psi$$ is a surjection from $\mathscr T_\Ga$ to the space of
 all $\Ga$-Patterson-Sullivan measures.  If $\Ga$ is Zariski dense, then the map $\psi \mapsto \nu_{\psi}$ is bijective. Moreover, if $\psi_1\ne \psi_2$ in $\mathscr \T_\Ga$, then $\nu_{\psi_1}$ and $\nu_{\psi_2}$ are mutually singular to each other.
\end{theorem} 

\begin{Rmk} One immediate consequence of the last statement of Theorem \ref{thm.uniquePS}
is that at most one Patterson-Sullivan measure can be a Hausdorff measure on $\La_\theta$ with respect to a fixed metric (e.g., Riemannian metric).
\end{Rmk}

 When $\psi\in \fa_\theta^*$ is positive on $\L_{\theta}-\{0\}$,
 the abscissa of convergence of the $\psi$-Poincar\'e series $$s \mapsto \sum_{\ga \in \Ga} e^{-s \psi(\mu(\ga))}$$ is a well-defined positive number which we denote  by $\delta_\psi$ \cite[Lemma 4.3]{KOW_indicators}. Equivalently,
$\delta_\psi$ is also given by \eqref{dp00}.
 \begin{lem}\cite[Lemma 4.5]{KOW_indicators} \label{lem.tangentcritexp}
If $\psi\in \fa_\theta^*$ is positive on $\L_\theta-\{0\}$, then
$$\delta_\psi  \psi \in \TG .$$
 In particular, $\psi \in \T_{\Ga}$ if and only if $\delta_{\psi} = 1$.
 \end{lem}

Since $\mu(g^{-1}) = \i(\mu(g))$ for all $g \in G$, we have that $\Ga$ is $\theta$-Anosov if and only if $\Ga$ is $\theta \cup \i(\theta)$-Anosov.  
If $\Ga$ is $\theta$-Anosov, then the canonical projection map $p:\La_{\theta \cup \i(\theta)} \to \La_{\theta}$ is a $\Ga$-equivariant homeomorphism.
Recalling that $\fa_\theta^*$ can be considered as a subset of $\fa_{\theta\cup\i(\theta)}^*$ from \eqref{inc}, we recall the following which will be of use.
\begin{lemma} \cite[Lemma 9.5]{KOW_indicators} \label{lem.cansymm} 
Let $\Ga$ be a $\theta$-Anosov subgroup.
For any $\psi \in \TG$, the measure $\nu_{\psi,\theta}$ coincides with the push-forward of 
 $\nu_{\psi, \theta\cup\i(\theta)}$ by $p$.
\end{lemma}

\subsection*{Gromov hyperbolic space and quasi-isometry} We collect a few basic facts about $\theta$-Anosov subgroups which will be used repeatedly.

Recall that a geodesic metric space\footnote{that is, there exists a geodesic between any two distinct points.} $(Z, d_Z)$ is called a Gromov hyperbolic space if it satisfies a uniformly thin-triangle property, that is, there exists $T > 0$ such that for any geodesic triangle in $Z$, one side of the triangle is contained in the $T$-neighborhood of the union of two other sides.
We denote by $\partial Z$ the Gromov boundary of $Z$, which is the space of  equivalence classes of geodesic rays in $Z$. For any $z_1\ne z_2\in Z\cup \partial Z$,
there may be more than one geodesic connecting $z_1$ and $z_2$.
By the notation $[z_1, z_2]$, we mean  ``a"  geodesic in $Z$ connecting $z_1$  to $z_2$.
For $w\in Z$, the nearest-point projection of $w$ to a geodesic $[z_1,z_2]$ is
any point $w'\in [z_1, z_2]$ satisfying $d_Z(w,w')=\inf \{ d_Z(w, z): z\in [z_1, z_2]\}$. This is coarsely well-defined. We refer to \cite{Bridson1999metric} for basics on Gromov hyperbolic spaces.
Recall that $d$ denotes the Riemannian distance function on $X=G/K$.

\begin{theorem}[{\cite[Corollary 1.6]{KLP_2018}, \cite[Proposition 5.16, Lemma 5.23]{KLP_Anosov}, see also \cite{GGKW_gt}}]
    \label{thm.anosovbasic}\label{homeo}
Let $\Ga$ be a $\theta$-Anosov subgroup. Fix a word metric $\d_\Ga$ on $\Ga$ with respect to a finite symmetric generating set. We have:
    \begin{enumerate}
        \item $(\Ga, \d_\Ga)$ is a Gromov hyperbolic space\footnote{note that the metric space $(\Ga, \d_\Ga)$ is clearly a proper geodesic space};
        
        \item  $\L_\theta-\{0\}$ is contained in the relative interior of $\fa_\theta^+$ in $\fa_\theta$;
        
        \item The orbit map
$ (\G,\d_\G) \to (\Ga o, d)$ given by $ \g \mapsto \g o $
is a quasi-isometric embedding, i.e.,
 there exist $Q=Q_\Ga\ge 1$ such that for all $\ga_1, \ga_2\in \Ga$, 
$$ Q^{-1}\cdot  \d_{\Ga}(\ga_1, \ga_2) -Q \le d(\ga_1 o, \ga_2 o) \le Q \cdot \d_{\Ga}(\ga_1, \ga_2) +Q ;$$ 
\item 
The orbit map $\Ga \to \Ga o$ uniquely extends  to a $\Ga$-equivariant
continuous map $f:\Ga \cup  \partial \Ga \to  \Ga o \cup \La_{\theta}$ and $f|_{\partial \Ga}$ is a homeomorphism onto $\La_\theta$.
For $\theta=\i(\theta)$, $f$ maps two distinct points of $\partial \Ga$ to points in general position.  
\end{enumerate}
\end{theorem}

 We will henceforth identity $\partial \Ga$ and $\La_\theta$ using $f$. For any $\xi\ne \eta \in \Ga \cup \partial \G$, 
note that $f([\xi, \eta])=[\xi, \eta] o$ is the image of $[\xi, \eta]$ under the orbit map.

\section{Metric-like functions on $\Ga$-orbits and diamonds} \label{sec.metriclike}
We fix a non-empty subset $\theta \subset \Pi$. In this section, we assume that $\theta$ is symmetric, i.e., $\theta = \i(\theta)$. Recall the notation $X=G/K$ and $o=[K]\in X$.

For a linear form $\psi \in \fa_{\theta}^*$,  define $\d_{\psi} : X \times X \to \R$ as follows:  for $g, h \in G$, \be \label{eqn.defmetricsymm}
\d_{\psi}(g o, h o) :=\psi(\mu(g^{-1} h))=\psi(\mu_{\theta}(g^{-1} h)).
\ee
Since the Cartan projection $\mu$ is bi-$K$-invariant,
$\d_\psi$ is a well-defined left $G$-invariant function.

The main goal of this section is to prove the following theorem saying that  when $\Ga$ is $\theta$-Anosov,
$\d_\psi$ behaves like a metric, restricted
to the $\Ga$-orbit $\Ga o$ for a proper class of $\psi$'s:
\begin{theorem} [Coarse triangle inequality]\label{thm.triangle}
 Let $\Ga$ be a  $\theta$-Anosov subgroup. Let $\psi\in \fat$ be such that $\psi>0$ on $\Ltm$.  Then there exists a constant $D=D_\psi>0$ such that for all $\ga_1, \ga_2,\ga \in \Ga$,
    $$\d_\psi(\ga_1 o, \ga_2 o)\le \d_\psi (\ga_1 o, \ga o)+\d_\psi (\ga o, \ga_2 o)+D.$$
\end{theorem}

Indeed, we prove Theorem \ref{thm.triangle} in a greater generality where the orbit $\Ga o$ is replaced by the image of a {\it uniformly regular} quasi-isometric embedding of a geodesic metric space into $X$.

\subsection*{Coarse triangle inequalities for uniformly regular quasi-isometric embeddings} 
We set $\cal W_\theta$ to be the set of all Weyl elements which fix $\fa_\theta$ pointwise.
We define a closed cone $\C$ in $\fa^+$ to be $\theta$-admissible if the following three conditions hold:
\begin{enumerate}
    \item $\C$ is $\i$-invariant: $\i(\C)=\C$;
    \item $\cal W_\theta \cdot \C=\bigcup_{w\in \cal W_\theta} \op{Ad}_w \C $ is convex;
    \item $\C\cap \left(\bigcup_{\alpha\in \theta} \ker \alpha\right) =\{0\}$.
\end{enumerate}

For a $\theta$-admissible cone $\C$, we say that an ordered pair $(x_1, x_2)$ of distinct points in $ X$ is {\em $\C$-regular} if for $g_1, g_2 \in G$ such that $g_1 o = x_1$ and $g_2 o = x_2$, we have 
$$\mu(g_1^{-1} g_2) \in \C.$$ In this case, $x_2 = g_2 o \in g_1 K (\exp \C) o$ and hence for some $g \in g_1 K$, $x_1 = g_1 o = go$ and $x_2 \in g (\exp \C) o$. Note that if $(x_1, x_2)$ is $\C$-regular, then $(x_2, x_1)$ is $\i(\C)$-regular and hence  $\C$-regular by the $\i$-invariance of $\C$. 

\begin{definition} Let
     $(Z, d_Z)$ be a metric space and $f:Z\to X$ be a map.
     For a cone $\C \subset \fa^+$ and a constant $B\ge 0$,
     $f$ is called {\em $(\C,B)$-regular} if the pair 
    $(f(z_1),f(z_2))$ is $\C$-regular for all $z_1, z_2 \in Z$ with $d_Z(z_1,z_2) \ge B$. We simply say $f$ is $\cal C$-regular if it is $(\cal C, B)$-regular for some $B \ge 0$.
\end{definition}

   Theorem \ref{thm.triangle} will be deduced as a special case of the following theorem: we write $\C_\theta=p_\theta(\C)$.
\begin{theorem}\label{thm:CTI}
    Let $Z$ be a geodesic metric space and
    $\C\subset \fa^+$ a $\theta$-admissible cone.
    Let $f: Z\to X$ be a $\C$-regular
    quasi-isometric embedding\footnote{A map $f:(Z, d_Z) \to (Y, d_Y) $ between metric spaces is called a $Q$-quasi-isometric embedding for $Q \ge 1$ if for all $z_1, z_2\in Z$, $ Q^{-1}  d_Z(z_1, z_2) -Q \le d_Y(f(z_1), f(z_2))\le Q d_Z(z_1, z_2) +Q .$ A map is called a quasi-isometric embedding if it is a $Q$-quasi-isometric embedding for some $Q \ge 1$.}.
    If $\psi \in \fat$ is positive on $\C -\{0\}$, then there exists a constant $D=D_\psi \ge 0$ such that for all $x_1,x_2,x_3\in f(Z)$,
    \[
     \d_\psi (x_1,x_3) \le \d_\psi (x_1,x_2) + \d_\psi (x_2,x_3) + D.
    \]
\end{theorem}

We continue to use notation $\inte\fa^+$ and $\inte\fa_\theta^+$ for relative interiors in the topology of $\fa$ and $\fa_\theta$ respectively.
Unless mentioned otherwise, for any {\it proper} cone $\C$ in $\fa^+$ (resp. $\fa_\theta^+$), we denote by $\inte \C$ the interior of $\C$ in the relative topology of $\fa^+$ (resp. $\fa_\theta^+$). 

\subsection*{Proof of Theorem \ref{thm.triangle} assuming Theorem \ref{thm:CTI}} Let $\psi\in \fa_\theta^*$ be such that $\psi>0$ on $\L_\theta-\{0\}$. We will construct a $\theta$-admissible cone $\cal C\subset \fa^+$  such that $\L-\{0\}\subset \inte \C $ and $\psi$ is positive on $\C -\{0\}$.

Since $\theta=\i(\theta)$ by hypothesis, it follows from \eqref{inverse} that $\i|_{\fa_\theta}$ is an involution preserving $\L_\theta$.
 Since $\psi$ is positive on $\L_\theta-\{0\}$ and $\L_{\theta} - \{ 0\} \subset \inte \fa_{\theta}^+$ (Theorem \ref{thm.anosovbasic}(2)), we can choose a 
closed convex cone $\cal C_0  \subset {\rm int}\, \fa_{\theta}^+ \cup \{0\}$ satisfying \begin{enumerate}
    \item  $\L_{\theta}-\{0\} \subset {\rm int}\, \cal C_0$; 
    \item $\i(\C_0)=\C_0$;
\item $\psi>0$ on $\C_0 - \{0\}$. 
\end{enumerate}

 We observe that  $\cal W_\theta \cdot \fa^+ $ is equal to the union of all Weyl chambers containing
  $\fa_\theta^+$, and hence is a convex cone by \cite[Lemma 2.6]{KLP_Anosov}.
  
Let $\alpha\in \theta$. It follows from \eqref{eqn.defanosov} that
$ \L\cap \ker \alpha =\{0\}$. Since  $\ker \alpha \cap \fa^+$ is contained in the boundary of $\cal W_\theta \cdot \fa^+$, it follows from the convexity of $\cal W_\theta\cdot \fa^+$ that
$\cal W_\theta \cdot \fa^+$ is contained in the half space $\{\alpha\ge 0\}.$ Hence both $\inte \fa_\theta^+$ and  $\cal W_\theta \cdot \L -\{0\}$ are contained in the open half-space $\{\alpha>0\}$. Therefore
we can find a linear form $h_\alpha\in \fa^*$ such that
$$(\ker \alpha\cap \fa^+ ) -\{0\} \subset \{h_\alpha<0\} \quad\text{and}\quad 
(\cal C_0\cup  \cal W_\theta \cdot \L) -\{0\}\subset 
\{h_\alpha >0\} .$$

Now set $H:=\bigcap_{\alpha\in \theta, w\in \cal W_\theta} \{h_\alpha\circ \op{Ad}_{w} \ge 0\}$, which is clearly a $\cal W_\theta$-invariant convex cone. By our choice of $h_\alpha$,
$\inte H$ contains $\L-\{0\}$. Since $\theta=\i(\theta)$ and hence
$\cal W_\theta =\cal W_{\i(\theta)}$, we have that $\i(H)$ is also a $\cal W_\theta$-invariant convex cone whose interior contains $\L-\{0\}=\i(\L)-\{0\}$.

Define \[\C\coloneqq 
p_\theta^{-1}(\C_0) \cap \fa^+\cap  H \cap \i(H).\]
By construction, we have $
 \C \cap \left(\bigcup_{\alpha \in \theta} \ker \alpha \right)= \{0\}$.
 In particular, $\C$ is a proper closed cone in $\fa^+$.
Then  $\inte \cal C$ contains $\L-\{0\}$.
Since $\psi>0$ on $\C_0$ and $\psi$ is $p_\theta$-invariant,
$\psi>0$ on $\cal C$. Since $\i(\cal C_0)=\cal C_0$, we
have $\i(\cal C)=\cal C$.
Using the fact that $p_\theta : \fa \to \fa_\theta$ is $\cal W_\theta$-equivariant, we  have
that $$\cal W_\theta \cdot \C = \cal W_\theta \cdot  (p_\theta^{-1}(\C_0) \cap \fa^+) \cap H \cap \i(H)
    = p_\theta^{-1}(\C_0) \cap (\cal W_\theta \cdot \fa^+)\cap H\cap \i(H).$$ 
    Since $p_{\theta}^{-1}(\C_0)$, $\cal W_{\theta} \cdot \fa^+$, $H$, and $\i(H)$ are convex,
    it follows that $\cal W_\theta\cdot  \C $ is convex. 
    
Therefore $\C$ is $\theta$-admissible.
Since the orbit map $ (\G,\d_\G) \to (X, d)$, $\g \mapsto \g o $, is a quasi-isometric embedding by Theorem \ref{thm.anosovbasic}(3) and any open cone
containing $\L$ contains $\mu(\Ga)$ except for finitely many points,
Theorem \ref{thm.triangle} follows from Theorem \ref{thm:CTI} once we prove that
the orbit map
is a $\C$-regular embedding, as below.
\qed

\begin{lem}\label{below}
    Let $\C\subset \fa^+$ be a closed cone such that $\inte \C\supset \L - \{0\}$. Then the orbit map $(\Ga, \d_{\Ga}) \to (X, d)$ is $\C$-regular.
\end{lem}
\begin{proof}
Suppose not. Then there exist two sequences $\{\ga_i\}, \{\ga_i'\} \subset \Ga$ such that $\d_{\Ga}(\ga_i, \ga_i')=|\ga_i^{-1}\ga_i'| > i$ and $\mu(\ga_i^{-1} \ga_i') \not\in \C$ for all $i\ge 1$. Setting $g_i=\ga_i^{-1}\ga_i'\in \G$, we then have that $\frac{\mu(g_i)}{\|\mu(g_i)\|} \notin \C$ for all $i \ge 1$.
Hence no limit of the sequence $\frac{\mu(g_i)}{\|\mu(g_i)\|} $ belongs to $\inte \C$.
On the other hand, since $|g_i|\to \infty$,
we have $\|\mu(g_i)\|\to \infty$ and hence any limit of the sequence
$\frac{\mu(g_i)}{\|\mu(g_i)\|} $ must belong to the asymptotic cone of $\mu(\Ga)$, that is, $\L$.  This yields a contradiction to the hypothesis $\L -\{0\} \subset \inte\C $.
\end{proof}

\medskip 
The rest of this section is devoted to the proof of Theorem \ref{thm:CTI}.
We begin by recalling the following theorem; in particular, the metric space $Z$
in Theorem \ref{thm:CTI} is always Gromov hyperbolic.
\begin{theorem} \cite[Theorem 1.4]{KLP_2018} \label{thm.KLPhyp}
   Let $Z$ and $ f:Z\to X$ be as in Theorem \ref{thm:CTI}.
  Then  $Z$ is Gromov hyperbolic. If
     $Z$ is proper in addition, then  $f$ continuously extends to $$ f : \bar Z \to X \cup \F_{\theta}$$ where $\bar Z = Z \cup \partial Z$ is the Gromov compactification and $ f$ maps two distinct points in $\partial Z$ to points in general position.
\end{theorem}

\subsection*{Diamonds}
The notion of diamonds in $X$, due to Kapovich-Leeb-Porti, plays a key role in the proof of Theorem \ref{thm:CTI}. We fix
$$\text{a $\theta$-admissible cone $\C \subset \fa^+$ }$$ in the following. 
  For a $\C$-regular pair $(x_1,x_2)$ of points in $X$, define the {\em $\C$-cone} with the tip at $x_1$ containing $x_2$ to be 
\[
 V_\C(x_1,x_2) = gM_{\theta}(\exp \C) o,
\]
where $g=g(x_1, x_2) \in G$ is any element such that $x_1 = go$ and $x_2 \in g (\exp \C) o$; it is easy to check such $g$ always exists and this definition is independent of the choice of $g$. 
For any $h\in G$, we have  $ hV_\C(x_1,x_2) = V_\C( hx_1, hx_2) $.

\begin{definition}[Diamonds] \label{def.diamond}
  For a $\C$-regular pair $(x_1,x_2)$ of points in $X$,  the  
{\em $\C$-diamond} with tips at $x_1$ and $x_2$ is defined as
\[
 \Diamond_\C(x_1,x_2) = V_\C(x_1,x_2) \cap V_\C(x_2,x_1).
\]
\end{definition}
The $\C$-cones and $\C$-diamonds are convex subsets of $X$, see \cite[Propositions 2.10 and 2.13]{KLP_Anosov}. Note also the equivariance property that
for $h\in G$,  $ h \Diamond_\C(x_1,x_2) =  \Diamond_\C(hx_1, hx_2) $.
It follows that
for any $\C$-regular pair $(x_1, x_2)$,
the diamond $\Diamond_\C(x_1,x_2)$ is of the form $h \Diamond_\C(o,ao)$ for some $a\in \exp \C$
and $h\in G$. Therefore the following example describes all diamonds up to translations.
\begin{example}
 For $a \in \exp \C$, the diamond
    $\Diamond_{\C}(o, a o)$ can be explicitly described as follows.
First note that as we can take $g(o, ao) =e$, we have $V_{\C}(o, a o) = M_{\theta}(\exp \C) o.$ Recalling that
$\i=-\op{Ad}_{w_0}$, we also have 
        $a o = a w_0 o $ and
       $ o  = (aw_0) (w_0^{-1} a^{-1} w_0) o \in aw_0 (\exp \C) o .$
 So we can take $g(ao, o)=aw_0$. 
 Since $w_0 M_{\theta} w_0^{-1} = M_{\theta}$ and $w_0 (\exp \C) w_0^{-1} = \exp (-\C)$,
 we have $V_{\C}(a o, o) = a w_0 M_{\theta} (\exp \C) o = a M_{\theta} \exp (- \C) o$.  Therefore  $$\Diamond_{\C}(o, a o) = M_{\theta} (\exp \C) o \cap a M_{\theta} \exp(-\C) o.$$ See Figure \ref{fig.diamondinfa}.
    \begin{figure}[h]
    \begin{tikzpicture}[scale=1.0, every node/.style={scale=1}]
   
        \draw[draw = white, fill=green!10, opacity = 0.5] (0, 0) -- (4, 1) -- (4, 4) -- (2, 4) -- (0, 0);
        \draw (0, 0) -- (4, 1);
        \draw (0, 0) -- (2, 4);
        \draw (4, 3.8) node[left]{$\C$};

        \draw[draw = white, fill=orange!10, opacity = 0.5] (3.5, 3) -- (1.5, -1) -- (-1, -1) -- (-1, 1.875) -- (3.5, 3);
        \draw (3.5, 3) -- (1.5, -1);
        \draw (3.5, 3) -- (-1, 1.875);
        \draw (-0.5, -0.8) node[right]{$\log a - \C$};

        \filldraw (3.5, 3) circle(1pt);
        \draw (3.5, 3) node[right]{$\log a$};

        \draw[fill=pink!20, opacity=0.9] (0, 0) -- (16/7, 4/7) -- (3.5, 3) -- (3.5 - 16/7, 3 - 4/7) -- (0, 0);
        \draw (1.75, 1.5) node{$\Diamond_{\C}(o, ao)$};

        \filldraw (0, 0) circle(1pt);
        \draw (0, 0) node[left]{\tiny $0$};
    \end{tikzpicture}
    \caption{Diamond drawn in $\fa$} \label{fig.diamondinfa}
    \end{figure}
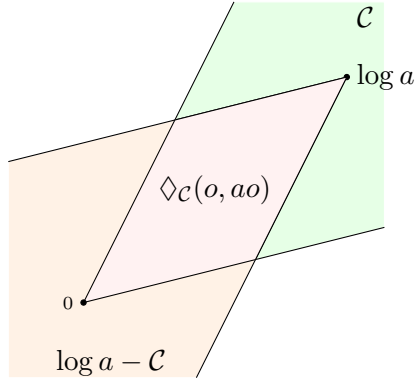
\end{example}

\begin{lemma}[Simultaneously nesting property] \label{lem.nesting}
If $(x_1, x_2)$ is $\C$-regular, then for any $x \in \Diamond_{\C}(x_1, x_2)$, there exist $g \in G$ and $a \in \exp \C$ such that $$x_1 = go, \quad x = ga o  ,\quad x_2 \in  g M_{\theta} (\exp \C) o \cap  g a (M_{\theta} \exp \C) o.$$
\end{lemma}
\begin{proof}
  We may first assume that $x_1 = o$. By the $\C$-regularity of the pair $(x_1, x_2)$, we have $x_2 \in K (\exp \C) o$. By multiplying an element of $K$ to $x_1$ and $x_2$, we may also assume that $x_1 = o$ and $x_2 \in (\exp \C) o$, and hence $x \in M_{\theta} (\exp \C) o$. We again multiply an element of $M_{\theta}$ to $x_1, x_2$ and $x$ if necessary so that we have $x_1 = o$, $x_2 \in M_{\theta}(\exp \C) o$, and $x = ao$ for some $a \in \exp \C$.
    Then it suffices to show that $x_2 = a k a' o$ for some $k \in M_{\theta}$ and $a' \in \exp \C$. We write $x_2 = m a_0 o$ for $m \in M_{\theta}$ and $a_0 \in \exp \C$. We then have $ o \in V_{\C}(x_2, o) = m a_0 k_0 M_{\theta} (\exp \C) o$ for some $k_0 \in K$. Hence we have $$k_0^{-1} w_0^{-1} (w_0 a_0^{-1} w_0^{-1}) \in M_{\theta} (\exp \C) K.$$ This implies $k_0^{-1} \in M_{\theta} w_0$ and hence $k_0 \in w_0 M_{\theta}$. Since $a o \in V_{\C}(x_2, o)$ as well, we now have $a o \in m a_0 w_0 M_{\theta} (\exp \C) o$. Then for some $k \in K$, we have $$a k \in m a_0 w_0 M_{\theta} \exp \C w_0^{-1} = m a_0 M_{\theta} \exp (- \C).$$ Hence for some $a' \in \exp \C$, we have $$a k a' \in m a_0 M_{\theta}.$$ Looking at $G/P_{\theta}$, we have $ k P_{\theta} = a^{-1} m a_0 M_{\theta} a'^{-1} P_{\theta} = P_{\theta}$. Therefore $k \in M_{\theta}$. Since $x_2 = m a_0 o = aka' o$, the claim follows.
\end{proof}

\begin{lem}
    For any $\C$-regular pair $(g_1 o, g_2 o)$ with $g_1, g_2\in G$ and  for any $ go \in  \Diamond_{\cal C}(g_1 o, g_2 o)$ with $g\in G$,
     \be\label{mue} \mu_{\theta}(g_1^{-1}g) + \mu_{\theta}(g^{-1} g_2) = \mu_{\theta}(g_1^{-1}g_2).\ee 
\end{lem}

    \begin{proof} By Lemma \ref{lem.nesting}, 
    there exists $h\in G$, $a, \tilde{a}, a' \in \exp \C$ and $\tilde{k}, k \in M_\theta $
    such that $g_1o=ho$, $go=hao$, and $g_2o= h \tilde{k} \tilde{a} o = ha k a'o$.
     Without loss of generality, we may assume $h=e$ in proving \eqref{mue}.
     
     We write 
$$  a  = a_1 a_2 \in A_{\theta}^+B_{\theta}^+ \quad \text{ and }\quad  a'  = a_1' a_2' \in A_{\theta}^+ B_{\theta}^+. $$
We then have $$aka' = a_2 k a_2' (a_1a_1').$$ Since $a_2 k a_2' \in S_{\theta}$, we can write its Cartan decomposition $a_2 k a_2' = m b m' \in M_{\theta}B_{\theta}^+M_{\theta}$, and hence $$aka' = m (b a_1 a_1')m'.$$
Let $w \in \cal W$ be a Weyl element such that $b a_1 a_1' \in w A^+ w^{-1}$.
Since $\tilde{a} = \exp \mu(aka')$, we must have $b a_1 a_1' = w \tilde{a} w^{-1}$. Hence we have $$g_2 o = aka' o = m w \tilde{a} o.$$ On the other hand, we also have $g_2 o = \tilde{k} \tilde{a} o$ where $\tilde{k} \in M_{\theta}$. This implies $mw \in M_{\theta}$; in particular, $w \in M_{\theta}$. Therefore $\tilde{a} = w^{-1} b a_1 a_1' w = (w^{-1} b w)(a_1 a_1') \in B_{\theta}A_{\theta}^+$, which implies \be \label{eqn.addalmostconc}
p_{\theta}(\log \tilde{a}) = \log a_1 + \log a_1' = p_{\theta}(\log a) + p_{\theta}(\log a').
\ee
Since $$
    \mu_{\theta}(g_1^{-1} g) = p_{\theta}(\log a), \quad
    \mu_{\theta}(g^{-1} g_2)   = p_{\theta}(\log a'), \text{ and} \quad
    \mu_{\theta}(g_1^{-1}g_2)   = p_{\theta}(\log \tilde{a}),
$$
this finishes the proof.
 \end{proof}

As an immediate corollary, we get that $\d_{\psi}$ is additive on each diamond  for any $\psi\in \fat$:
\begin{lemma}[Additivity of $\d_\psi$ on diamonds]\label{dia}
  Let $\psi \in \fa_{\theta}^*$.
   For any $\C$-regular pair $(x_1, x_2)$ and  for any $x\in  \Diamond_{\cal C}(x_1, x_2)$,
     we have  $$\d_\psi(x_1,x)+ \d_{\psi}(x,x_2) = \d_\psi(x_1,x_2).$$
  \end{lemma}

\subsection*{KLP Morse lemma}
The Morse lemma due to Kapovich-Leeb-Porti, which we will call the KLP Morse lemma, is stated as follows \cite[Theorem 5.16, Corollary 5.28]{KLP_2018}: the image of an interval in $\R $ under a $Q$-quasi-isometry is called a $Q$-quasi-geodesic.

\begin{theorem}[KLP Morse lemma]\label{thm:ML} Let $\C,\cal C' \subset \fa^+$ be  $\theta$-admissible closed cones
such that $\inte \cal C'$ contains $ \cal C-\{0\}$. Let $Q, B\ge 1$ be constants.
There exists a constant $D_0 = D_0(\C,\C', Q, B) \ge 0$  so that the following holds: let $I \subset \R$ be an interval and $c : I \to X$ a $(\C, B)$-regular $Q$-quasi-geodesic.
    \begin{enumerate}
        \item If $I = [a, b]$ with $b - a \ge B$, then the image $c(I)$ is contained in the $D_0$-neighborhood of the diamond $\Diamond_{\C'}(c(a),c(b))$.
        \item If $I = [a, \infty)$ for some $a \in \R$, then $c(I)$ is contained in the $D_0$-neighborhood of  the cone $g M_{\theta}(\exp \C') o$ where $g \in G$ is such that $g o = c(a)$ and $g^+ = c(\infty) \in \F_{\theta}$.
        \item If $I = \R$, then $c(I)$ is contained in the $D_0$-neighborhood of  the parallel set $g M_{\theta} A o$ where $g \in G$ is such that $g^{\pm} = c(\pm \infty) \in \F_{\theta}$.
    \end{enumerate}
 We note that the above applies for an interval in $\Z$, as any
  $\C$-regular quasi-isometric embedding $c:I\cap \Z \to X$ can be extended to
 a  $\C$-regular quasi-geodesic $I\to X$ simply by setting $c(t) := c( \lfloor t \rfloor)$
 where $\lfloor t \rfloor$ is the largest integer not bigger than $t$.
\end{theorem}

\begin{figure}[h]
\begin{tikzpicture}[scale=0.9]
    \draw (-2, 0) -- (2, 0) -- (0, 3.4641016151) -- (-2, 0);
    \draw[very thick] (2, 0) -- (0, 3.4641016151);
    \draw (1.2, 3.4641016151/2) node[above] {$\fa_{\theta}^+$};

    \draw[very thick, color=red, fill=red!30, opacity=0.3] (1.7, 0.5196152423) .. controls (-0.5, 0.5) .. (0.3, 2.9444863729) --  (1.7, 0.5196152423);

    \draw[very thick, color=blue, fill=blue!30, opacity=0.5] (1.5, 0.8660254038) .. controls (0, 1) .. (0.5, 2.5980762114) -- (1.5, 0.8660254038);

    \draw[color=blue] (0.75, 1.5) node {$\C$};

    \draw[color=red] (-0.5, 0.5) node {$\C'$};

    \draw[very thick, color=teal] (0, 3.4641016151) -- (-2, 0) -- (2, 0);
    \draw[color=teal] (-2, 0) node[left] {$\bigcup_{\alpha \in \theta} \ker \alpha$};
    \draw[color=white] (2, 0) node[right] {$\bigcup_{\alpha \in \theta} \ker \alpha$};
\end{tikzpicture}
\caption{Choice of $\C'$ viewed on the unit sphere of $\fa^+$} \label{fig.choiceofcones}
\end{figure}
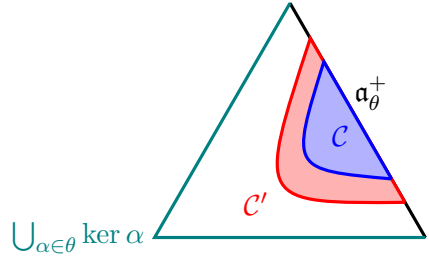

As an application of Theorem \ref{thm:ML}, we get the following:

  \begin{corollary}\label{tri1}
    Given $\cal C, \cal C', Q, B\ge 1$ as in Theorem \ref{thm:ML} and $\psi\in \fa_\theta^*$, there exists a constant $D_1 = D_1(\C,\C', Q, B,\psi) \ge 0$ so that the following holds: let $I \subset \R$ be an interval and $c : I \to X$ a $(\C, B)$-regular $Q$-quasi-geodesic. Then
 for all $a \le t \le b$ in $I$, we have
    \be \label{ca}
    |\d_\psi( c(a), c(b)) - \d_\psi( c(a), c(t)) - \d_\psi(c(t), c(b))| \le D_1.
    \ee 
\end{corollary}

\begin{proof}
  Suppose that $b - a \ge B$. Since $c$ is $(\cal C, B)$-regular,
    the pair $(c(a), c(b))$ is $\cal C$-regular. 
    Applying Theorem \ref{thm:ML}, we obtain that the image of $ c : [a,b] \to X$ lies in the $D_0$-neighborhood of $\Diamond_{\mathcal C'}(c(a), c(b))$. For each $a < t < b$, 
    choose $x_t\in  \Diamond_{\mathcal C'}(c(a), c(b))$ so that
    $d(x_t, c(t)) \le  D_0$. Hence
    by Lemma \ref{dia},
    \be \label{eqn.tri1det}
    \d_\psi(c(a),x_t)+ \d_{\psi}(x_t, c(b)) = \d_\psi(c(a),c(b)).
    \ee 

  For each $a \le t \le b$, write $c(t)=g_t o$ and $x_t=h_t o$  for $g_t, h_t\in G$. We then have $\| \mu(h_t^{-1} g_t)\|\le D_0$.
  By applying
  Lemma \ref{lem.cptcartan} to a compact subset $\{g \in G : \|\mu(g)\|\le D_0\}$, 
  we have for all $a<t<b$,
$$        |\d_{\psi}(c(a), x_t) - \d_{\psi}(c(a) , c(t)) |  = | \psi( \mu(g_a^{-1} h_t) - \mu(g_a^{-1} g_t )) | \le C
    $$
  where $C>0$ is a uniform constant depending only on $\psi$ and $D_0$.
    Similarly, we have $|\d_\psi (x_t, c(b)) - \d_\psi  (c(t), c(b)) |\le C $.
By \eqref{eqn.tri1det},  this implies that
$$
    |\d_\psi( c(a), c(b)) - \d_\psi( c(a), c(t)) - \d_\psi(c(t), c(b))| \le  2 C.
    $$

Setting $D_1= 2C + 3 \|\psi\|(QB + Q)$ where $\|\psi\|$ is the operator norm of $\psi$, we have shown that \eqref{ca} holds whenever $b-a\ge B$.
  If $b - a < B$, then the image of $c([a, b])$ has diameter smaller than $Q(b-a) + Q < QB + Q$. Then $$\d_{\psi}(c(t_1), c(t_2)) < \|\psi\| (QB + Q)$$ for all $t_1, t_2 \in [a, b]$, and hence the left hand side of \eqref{ca} is bounded above by $ 3\|\psi\| (QB + Q)\le D_1$. This completes the proof.
\end{proof}

 We are ready to give:
\subsection*{Proof of Theorem \ref{thm:CTI}} Let $f : Z \to X$ be as in Theorem \ref{thm:CTI}.
Let $\psi \in \fa_{\theta}^*$ be such that $\psi >0$ on $ \C - \{0\}$.
 Choose a $\theta$-admissible cone  $\cal C' \subset \fa^+$ 
such that $\inte \cal C'$ contains  $\cal C-\{0\}$ and such that $\psi>0$ on $\C'-\{0\}$.
 Let $x_1, x_2, x_3 \in f(Z)$ be a triple of distinct points. We choose $z_1, z_2, z_3 \in Z$ such that $x_i = f(z_i)$ for $i = 1,2, 3$.
 Choose geodesics $c_1$ and $c_2$ in $Z$ 
 connecting $z_1$ to $z_2$ and $z_2$ to $z_3$ respectively.
  By Theorem \ref{thm.KLPhyp}, $(Z, d_Z)$ is Gromov hyperbolic.
  We denote by $z$ the nearest-point projection of $z_2$ to a geodesic segment connecting $z_1$ and $z_3$. Then by the Gromov hyperbolicity of $(Z, d_Z)$, there exists a uniform constant $\delta > 0$ so that 
  the $\delta$-neighborhood of $z$ intersects both geodesics $c_1$ and $c_2$. We choose two points $y_1 \in c_1$ and $y_2 \in c_2$ which are $\delta$-close to $z$. We concatenate the segment of $c_1$ connecting $z_1$ and $y_1$, a geodesic connecting $y_1$ and $y_2$, and the segment of $c_2$ connecting $y_2$ and $z_3$, and denote the concatenated path by $c$. We can parameterize $c : [0, b] \to Z$ so that $c$ is a $q$-quasi-geodesic for some  $b > 0$ and uniform $q \ge 1$ by the Gromov hyperbolicity of $(Z, d_Z)$ and the choice of $y_1$ and $y_2$.

  Since $f$ is a $\C$-regular quasi-isometric embedding, so is $f \circ c$. Hence we  get
\be \label{eqn.jan091}
\d_\psi(x_1, x_3) \le  \d_\psi(x_1, f(y_1)) + \d_\psi( f(y_1) ,x_3) +D_1
\ee where $D_1$ is the constant given by Corollary \ref{tri1}.
Applying Corollary \ref{tri1} to the restriction of $f \circ c$ to the interval $[c^{-1}(y_1), b]$ again, we have 
\be \label{eqn.jan092}
\d_\psi( f(y_1) ,x_3) \le \d_\psi(f(y_1), f(y_2)) + \d_{\psi}( f(y_2), x_3) + D_1.
\ee
Since $d_Z(y_1, y_2) \le 2 \delta$, combining \eqref{eqn.jan091} and \eqref{eqn.jan092} yields
\be \label{eqn.jan09conc1}
\d_{\psi}(x_1, x_3) \le \d_\psi(x_1, f(y_1)) + \d_{\psi}( f(y_2), x_3) + D_1'.
\ee
 where $D_1':= \sup \{ \d_\psi(f(w_1), f(w_2)) :{d_Z(w_1, w_2) \le 2 \delta} \} + 2 D_1<\infty $.

Since $f$ is $\C$-regular and $\psi > 0$ on $\C - \{0\}$, there exists $D_2 > 0$ such that \be \label{eqn.jan09yes}
\d_{\psi}(f(w_1), f(w_2)) \ge - D_2
\ee for all $w_1, w_2 \in Z$; indeed, if $f$ is $(\C, B)$-regular for some $B\ge0$,
 then $\d_{\psi}(f(w_1), f(w_2)) \ge 0$ whenever
$d_Z(w_1, w_2)\ge B$, and 
$\sup \{ |\d_\psi (f(w_1), f(w_2))| :d_Z(w_1, w_2)<B\}  $ is bounded  by a uniform constant depending only on $B$,
the quasi-isometry constant of $f$, and $\|\psi\|$.
 
Hence applying \eqref{eqn.jan09yes} and Corollary \ref{tri1} to $f ( c_1 )$,
we have \be \label{eqn.jan09conc2}
\begin{aligned}
    \d_\psi (x_1, f(y_1)) & \le \d_\psi (x_1, f(y_1))  +\d_\psi(f(y_1),x_2) + D_2\\
    & \le \d_\psi(x_1,x_2)  +D_1 + D_2.
\end{aligned} \ee
Similarly, we also get
 \be \label{eqn.jan09conc3}
 \d_\psi(f(y_2),x_3) \le \d_\psi(x_2,x_3)+D_1 + D_2.
 \ee
 Combining \eqref{eqn.jan09conc1}, \eqref{eqn.jan09conc2} and \eqref{eqn.jan09conc3}, we obtain
 $$\d_{\psi}(x_1, x_3) \le \d_{\psi}(x_1, x_2) + \d_{\psi}(x_2, x_3) + D_1' + 2( D_1 + D_2).$$
This completes the proof of Theorem \ref{thm:CTI}.
\qed

\medskip

We state the following consequence of the KLP Morse lemma applied to Anosov subgroups:
\begin{theorem}[Morse lemma for Anosov subgroups] \label{morse2} Let $\theta=\i(\theta)$.
    Let $\Ga$ be a $\theta$-Anosov subgroup  and $f : \Ga \cup \partial \Ga \to \Ga o \cup \La_{\theta}$ be the extension of the orbit map $\ga \mapsto \ga o$ given in Theorem \ref{thm.anosovbasic}(4). Then there exist a cone $\C \subset \fa^+$ and constants $B, D_0 \ge 1$ such that for any geodesic $[\xi, \eta]$ in $\Ga$, the following holds:
    \begin{enumerate}
        \item If $\xi, \eta \in \Ga$ and $\d_{\Ga}(\xi, \eta) \ge B$, then $f([\xi, \eta])$ is contained in the $D_0$-neighborhood of the diamond $\Diamond_{\C}(f(\xi), f(\eta))$.
        \item If $\xi \in \Ga$ and $\eta \in \partial \Ga$, then $f([\xi, \eta])$ is contained in the $D_0$-neighborhood of $g M_{\theta} (\exp \C) o$ where $g \in G$ is such that $g o = \xi$ and $g P_{\theta} = f(\eta)$.
        \item If $\xi, \eta \in \partial \Ga$, then $f([\xi, \eta])$ is contained in the $D_0$-neighborhood of $g M_{\theta} A o$ where $g \in G$ is  such that $g P_{\theta} = f(\xi)$ and $g w_0 P_{\theta} = f(\eta)$.
    \end{enumerate}
Moreover, the cone $\C$ can be taken arbitrarily close to $\L$ as long as its interior contains $\L - \{0\}$.
\end{theorem}

\begin{proof}
    Let $\C\subset \fa^+$  be the $\theta$-admissible cone as in the proof of Theorem \ref{thm.triangle}, and choose a $\theta$-admissible cone $\C' \subset \fa^+$ whose interior in $\fa^+$ contains $\cal C-\{0\}$. Then by Lemma \ref{below} and Theorem \ref{thm.anosovbasic}(3), the orbit map $f|_{\Ga}$ is a $(\C, B) $-regular $Q$-quasi-isometry between $(\Ga, \d_\Ga)$ and $(\Ga o, d)$ for some $B, Q\ge 1$.  Let $D_0=D_0(\C, \C', Q, B)$
be as given by Theorem \ref{thm:ML}.

Now note that  any geodesic $[\xi,\eta]$ in $(\G, \d_\Ga)$ can be written as $[\xi, \eta]=\{\ga_i:i\in I\}$ for an interval $I$ in $\Z$, and $\iota: i\mapsto \ga_i$ is an isometry between $I$ and $[\xi, \eta]$.
Since  $c:=f \circ \iota $ is a $(\cal C, B)$-regular Q-quasi-geodesic, we can apply Theorem \ref{thm:ML} which implies the above claims (1)-(3) where the cone $\C$ in the statement is given by $\C'$ in this proof.
Note from the proof of Theorem \ref{thm.triangle} that the cone $\C'$ can be taken arbitrarily close to the limit cone $\L$ of $\Ga$ as long as $\inte \C'$ contains $\L - \{0\}$.
\end{proof}

\section{Conformal premetrics on limit sets} \label{sec.confmetric}
Let $\Ga$ be a $\theta$-Anosov subgroup of a connected semisimple real algebraic group $G$.
We assume $\theta = \i(\theta)$ in this section.  Fix a linear form $\psi \in \fa_\theta^*$ positive on $\L_{\theta} - \{0\}$.
The goal of this section is to define  a premetric $d_\psi$ on the limit set $\La_\theta$, which is conformal, almost symmetric, and satisfies almost triangle inequality with bounded multiplicative error.
 We also discuss how this definition can be extended to non-symmetric $\theta$ at the end of the section.

Recall the definition of the Gromov product from Definition \ref{def.defgromovprod}. The $\theta$-Anosov property of $\Gamma$ implies that any two distinct points in $\La_\theta$ are in general position:
if $\xi\ne \eta$ in $\La_\theta$, then $(\xi, \eta)\in \F_\theta^{(2)}$.
Therefore the following premetric on $\La_\theta$ is well-defined:
\begin{Def}
    For $\xi, \eta \in \La_{\theta}$, we set \be\label{eqn.defconfmetric}
d_{\psi}(\xi, \eta) =\begin{cases}
e^{-\psi( \cal G^{\theta}(\xi, \eta))} & \quad \text{if } \xi \neq \eta
\\ 
  0 &\quad\text {if } \xi = \eta.\end{cases}  \ee 
\end{Def}

We first observe the following $\Ga$-conformal property of $d_\psi$:
\begin{lemma} \label{lem.gromovconformal}
    For $\ga \in \Ga$ and $\xi, \eta \in \La_{\theta}$, we have $$d_{\psi}(\ga^{-1}\xi, \ga^{-1}\eta) = e^{\frac{1}{2}\psi(\beta_{\xi}^{\theta}(e, \ga) + \i (\beta_{\eta}^{\theta}(e, \ga )))} d_{\psi}(\xi, \eta).$$ 
\end{lemma}

\begin{proof}
   Let $\xi \neq \eta$, and
     $g \in G$ be such that $g^+ = \xi$ and $g^- = \eta$. Then for any $\ga\in \Ga$, $$\begin{aligned}
        2 \cal G^{\theta}(\ga^{-1} \xi, \ga^{-1} \eta) & = \beta_{\ga^{-1} \xi}^{\theta}(e, \ga^{-1} g ) + \i(\beta_{\ga^{-1} \eta}^{\theta}(e, \ga^{-1} g )) \\
        & = 2 \cal G^{\theta}(\xi, \eta) +\beta_{\xi}^{\theta}(\ga , e) + \i(\beta_{\eta}^{\theta}(\ga , e)) \\
        & = 2 \cal G^{\theta}(\xi, \eta) - \beta_{\xi}^{\theta}(e, \ga) - \i(\beta_{\eta}^{\theta}(e, \ga)).
    \end{aligned}$$ Now the claim follows from the definition of $d_{\psi}$.
\end{proof}

Recall that $\cal G^{\theta}(\xi, \eta)
=\i (\cal G^{\theta}( \eta,\xi))$ for all $\xi,\eta\in \La_\theta$. 
Hence
if $\psi$ is $\i$-invariant, then $d_\psi$ is symmetric.
We have the following in general:
\begin{prop}[Metric-like properties of $d_\psi$] \label{prop.lotriangle}\leavevmode
\begin{enumerate}
    \item There exists $R=R(\psi)>1$ such that  for all $\xi, \eta \in \La_{\theta}$,
    $$R^{-1} d_\psi(\eta,\xi)\le d_\psi(\xi, \eta) \le R\;  d_\psi (\eta,
    \xi) .$$
    
    \item There exists $N=N(\psi) > 0$ such that 
    for all $\xi_1, \xi_2, \xi_3 \in \La_{\theta}$, $$d_{\psi}(\xi_1, \xi_3) \le N (d_{\psi}(\xi_1, \xi_2) + d_{\psi}(\xi_2, \xi_3)).$$
\end{enumerate}
\end{prop}
The second property was 
 obtained in \cite[Lemma 6.11]{LO_invariant} and the same proof can be repeated for a general $\theta$ in verbatim.
The first property follows from Lemma \ref{lem.gromovsym} below.  For  $x\ne  y$ in the Gromov boundary $\partial \Ga$ and a bi-infinite geodesic $[x, y]$ in $\Ga$, we denote by $\ga_{x, y} \in [x, y]$
the nearest-point projection of the identity $e$ to $[x,y]$ in $(\Ga, \d_\Ga)$, that is, $\ga_{x, y} \in [x, y]$ is an element such that $\d_\Ga (e, \ga_{x,y})=\inf\{ \d_\Ga (e, g):g\in [x,y]\}$, which is coarsely well-defined.
 Recall the map $f : \Ga \cup \partial \Ga \to \Ga o \cup \La_{\theta}$ from Theorem \ref{thm.anosovbasic}(4). The following was proved in \cite[Lemma 6.6]{LO_invariant} for $\theta=\Pi$ and the same proof works for a general~$\theta$:
\begin{lemma} \label{lem.gromovsym}
    There exists $C_1 > 0$ such that for any $x \neq y \in \partial \Ga$, $$\left\| \cal G^{\theta}(f(x), f(y)) - \frac{\mu_{\theta}(\ga_{x,y}) + \i(\mu_{\theta}(\ga_{x,y}))}{2} \right\| < C_1.$$
    In particular, for $\xi \neq \eta \in \La_{\theta}$, we have $$\| \cal G^{\theta}(\xi, \eta) - \cal G^{\theta}(\eta, \xi) \| < 2C_1.$$
\end{lemma}

\subsection*{Symmetrization} 
Consider the following symmetrization of $\psi\in \fa_\theta^*$:
$$\bar \psi := \frac{\psi + \psi \circ \i}{2}\in \fa_{\theta\cup \i(\theta)}^*.$$ 
Since we are assuming $\theta=\i(\theta)$, we have
$\bar\psi\in \fa_\theta^*$ as well. Since $\L_{\theta}$ is $\i$-invariant, we have $\bar \psi > 0$ on $\L_{\theta} - \{0\}$.
Lemma \ref{lem.gromovsym} implies that $d_{\bar \psi}$ and $d_{\psi}$ are Lipschitz equivalent:

\begin{proposition} \label{prop.bilip}
    There exists $R \ge 1$ such that for any $\xi, \eta \in \La_{\theta}$, we have $$R^{-1} d_{\psi}(\xi, \eta) \le d_{\bar \psi} (\xi, \eta) \le R d_{\psi}(\xi, \eta).$$
\end{proposition}
\begin{proof}
  Since $\cal G^{\theta}(\eta, \xi) = \i ( \cal G^{\theta}(\xi, \eta))$ for all $\eta \ne \xi$ in $\La_\theta$,
 it follows from Lemma \ref{lem.gromovsym} with the constant $C_1$ therein that $$| \psi(\cal G^{\theta} (\xi, \eta)) - \bar \psi ( \cal G^{\theta}(\xi, \eta))| = \frac{1}{2} | \psi (\cal G^{\theta}(\xi, \eta) - \cal G^{\theta}(\eta, \xi)) | < \| \psi \| C_1.$$  It suffices to  set $R = e^{\|\psi\| C_1}$ to finish the proof.
\end{proof}

We also record the following Vitali covering type lemma which is a standard consequence of Proposition \ref{prop.lotriangle}(2) (cf. \cite{LO_invariant}):
here $B_\psi(\xi, r)=\{\eta\in \La_\theta: d_\psi(\xi, \eta)<r\}$.
\begin{lemma} \cite[Lemma 6.12]{LO_invariant} \label{lem.lovitali}
    There exists $N_0 = N_0(\psi) \ge 1$ satisfying the following: for any finite collection $B_{\psi}(\xi_1, r_1), \cdots, B_{\psi}(\xi_n, r_n)$ with $\xi_i \in \La_{\theta}$ and $r_i > 0$ for $i = 1, \cdots, n$, there exists a disjoint subcollection $B_{\psi}(\xi_{i_1}, r_{i_1}), \cdots, B_{\psi}(\xi_{i_k}, r_{i_k})$ such that $$\bigcup_{i = 1}^n B_{\psi}(\xi_i, r_i) \subset \bigcup_{j = 1}^k B_{\psi}(\xi_{i_j}, N_0 r_{i_j}).$$
\end{lemma}

\begin{Rmk}\label{nons}
Recall that the canonical projection $p : \La_{\theta \cup \i(\theta)} \to \La_{\theta}$ is a $\Ga$-equivariant homeomorphism and that $\fa_\theta^*\subset \fa_{\theta\cup \i(\theta)}^*$.
Using this homeomorphism,
we can also define a function $d_{\psi}$ on $\La_{\theta}$ even when $\theta$ is not symmetric, so that $p : (\La_{\theta \cup \i(\theta)}, d_{\psi}) \to (\La_{\theta}, d_\psi)$ is an isometry: $$d_{\psi}(\xi, \eta):= d_{\psi}(p^{-1}({\xi}), p^{-1}(\eta))$$ for all $\xi, \eta \in \La_{\theta}$.
In this regard, the above discussion is still valid without the symmetric hypothesis on $\theta$.
\end{Rmk}

\section{Compatibility of shadows and $d_{\psi}$-balls} \label{sec.ballinshadow}

As before, let $\Ga$ be a  $\theta$-Anosov subgroup of a connected semisimple real algebraic group $G$.
We fix a word metric $\d_{\Ga}$ on $\Ga$.
Fix a linear form $$\psi \in \fa_{\theta}^*$$ which is positive on $\L - \{0\}$ and
$\psi = \psi \circ \i$.
Recall the premetric $\d_\psi$ on $\Ga o$ defined in \eqref{eqn.defmetricsymm} and the conformal premetric
$d_\psi$ on $\La_\theta$ defined by \eqref{eqn.defconfmetric}.
\begin{lem} Both $(\Ga o, \d_\psi)$ and $(\La_\theta, d_\psi)$ are symmetric.
    \end{lem}
\begin{proof}
For $g_1, g_2\in G$, we have 
$$\d_\psi(g_1 o, g_2 o)=\psi(\mu (g_1^{-1}g_2))
=\psi \circ \i (\mu(g_2^{-1} g_1))=\d_\psi (g_2 o, g_1 o).$$
The second claim follows similarly since $\cal G^{\theta \cup \i(\theta)}(\xi, \eta)= \i \cal G^{\theta \cup \i(\theta)}(\eta, \xi)$ for all $\xi, \eta\in \F_{\theta \cup \i(\theta)}$ in general position.  
\end{proof}

Shadows play a basic role in studying the metric property of $(\La_\theta, d_\psi)$ in relation with the geometry of the symmetric space $X$, as in the original work of Sullivan.
We recall the definition of shadows in $\F_{\theta}$.
 For $p, q \in X$, the shadow $O_R^{\theta}(p, q)$ of the Riemannian ball $B(q, R)$ viewed from $p$ is defined as $$ 
\begin{aligned}
    O_R^{\theta}(p,q)   = \{ gP_{\theta} \in \F_{\theta} : g \in G, \ go = p, \ d(q, gA^+o) < R \}.
    \end{aligned}
    $$
We refer to \cite{KOW_ergodic} and \cite{KOW_indicators} for basic properties of these shadows.

 The main technical ingredient of this paper is the following theorem which says that shadows in $\La_\theta$ are comparable with $d_{\psi}$-balls.
\begin{theorem} \label{thm.ballshadowball}  \label{thm.ballinshadow} \label{thm.shadowinball}
    Let $\psi \in \fa_{\theta}^*$ be such that $\psi > 0$  on $\L- \{0\}$ and  $\psi = \psi \circ \i$. Then there exist constants $c, R_0 > 0$ such that for any $R > R_0$, there exists $c' = c'_R > 0$ so that the following holds: for any $\xi \in \La_{\theta}$ and any $g \in \Ga$ on a geodesic ray $[e,\xi]$
    in $\Ga$, we have \be\label{incbs} B_{\psi}(\xi, c e^{-\d_{\psi}(o, go)}) \subset O_R^{\theta}(o, g o) \cap \La_{\theta} \subset B_{\psi}(\xi, c' e^{-\d_{\psi}(o, go)}).\ee 
\end{theorem}

Since the proof of this theorem is quite lengthy, we will prove the first inclusion in this section and the second inclusion in the next section. The rest of this section is devoted to the proof of the first inclusion. In view of Remark \ref{nons}, we assume
that $$\theta = \i(\theta).$$ 
Strictly speaking, $\d_\psi$ is not a metric on the $\Ga$-orbit $\Ga o$. Nevertheless, we will still employ terminologies
for the metric space on $(\G o, \d_\psi)$ for convenience. For instance, for a subset $B\subset \Ga o$, 
$\d_\psi( g o, B)= \inf_{ho\in B}  \d_\psi (go, ho)$ and
the $R$-neighborhood
of $B$ is given by $\{go\in \Ga o: \d_\psi( g o, B)<R\}$, etc.

Two main ingredients of the proof of the first inclusion of \eqref{incbs} are the following, which allow us to treat  $(\Ga o, \d_\psi) $
almost like a Gromov hyperbolic space:
\begin{enumerate}
    \item $(\Ga o, \d_\psi)$ satisfies a triangle inequality up to an additive error (Theorem \ref{thm.triangle});
    \item the $\psi$-Gromov product $\psi(\cal G(\xi, \eta))$ is equal to
    the premetric $\d_\psi (o, [\xi, \eta] o)$ up to an additive error (Proposition \ref{ss}). 
\end{enumerate}
In the rank one case, the property (2) is a well-known consequence of a uniform thin-triangle property and the Morse lemma of the rank one
symmetric space. Higher rank symmetric spaces have neither of these properties. Our proof of (2) is based on the KLP Morse lemma using diamonds as well as
a uniform thin-triangle property of the orbit $(\Ga o, \d_{\psi})$.

We begin with the following:
\begin{prop} \label{prop.thintriangle}\label{Qp}
The orbit map $(\Ga, \d_{\Ga}) \to (\Ga o, \d_{\psi})$, $\ga \mapsto \ga o$,
is a {\it quasi-isometry}, i.e.,  there exists $Q_{\psi} \ge 1$ such that for any $\ga_1, \ga_2 \in \Ga$, $$Q_{\psi}^{-1}\cdot  \d_{\Ga}(\ga_1, \ga_2) - Q_{\psi} \le \d_{\psi}(\ga_1 o, \ga_2 o) \le Q_{\psi} \cdot \d_{\Ga}(\ga_1, \ga_2) + Q_{\psi}.$$
In particular, the images of geodesic triangles in $\Ga$ under the orbit map are uniformly thin, that is, there exists $T_{\psi}>0$ such that for any $\xi_1, \xi_2, \xi_3 \in \Ga \cup \partial \Ga$, the image $[\xi_1, \xi_{3}]o$ is contained in the $T_{\psi}$-neighborhood of $([\xi_1, \xi_2] \cup [\xi_2, \xi_3])o$ with respect to $\d_{\psi}$. 
    \end{prop} 
\begin{proof}
 The second part follows  since $(\Ga,\d_\Ga)$ is a Gromov hyperbolic space (Theorem \ref{thm.anosovbasic}(1)), and hence it has a uniform thin-triangle property which is a quasi-isometry invariance. Since  the orbit map $(\Ga, \d_{\Ga}) \to (\Ga o, d)$ is a quasi-isometry (Theorem \ref{thm.anosovbasic}(3)), the first part of the above proposition follows from the following claim that the identity map $(\Ga o, d)\to (\Ga o, \d_\psi)$ is a quasi-isometry:
there exists $C_\psi\ge 1$ such that 
  for all  $\ga_1, \ga_2 \in \Ga$, we have 
 \be  \label{eqn.qiriempsi}
  C_{\psi} ^{-1} d(\ga_1 o, \ga_2 o)  - C_{\psi} \le \d_{\psi}(\ga_1 o, \ga_2 o) \le C_{\psi} d(\ga_1 o, \ga_2 o) + C_{\psi}.\ee 
 We can take a  cone $\C$ whose relative interior in $\fa^+$ contains $\L - \{0\}$ in $\fa^+$ such that $\psi > 0$ on $\C - \{0\}$.  Hence we can chose $C_1>1$ so that
 $$ C_1^{-1}< \min_{u\in \C, \|u\|=1} \psi(u) \le  \max_{u\in \C, \|u\|=1} \psi(u) <C_1.$$ 
 
 On the other hand, $\mu(\ga) \in \C$ for all but finitely many $\ga \in \Ga$ (Lemma \ref{below}), and hence
 $C_2:=\max\{| \psi (\mu(\ga))|:\mu(\ga)\notin \C\}<\infty$. If we set $C=C_1+C_2$, then 
$$C^{-1} \|\mu(\ga)\| - C \le \psi(\mu(\ga)) \le C \| \mu(\ga) \| + C.$$
 Since both $d$ and $\d_\psi$ are left $\Ga$-invariant, this implies the claim.
\end{proof}

We  use the Morse property to obtain that the image of a geodesic ray under the orbit map has a uniform progression:

\begin{lemma}[Uniform progression lemma] \label{lp}
For any $r>0$, there exists $n_r>0$ such that
for any geodesic ray $\{ \ga_0 = e,\g_1,\ga_2, \cdots\}$ in $(\G,\d_\G)$,
 \[
 \d_\psi (o,\g_{i+n}o) \ge \d_\psi (o,\g_{i}o)+ r
 \]
 for all $i\in \N$ and all $n\ge n_r$.
\end{lemma}

\begin{proof}
    Fix $r>0$.  
    By Theorem \ref{morse2}, there exist a cone $\C \subset \fa^+$ and $B, D_0 \ge 0$ so that
    for all $n \ge B$ and $i \ge 0$,
    the sequence $o, \ga_1 o, \cdots, \ga_{i+n} o$ is contained in the $D_0$-neighborhood of the diamond $\Diamond_{\C}( o,\ga_{i+n} o)$ in $(X, d)$. We may also assume that $\psi > 0$ on $\C - \{0\}$ as $\C$ can be arbitrarily close to $\L$.
    For each $i \ge 0$, choose a point $x_i \in \Diamond_{\C}(o, \ga_{i + n} o)$ which is $D_0$-close to $\ga_i o$.
    Applying Lemma \ref{dia}, we obtain that 
    \begin{equation}\label{eq:lp}
        \d_\psi( o,  x_i)+ \d_\psi(x_i, \ga_{i+n} o) = \d_\psi( o,  \ga_{i+n} o).
    \end{equation}
  Since the orbit map $(\Ga, \d_\Ga)\to (\Ga o, d)$  is a
    $Q_{\psi}$-quasi-isometry by Proposition \ref{prop.thintriangle}, 
    we get that for all $i \ge 0$,
    \begin{equation}\label{eq1:lp}
         \d_\psi( \ga_i o,  \ga_{i+n} o) \ge Q_{\psi}^{-1} \d_\Ga (\ga_i, \ga_{i+n}) -Q_{\psi} =
         Q_{\psi}^{-1} n  -Q_{\psi}.
    \end{equation}
    
    By applying Lemma \ref{lem.cptcartan} to a compact subset $\{ g\in G: \|\mu(g)\| \le D_0\}$, we have
    \be \label{eq2:lp}
    \begin{aligned}
    & |\d_\psi( o,  x_i) - \d_\psi( o, \ga_i o)| \le C  \quad\text{ and}\\
    & |\d_\psi(  x_i, \ga_{i+n} o) - \d_\psi(\ga_i o, \ga_{i+n} o)| \le C 
   \end{aligned}
   \ee
where $C$  depends only on $D_0$ and $\|\psi\|$.
Putting \eqref{eq:lp}, \eqref{eq1:lp}, and \eqref{eq2:lp} together, we get
    \begin{align*}
        \d_\psi( o,  \ga_{i+n} o) &= 
        \d_\psi( o,  x_i)+ \d_\psi(x_i, \ga_{i+n} o)\\
        &\ge \d_\psi( o,  \ga_i o) + \d_\psi( \ga_i o, \ga_{i+n} o) - 2 C \\
        &\ge \d_\psi( o,  \ga_i o) + (Q_{\psi}^{-1}n -Q_{\psi}) - 2 C .
    \end{align*}
    Hence setting $n_r = B + Q_{\psi}(r + 2 C  + Q_{\psi})$ finishes the proof.
\end{proof}

 \begin{lem}[Small inscribed triangle] \label{lem.inscribetri}
    There exists $C > 0$ satisfying the following property: 
    Let $[\xi, \eta]$ be a bi-infinite geodesic in $(\Ga, \d_{\Ga})$. 
    If  $\ga o $ is the nearest-point projection of $o$ to $[\xi, \eta] o$ in the $\d_\psi$-metric, i.e., $\ga \in [\xi, \eta]$ is such that
 $\d_{\psi}(o, \ga o) =  \d_{\psi}(o, [\xi, \eta]o)$,
    then there exist $u \in [e, \xi]$ and $v \in [e, \eta]$ so that  $\{uo, v o, \gamma o\}$ has $\d_\psi$-diameter less than $C$.
 \end{lem}

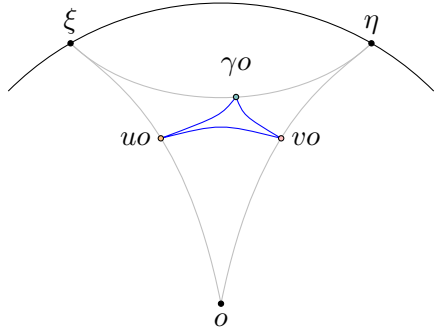
\begin{figure}[h]
\begin{tikzpicture}[scale = 1,  every node/.style={scale=1}]

    \draw (2.8284271247, 2.8284271247) arc(45:135:4);
 \draw[color=gray!50] (0, 0) .. controls (0.5, 2.5) and (1.5, 3) .. (2, 3.4641016151);    \draw[color=gray!50] (2, 3.4641016151) .. controls (1, 2.5) and (-1, 2.5) .. (-2, 3.4641016151);
 \draw[color=gray!50] (-2, 3.4641016151) .. controls (-1.5, 3) and (-0.5, 2.5) .. (0, 0);

    \filldraw (2, 3.4641016151) circle(1pt);
    \filldraw (2, 3.4641016151) node[above] {$\eta$};
    \filldraw (-2, 3.4641016151) circle(1pt);
    \filldraw (-2, 3.4641016151) node[above] {$\xi$};

    \filldraw (0, 0) circle(1pt);
    \draw (0, 0) node[below] {$o$};

    \draw[blue] (0.8, 2.2) .. controls (0, 2.4) .. (-0.8, 2.2) .. controls (0, 2.5) .. (0.2, 2.75) .. controls (0.3, 2.5) .. (0.8, 2.2);

    \filldraw[fill=pink] (0.8, 2.2) circle(1pt);
    \draw (0.8, 2.2) node[right] {$v  o$};

    \filldraw[fill=orange!50] (-0.8, 2.2) circle(1pt);
    \draw (-0.8, 2.2) node[left] {$u o$};

    \filldraw[fill=teal!50] (0.2, 2.75) circle(1pt);
    \draw (0.2, 2.9) node[above] {$\ga o$};

\end{tikzpicture}
\caption{Small inscribed triangle} \label{fig.smallinscribed}
\end{figure}

\begin{proof} Recall from Proposition \ref{prop.thintriangle} that there exists $T_\psi > 0$ so that every triangle in $\Gamma o$, obtained as the image of a geodesic triangle in $(\Ga, \d_{\Ga})$ under the orbit map, is $T_{\psi}$-thin in the $\d_\psi$-metric. 
By the $T_{\psi}$-thinness of $(\Ga o, \d_\psi)$, 
we have either  $\d_{\psi} (\ga o, [e, \xi] o)\le T_{\psi}$ or $\d_{\psi} (\ga o, [e, \eta] o) \le T_{\psi}$. We will assume the latter case; the other case can be treated similarly. We write $[e, \eta] = \{v_i\}_{i \ge 0}$.
We then can choose $j$ so that $j = \min \{ i \ge 0: \d_\psi( \ga o, v_io)\le T_{\psi} + D\}$ where $D$ is given in Theorem \ref{thm.triangle}. Let 
 $n' = n_{3 T_{\psi} { +3D}}$ be the constant from Lemma \ref{lp}. 
 If $j<n'$, then we set $u=v=e$ and note that
 \be \label{eww} 
 \begin{aligned} 
 \d_\psi (\ga o, o) &\le \d_\psi(\ga o, v_j o)+ \d_\psi (v_j o, o) + D \\ 
 &\le D_1:= T_\psi + (n' +1)Q_\psi +2D 
 \end{aligned}
 \ee
 where $Q_\psi$ is given by Proposition \ref{Qp}. Hence the triangle
$\{uo, vo, \ga o\}=\{o, o, \ga o\}$ has $\d_\psi$-diameter at most $D_1$.

Now suppose that $j>n'$. We claim that
$$\d_\psi(v_{j - n'}o, [\xi, \eta] o) > T_{\psi}.$$
Indeed, otherwise, 
$\d_\psi(v_{j - n'}o, \ga'o) \le T_{\psi}$ for some $\ga' \in [\xi, \eta]$, and hence we have 
$$\begin{aligned}
    \d_{\psi}(o, \ga' o) & \le \d_{\psi}(o, v_{j - n'} o) + \d_{\psi}( v_{j - n'} o, \ga' o) {+ D} \\
    &  \le (\d_{\psi}(o, v_{j} o) - 3 T_{\psi} { - 3 D}) + T_{\psi} { + D}\\
    & = \d_{\psi}(o, v_{j} o) - 2 T_{\psi} {-2 D}\\
    & \le \d_{\psi}(o, v_{j} o) - {\d_{\psi}(\ga o, v_{j}o)}   - T_{\psi} {- D}\\
    & \le \d_{\psi}(o, \ga o) - T_{\psi},
\end{aligned}$$
where the first and the last inequalities follow from Theorem \ref{thm.triangle} and the second is from Lemma \ref{lp}. This yields a contradiction to the minimality of $\d_{\psi}(o, \ga o)$, proving the claim. 

\begin{figure}[h]
\begin{tikzpicture}[scale = 1.1, every node/.style={scale=1}]

    \draw (2.8284271247, 2.8284271247) arc(45:135:4);

    \draw[color=gray!50] (0, 0) .. controls (0.5, 2.5) and (1.5, 3) .. (2, 3.4641016151) .. controls (1, 3) and (-1, 3) .. (-2, 3.4641016151) .. controls (-1.5, 3) and (-0.5, 2.5) .. (0, 0);

    \filldraw (0, 0) circle(1pt);
    \draw (0, 0) node[below] {$o$};

    \draw[dotted] (0.8, 2.2) -- (0.2, 3);
    \draw (0.5, 2.6) node {\tiny$ T_{\psi} + D$};

    \draw[dotted] (0.5, 1.3) -- (-0.2, 1.2);
    \draw (0.15, 1.2) node[above] {\tiny$T_{\psi}$};
    
    \filldraw[color=red] (0.2, 0.5) circle(1pt);
    \filldraw[color=red] (0.2, 1) circle(1pt);
    \filldraw[fill=red] (0.5, 1.3) circle(1pt);
    \draw (0.5, 1.1) node[right] {$v_{j-n'} o$};
    \filldraw[color=red] (0.6, 1.7) circle(1pt);
    \filldraw[color=red] (0.7, 1.5) circle(1pt);
    \filldraw[fill=red] (0.8, 2.2) circle(1pt);
    \draw (0.8, 2.2) node[right] {$v_j o$};
    \filldraw[color=red] (0.9, 2) circle(1pt);
    \filldraw[color=red] (1.0, 2.7) circle(1pt);
    \filldraw[color=red] (1.3, 2.8) circle(1pt);
    \filldraw[color=red] (1.5, 2.7) circle(1pt);
    \filldraw[color=red] (1.7, 3) circle(1pt);

    \filldraw[color=orange] (-0.2, 0.8) circle(1pt);
    \filldraw[fill=orange] (-0.2, 1.2) circle(1pt);
    \draw (-0.2, 1) node[left] {$u o$};
    \filldraw[color=orange] (-0.5, 1.5) circle(1pt);
    \filldraw[color=orange] (-0.6, 1.4) circle(1pt);
    \filldraw[color=orange] (-0.7, 1.8) circle(1pt);
    \filldraw[color=orange] (-0.8, 2.5) circle(1pt);
    \filldraw[color=orange] (-0.9, 2.3) circle(1pt);
    \filldraw[color=orange] (-1.0, 2.6) circle(1pt);
    \filldraw[color=orange] (-1.3, 2.9) circle(1pt);
    \filldraw[color=orange] (-1.5, 2.7) circle(1pt);
    \filldraw[color=orange] (-1.7, 3) circle(1pt);

    \filldraw[color=blue] (1.8, 3.5) circle(1pt);
    \filldraw[color=blue] (1.6, 3.5) circle(1pt);
    \filldraw[color=blue] (1.3, 3.2) circle(1pt);
    \filldraw[color=blue] (1.2, 3.3) circle(1pt);
    \filldraw[color=blue] (1, 3.3) circle(1pt);
    \filldraw[color=blue] (0.7, 3.1) circle(1pt);
    \filldraw[color=blue] (0.4, 3.3) circle(1pt);
    \filldraw[fill=blue] (0.2, 3) circle(1pt);
    \draw (0.2, 2.9) node[left] {$\ga o$};
    \filldraw[color=blue] (-0.1, 3.3) circle(1pt);
    \filldraw[color=blue] (-0.3, 3.2) circle(1pt);
    \filldraw[color=blue] (-0.5, 3.3) circle(1pt);
    \filldraw[color=blue] (-0.8, 3) circle(1pt);
    \filldraw[color=blue] (-0.9, 3.1) circle(1pt);
    \filldraw[color=blue] (-1.1, 3.3) circle(1pt);
    \filldraw[color=blue] (-1.4, 3.3) circle(1pt);
    \filldraw[color=blue] (-1.7, 3.3) circle(1pt);

\end{tikzpicture}
\caption{Choice of $\ga o$, $v_j o$, $v_{j-n'} o$ and $u_k o$} \label{fig.choices}
\end{figure}

Since the triangle consisting of the sides 
$[\xi, \eta]o$, $[e, \xi]o$, and $[e, \eta] o = \{v_i o\}_{i \ge 0}$
 is $T_{\psi}$-thin, the above claim implies that  $v_{j - n'} o$ lies in the $T_{\psi}$-neighborhood of $[e, \xi] o$.  Hence there exists  $u \in [e, \xi]$  such that
$\d_\psi(v_{j-n'} o, u o)\le T_{\psi}$ (see Figure \ref{fig.choices}).
Since $\d_{\psi}(v_j, v_{j - n'}) \le Q_{\psi} n' + Q_{\psi}$, we have so far obtained \begin{itemize}
    \item $\d_{\psi}(\ga o, v_j o) \le T_\psi  + D $;
    \item $\d_{\psi}(v_j o, u o) \le Q_{\psi} n' + Q_{\psi} + T_{\psi} + D$;
    \item $\d_{\psi}(\ga o, u o) \le Q_{\psi} n' + Q_{\psi} + 2T_{\psi} + 3D$.
\end{itemize}
Therefore the triangle $\{uo, v_j o, \ga o\}$
has $\d_\psi$-diameter  at most  $D_2 = Q_{\psi} n' + Q_{\psi} + 2T_{\psi} + 3D$. It remains to set $C=\max (D_1, D_2)$.
\end{proof}

 The following was shown for $\theta = \Pi$ in \cite[Lemma 5.7]{LO_invariant} which directly implies the statement for general $\theta$:

    \begin{lemma} \label{lem.buseandcartan}
        There exists $\kappa > 0$ such that for any $g, h \in G$ and $R>0$, we have $$\sup_{\xi \in O^\theta_R(g o, ho)}  \| \beta_{\xi}^{\theta}(go, ho) - \mu_{\theta}(g^{-1}h) \| \le \kappa R.$$
    \end{lemma}
    
We now prove that the $\psi$-Gromov product $\psi(\cal G(\xi, \eta))$ behaves like the distance $\d_\psi (o,
    [\xi, \eta] o)$ up to an additive error:
    
\begin{prop}[Comparison between $\psi$-Gromov product and $\d_{\psi}$-distance] \label{ss}
    There exists $C_1 > 0$ such that
    for any $\xi\ne  \eta \in \La_{\theta}=\partial \Ga$, we have 
     $$|\psi(\cal G^{\theta}(\xi, \eta)) - \d_{\psi}(o, [\xi, \eta] o) | \le C_1.$$
\end{prop}

\begin{proof} Let $\ga\in [\xi, \eta]$  be such that $\d_{\psi}(o, \ga o) = \d_{\psi}(o, [\xi, \eta] o)$.
 Consider geodesic rays $[e, \xi] $ and $[e, \eta] $  in $(\Ga, \d_\Ga)$.
  Let $k, \ell \in K$ and $h \in G$ be such that $kP_{\theta} = \xi$, $\ell P_{\theta} = \eta$, $h P_{\theta}=\xi$ and $h w_0 P_{\theta} =  \eta$.
For the constant $D_0$ given by Theorem \ref{morse2}, we have
    \be \label{eqn.morsetri}
    \begin{aligned}
   & \sup_{u \in [e, \xi]} d(u o, kM_{\theta}A^+ o) \le D_0; \\
    &\sup_{v \in [e, \eta]} d(v o, \ell M_{\theta}A^+ o)  \le  D_0; \\
   & \sup_{g \in [\xi, \eta]}  d (g o, hM_{\theta} A o)  \le D_0.
    \end{aligned}
    \ee 
 Since $\ga\in [\xi, \eta]$ by the choice,
 the third inequality implies
    that $d (\ga o, hM_\theta A o) \le D_0$. We may assume that $h$ satisfies that $d(h o, \ga o) \le D_0$, by replacing $h$ with an element of $hM_{\theta}A$ if necessary.
 
    We first claim that  for some uniform $R > 0$ depending only on $\Ga$ and $\psi$,
    $$\xi, \eta \in O_R^{\theta}(o, \ga o) .$$
To show the claim, let $C > 0$ be the constant given by Lemma \ref{lem.inscribetri} and choose $u \in [e, \xi]$ and $v \in [e, \eta]$ so that the triangle $\{u o, v o, \ga o\}$ has  $\d_{\psi}$-diameter smaller than $C$ (see Figure \ref{fig.shadow}). Hence, for the constant $C' := C_{\psi}(C + C_{\psi})$, where $C_{\psi}$ is given in  \eqref{eqn.qiriempsi}, the Riemannian diameter of the triangle $\{u o, v o, \ga o\}$ is less than $C'$. It then follows from the first two inequalities of \eqref{eqn.morsetri} that $$d(\ga o, k M_{\theta}A^+ o) <  D_0 + C' \quad \text{and} \quad d(\ga o, \ell M_{\theta}A^+ o) <  D_0 + C'.$$ Since $k P_{\theta} = \xi$ and $\ell P_{\theta} = \eta$, we have $$\xi, \eta \in O_{D_0 + C'}^{\theta}(o, \ga o),$$ showing the claim with $R = D_0 + C'$

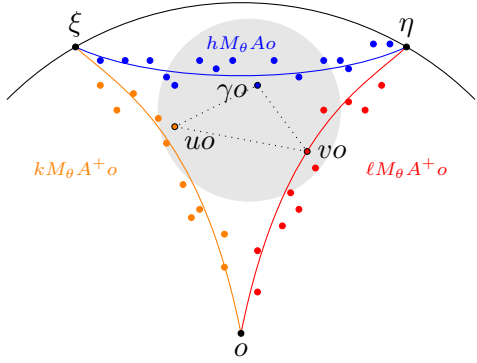
\begin{figure}[h]
\begin{tikzpicture}[scale = 1.1, every node/.style={scale=1}]

    \filldraw[color=gray!20, opacity=0.5] (0.1, 2.7) circle(1.1);

    \draw (2.8284271247, 2.8284271247) arc(45:135:4);

    \draw[color=red] (0, 0) .. controls (0.5, 2.5) and (1.5, 3) .. (2, 3.4641016151);
    \draw[red] (2, 2) node {\tiny$\ell M_{\theta} A^+ o$};

    \draw[color=blue] (2, 3.4641016151) .. controls (1, 3) and (-1, 3) .. (-2, 3.4641016151);
    \draw[blue] (0, 3.5) node {\tiny$hM_{\theta}A o$};

    \draw[color=orange] (-2, 3.4641016151) .. controls (-1.5, 3) and (-0.5, 2.5) .. (0, 0);
    \draw[orange] (-2, 2) node {\tiny$k M_{\theta} A^+ o$};

    \filldraw (2, 3.4641016151) circle(1pt);
    \filldraw (2, 3.4641016151) node[above] {$\eta$};
    \filldraw (-2, 3.4641016151) circle(1pt);
    \filldraw (-2, 3.4641016151) node[above] {$\xi$};

    \filldraw (0, 0) circle(1pt);
    \draw (0, 0) node[below] {$o$};

    \draw[dotted] (0.8, 2.2) -- (-0.8, 2.5) -- (0.2, 3) -- (0.8, 2.2);

    \filldraw[color=red] (0.2, 0.5) circle(1pt);
    \filldraw[color=red] (0.2, 1) circle(1pt);
    \filldraw[color=red] (0.5, 1.3) circle(1pt);
    \filldraw[color=red] (0.6, 1.7) circle(1pt);
    \filldraw[color=red] (0.7, 1.5) circle(1pt);
    \filldraw[fill=red] (0.8, 2.2) circle(1pt);
    \draw (0.8, 2.2) node[right] {$v o$};
    \filldraw[color=red] (0.9, 2) circle(1pt);
    \filldraw[color=red] (1.0, 2.7) circle(1pt);
    \filldraw[color=red] (1.3, 2.8) circle(1pt);
    \filldraw[color=red] (1.5, 2.7) circle(1pt);
    \filldraw[color=red] (1.7, 3) circle(1pt);

    \filldraw[color=orange] (-0.2, 0.8) circle(1pt);
    \filldraw[color=orange] (-0.2, 1.2) circle(1pt);
    \filldraw[color=orange] (-0.5, 1.5) circle(1pt);
    \filldraw[color=orange] (-0.6, 1.4) circle(1pt);
    \filldraw[color=orange] (-0.7, 1.8) circle(1pt);
    \filldraw[fill=orange] (-0.8, 2.5) circle(1pt);
    \draw (-0.5, 2.5) node[below] {$u o$};
    \filldraw[color=orange] (-0.9, 2.3) circle(1pt);
    \filldraw[color=orange] (-1.0, 2.6) circle(1pt);
    \filldraw[color=orange] (-1.3, 2.9) circle(1pt);
    \filldraw[color=orange] (-1.5, 2.7) circle(1pt);
    \filldraw[color=orange] (-1.7, 3) circle(1pt);

    \filldraw[color=blue] (1.8, 3.5) circle(1pt);
    \filldraw[color=blue] (1.6, 3.5) circle(1pt);
    \filldraw[color=blue] (1.3, 3.2) circle(1pt);
    \filldraw[color=blue] (1.2, 3.3) circle(1pt);
    \filldraw[color=blue] (1, 3.3) circle(1pt);
    \filldraw[color=blue] (0.7, 3.1) circle(1pt);
    \filldraw[color=blue] (0.4, 3.3) circle(1pt);
    \filldraw[fill=blue] (0.2, 3) circle(1pt);
    \draw (0.2, 2.9) node[left] {$\ga o$};
    \filldraw[color=blue] (-0.1, 3.3) circle(1pt);
    \filldraw[color=blue] (-0.3, 3.2) circle(1pt);
    \filldraw[color=blue] (-0.5, 3.3) circle(1pt);
    \filldraw[color=blue] (-0.8, 3) circle(1pt);
    \filldraw[color=blue] (-0.9, 3.1) circle(1pt);
    \filldraw[color=blue] (-1.1, 3.3) circle(1pt);
    \filldraw[color=blue] (-1.4, 3.3) circle(1pt);
    \filldraw[color=blue] (-1.7, 3.3) circle(1pt);

\end{tikzpicture}
\caption{The dotted triangle is of diameter less than $C$ and the gray ball has radius $R$.} \label{fig.shadow}
\end{figure}

Therefore by Lemma \ref{lem.buseandcartan}, we get  $$\| \beta_{\xi}^{\theta}(o, \ga o) - \mu_{\theta}(\ga) \| \le \kappa R \quad \text{and} \quad \| \beta_{\eta}^{\theta}(o, \ga o) - \mu_{\theta}(\ga) \| \le \kappa R.$$ 
Since $\beta_{\xi}^{\theta}(o, \ga o) = \beta_{\xi}^{\theta}(o, ho) + \beta_{\xi}^{\theta}(ho, \ga o)$ and $\|\beta_{\xi}^{\theta}(ho, \ga o) \| \le d(ho, \ga o) \le D_0$, we have
$$\|\beta_{\xi}^{\theta}(o, h o) - \mu_{\theta}(\ga) \| \le \kappa R + D_0$$ and similarly $$\|\beta_{\eta}^{\theta}(o, h o) - \mu_{\theta}(\ga) \| \le \kappa R + D_0.$$
Recalling the definition $\cal G^{\theta}(\xi, \eta) = \frac{1}{2}(\beta_{\xi}^{\theta}(o, ho) + \i(\beta_{\eta}^{\theta}(o, ho)))$, and using $\psi = \psi \circ \i$, we obtain that $$|\psi(\cal G^{\theta}(\xi, \eta)) - \d_{\psi}(o, \ga o) | \le \| \psi \| ( \kappa R + D_0),$$
as desired.
\end{proof}

We are now ready to prove the first inclusion in Theorem \ref{thm.ballshadowball} which we formulate again as follows:

\begin{prop} \label{prop.ballinshadowpart}
 There exist constants $c, R_0 > 0$ such that for any $\xi \in \La_{\theta}$ and  $g  \in [e,\xi]$ in $\Ga$, we have \be\label{inccc} B_{\psi}(\xi, c e^{-\d_{\psi}(o, go)}) \subset O_{R_0}^{\theta}(o, g o) \cap \La_{\theta}.\ee 
\end{prop}
\begin{proof}
    Let $C_1, D > 0$ be the constants given by Proposition \ref{ss} and Theorem \ref{thm.triangle} respectively. Recall the constant $T_{\psi}$ in Proposition \ref{prop.thintriangle}; the image of any geodesic triangle in $(\Ga, \d_{\Ga})$ under the orbit map, is $T_{\psi}$-thin in the $\d_\psi$-metric. We now claim that \eqref{inccc} holds with $c := e^{- (2 T_{\psi} + C_1 + D)} $.
  Fix $\xi \in \La_{\theta}$ and an element $g\in [e, \xi]$.  Let $\eta \in B_{\psi}(\xi, c e^{-\d_{\psi}(o, go)} )$, that is,
    \be\label{ball} \psi(\cal G^{\theta}(\xi, \eta)) > \d_{\psi}(o, go) + 2 T_{\psi} +  C_1 + D.\ee 
  
    Let $\ga \in [\xi, \eta]$ be chosen so that $\d_{\psi}(o, \ga o) = \d_{\psi}(o, [\xi, \eta] o)$. By Proposition \ref{ss}, we have $$\d_{\psi}(o, \ga o) \ge \psi(\cal G^{\theta}(\xi, \eta)) -  C_1 .$$
    Hence by \eqref{ball},
    \be\label{ball2} \d_{\psi}(o, \ga o) >\d_{\psi}(o, go) + 2 T_{\psi} + D.\ee 
    Let $g' \in [\xi, \eta]$ be such that $\d_{\psi}(go, g' o) = \d_{\psi}(go, [\xi, \eta] o)$.
    By Theorem \ref{thm.triangle}, 
    we also have $$\d_{\psi}(o, \ga o) \le \d_{\psi}(o, g' o) \le \d_{\psi}(o, g o) + \d_{\psi}(g o, [\xi, \eta] o) + D.$$ Together with \eqref{ball2},
    this implies \be\label{ball3} \d_{\psi}(g o, [\xi, \eta] o) > 2 T_{\psi} .\ee 

    \begin{figure}[h]
\begin{tikzpicture}[ scale = 1, every node/.style={scale=1}]

    \filldraw[fill=gray!10, opacity=0.3] (-0.2, 1.2) circle(1.4);
    \filldraw[fill=gray!30, opacity=0.3] (-0.2, 1.2) circle(0.7);

    \draw (2.8284271247, 2.8284271247) arc(45:135:4);

    \draw[thick, color=red] (0, -1) .. controls (0.5, 2.5) and (1.5, 3) .. (2, 3.4641016151);
    \draw[red] (0.5, -0.2) node[right] {\small $\ell M_{\theta} A^+ o$};

    \draw[color=gray!50] (2, 3.4641016151) .. controls (1, 3) and (-1, 3) .. (-2, 3.4641016151);

    \draw[color=gray!50] (-2, 3.4641016151) .. controls (-1.5, 3) and (-0.5, 2.5) .. (0, -1);

    \filldraw (2, 3.4641016151) circle(1pt);
    \filldraw (2, 3.4641016151) node[above] {$\eta$};
    \filldraw (-2, 3.4641016151) circle(1pt);
    \filldraw (-2, 3.4641016151) node[above] {$\xi$};

    \filldraw (0, -1) circle(1pt);
    \draw (0, -1) node[below] {$o$};

    \draw[dashed] (-0.2, 1.2) -- (-1.6, 1.2);
    \draw[dashed] (-0.2, 1.2) -- (0.5, 1.2);
    \draw (0.15, 1.2) node[above] {\tiny $T_{\psi}$};

    \draw (-0.3, 1.2) node[above] {$g o$};
    
    \filldraw[fill=orange] (-0.2, 1.2) circle(1pt);
    \draw (-1.25, 1.2) node[above] {\tiny $2 T_{\psi}$};
    
    \filldraw[color=red] (0.2, 0.5) circle(1pt);
    \filldraw[color=red] (0.2, 1) circle(1pt);
    \filldraw[color=red] (0.5, 1.3) circle(1pt);
    \filldraw[color=red] (0.6, 1.7) circle(1pt);
    \filldraw[color=red] (0.7, 1.5) circle(1pt);
    \filldraw[color=red] (0.8, 2.2) circle(1pt);
    \filldraw[color=red] (0.9, 2) circle(1pt);
    \filldraw[color=red] (1.0, 2.7) circle(1pt);
    \filldraw[color=red] (1.3, 2.8) circle(1pt);
    \filldraw[color=red] (1.5, 2.7) circle(1pt);
    \filldraw[color=red] (1.7, 3) circle(1pt);

    \filldraw[color=orange] (-0.2, 0.8) circle(1pt);
    \filldraw[color=orange] (-0.5, 1.5) circle(1pt);
    
    \filldraw[color=orange] (-0.6, 1.4) circle(1pt);
    \filldraw[color=orange] (-0.7, 1.8) circle(1pt);
    \filldraw[color=orange] (-0.8, 2.5) circle(1pt);
    \filldraw[color=orange] (-0.9, 2.3) circle(1pt);
    \filldraw[color=orange] (-1.0, 2.6) circle(1pt);
    \filldraw[color=orange] (-1.3, 2.9) circle(1pt);
    \filldraw[color=orange] (-1.5, 2.7) circle(1pt);
    \filldraw[color=orange] (-1.7, 3) circle(1pt);

    \filldraw[color=blue] (1.8, 3.5) circle(1pt);
    \filldraw[color=blue] (1.6, 3.5) circle(1pt);
    \filldraw[color=blue] (1.3, 3.2) circle(1pt);
    \filldraw[color=blue] (1.2, 3.3) circle(1pt);
    \filldraw[color=blue] (1, 3.3) circle(1pt);
    \filldraw[color=blue] (0.7, 3.1) circle(1pt);
    \filldraw[color=blue] (0.4, 3.3) circle(1pt);
    \filldraw[color=blue] (0.2, 3) circle(1pt);
    \filldraw[color=blue] (-0.1, 3.3) circle(1pt);
    \filldraw[color=blue] (-0.3, 3.2) circle(1pt);
    \filldraw[color=blue] (-0.5, 3.3) circle(1pt);
    \filldraw[color=blue] (-0.8, 3) circle(1pt);
    \filldraw[color=blue] (-0.9, 3.1) circle(1pt);
    \filldraw[color=blue] (-1.1, 3.3) circle(1pt);
    \filldraw[color=blue] (-1.4, 3.3) circle(1pt);
    \filldraw[color=blue] (-1.7, 3.3) circle(1pt);

\end{tikzpicture}
\caption{$go$ is far  from $[\xi, \eta]o$ and hence close to $[e, \eta]o$; so $\eta$ lies in the shadow
$O^{\theta}_{T' + D_0}(o, go)$.} \label{fig.shadow2}
\end{figure}
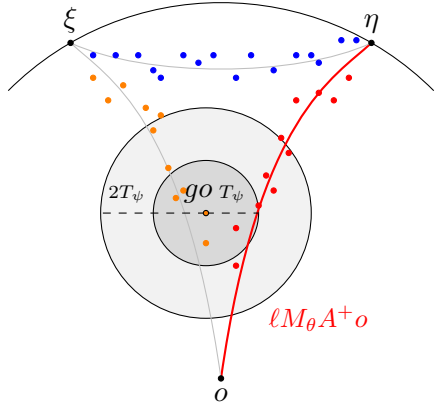

Since the triangle $[e, \xi]o \cup [\xi, \eta]o \cup [e, \eta ]o$ is $T_{\psi}$-thin in  $\d_{\psi}$-metric,
 $go$ is contained in the $T_{\psi}$-neighborhood of $[\xi, \eta]o \cup [e, \eta ]o$. Since $\d_{\psi}(g o, [\xi, \eta] o) > 2 T_{\psi}$ by \eqref{ball3}, we must have $\d_{\psi}(g o, [e, \eta] o) \le  T_{\psi} $ (see Figure \ref{fig.shadow2}). For the constant $T' := C_{\psi}(T_{\psi} + C_{\psi})$ where $C_{\psi}$ is as in  \eqref{eqn.qiriempsi}, we have $$d(g o, [e, \eta] o) \le T'.$$
  With the constant $D_0$ given in Theorem \ref{morse2}, there exists $\ell \in K$ so that $\ell P_{\theta} = \eta$ and $[e, \eta] o$ is contained in the $D_0$-neighborhood of $\ell M_{\theta}A^+ o$ in the Riemannian distance $d$. This implies that $$\eta \in O_{T' + D_0}^{\theta}(o, go)\cap \La_\theta.$$ This completes the proof with $R_0=T'+D_0$.
\end{proof}

\section{Shadows inside balls: the second inclusion in Theorem \ref{thm.ballshadowball}} \label{sec.shadowinball}
We continue the setup from section \ref{sec.ballinshadow}. Hence $\i(\theta)=\theta$ and 
$\psi\in \fa_\theta^*$ is a linear form such that
$\psi>0$ on $\L-\{0\}$ and $\psi=\psi \circ \i$.
In this section, we prove the second inclusion of Theorem \ref{thm.ballshadowball}, which can be stated as follows:

\begin{prop} \label{si}
    For any $r > 0$, there exists $c' = c'_r > 0$ such that for any $\xi \in \La_{\theta}$ and any $g \in [e,\xi]$ in $\Ga$, we have $$O_r^{\theta}(o, g o) \cap \La_{\theta} \subset B_{\psi}(\xi, c' e^{-\d_{\psi}(o, go)}).$$
\end{prop}

In addition to the coarse triangle inequality of $\d_\psi$ (Theorem \ref{thm.triangle}) and the uniform progression lemma (Lemma \ref{lp}), we will use the property that the shadows in $(\Ga, \d_\Ga)$
are comparable to shadows in $\La_\theta$ (Proposition \ref{prop.shadowgromovbooundary}) and that the half-spaces spanned by
shadows of balls in $(\Ga, \d_{\Ga})$ stay deeper than the balls from the viewpoints (see Figure \ref{fig.shadowopplight} and Lemma \ref{lem.hypgp}).

In the Gromov hyperbolic space $(\Ga,\d_\Ga)$,  for $R>0$ and $\ga_1, \ga_2\in \Ga$,
the shadow $O_R^{\Ga}(\ga_1, \ga_2)$ is defined as the set of
all $\xi\in \partial \Ga$ such that a geodesic ray $[\ga_1, \xi]$ 
intersects the $R$-ball centered at $\ga_2$:
$$ O_R^{\Ga}(\ga_1, \ga_2)=\{\xi\in \partial \Ga: \d_\Ga (\ga_2, [\ga_1, \xi])<R\} .   $$
Clearly, shadows are $\Ga$-equivariant in the sense that for any $\ga\in \G$, we have $\ga O_R^{\Ga}(\ga_1, \ga_2)=
O_R^{\Ga}(\ga \ga_1, \ga \ga_2)$.

The following proposition states that shadows in $\partial \Ga$ and shadows in $\La_{\theta}$ are compatible via the boundary map $f : \partial \Ga \to \La_{\theta}$:
recall that the orbit map $(\Ga, \d_{\Ga}) \to (\Ga o, d)$ is a $Q$-quasi-isometry for some $Q \ge 1$ (Theorem \ref{thm.anosovbasic}(3)) and let $R_0 := Q + D_0 + 1$ where $D_0$ is given in Theorem \ref{morse2}.
\begin{prop} \label{prop.shadowgromovbooundary}
 For any $R>R_0$, there exists $R_1, R_2 > 0$ such that for any $\ga_1, \ga_2 \in \Ga$, $$f(O_{R_1}^{\Ga}(\ga_1, \ga_2)) \subset O_R^{\theta} (\ga_1 o, \ga_2 o) \cap \La_{\theta} \subset f(O_{R_2}^{\Ga}(\ga_1, \ga_2)).$$
\end{prop}

In proving this proposition, we will also need to consider shadows whose viewpoints are on the boundary $\F_\theta$.
For $\eta\in \F_\theta$, $p\in X$, and $R>0$, the
$\theta$-shadow $O_R^{\theta}(\eta, p)$ is defined as follows: $$O_R^{\theta}(\eta, p) = \{g P_{\theta} \in \F_{\theta} : g \in G, \ gw_0P_{\theta} = \eta, \ d(p, go)<R  \}.$$
We will need the following proposition on continuity of shadows:
\begin{prop} [{Continuity of shadows on viewpoints,  \cite[Proposition 3.4]{KOW_ergodic}}] \label{prop.contshadow} 
    Let $p \in X$, $\eta \in \F_{\theta}$ and $r>0$.
    If a sequence $q_i \in X$ converges to $ \eta $ as $i \to \infty$ as in Definition \ref{fc}, then for any $0<\e<r$, we have \be \label{eqn.approxshadows}
    O_{r - \varepsilon}^{\theta}(\eta, p) \subset O_r^{\theta}(q_i, p) \subset O_{r + \varepsilon}^{\theta}(\eta, p) \quad \text{for all large } i\ge 1.
    \ee
\end{prop}

\subsection*{Proof of Proposition \ref{prop.shadowgromovbooundary}} Let $R>R_0$.
    By the $\Ga$-equivariance of $f$ as well as of shadows, we may assume $\ga_1 = e$ and write $\ga_2 = \ga$. 
 By applying Theorem \ref{morse2}(2), 
 we get that for any $\xi \in \partial \Ga$ and $k\in K$ with $kP_{\theta} =f(\xi)$,
 the image  $[e, \xi]o$ is contained in the  $D_0$-neighborhood of $kM_\theta (\exp \C) o\subset
 k M_{\theta} A^+ o$ in the symmetric space $(X, d)$.
 Since $R> R_0=Q+D_0+1 $, we can
 choose $R_1>0$ so that $ Q R_1+ Q+D_0 < R$. 
Now if  $\xi\in O^\Ga_{R_1}(e, \ga)$, and hence $[e, \xi]$
 intersects the ball $\{g\in \Ga: \d_\Ga( \ga , g)<R_1\}$,
 then $kM_{\theta}A^+ o$ intersects the $ Q R_1+ Q+D_0 $-neighborhood of $\ga o$, and hence the $R$-neighborhood of $\ga o$.  
Therefore  $f(\xi) \in O_{R}^{\theta}(o, \ga o)$. This shows the first inclusion.

To prove  the second inclusion, 
 suppose that the claim does not hold for some $R>R_0$. Then for each $i \ge 1$, there exists $\ga_i \in \Ga$ such that $$O_R^{\theta}(o, \ga_i o) \cap \La_{\theta} \not \subset f(O_i^{\Ga}(e, \ga_i));$$ in other words, there exists  $x_i \in \partial \Ga - O_i^{\Ga}(e, \ga_i)$ such that $f(x_i) \in O_R^{\theta}(o, \ga_i o)$. By the $\Ga$-equivariance of $f$, it follows that $$\ga_i^{-1}x_i \notin O_i^{\Ga}(\ga_i^{-1}, e) \quad \text{and} \quad f(\ga_i^{-1}x_i) \in O_R^{\theta}(\ga_i^{-1} o, o) \quad \text{for all } i \ge 1.$$ After passing to a subsequence, we may assume that $\ga_i^{-1} \to y \in \partial \Ga$ and $\ga_i^{-1}x_i \to x$ as $i \to \infty$.
By Theorem \ref{homeo}(4), we deduce $\ga_i^{-1} o \to f(y)$ as $i \to \infty$.
    Applying Proposition \ref{prop.contshadow} to $q_i = \ga_i^{-1} o$, $p = o$ and $\eta = f(y)$, we have for some $\varepsilon > 0$ that 
    $$ O_{R}^{\theta}(\ga_i^{-1}o, o) \subset O_{R + \varepsilon/2}^{\theta}(f(y), o) \quad \text{for all } i \ge 1.$$
    Since $f(\ga_i^{-1} x_i) \in O_{R}^{\theta}(\ga_i^{-1}o, o)$ for all $i \ge 1$ and $f(\ga_i^{-1} x_i) $ converges to $ f(x)$ as $i \to \infty$, we have
    $$f(x) \in O_{R + \varepsilon}^{\theta}(f(y), o).$$ This implies that $f(x)$ is in general position with $f(y)$, i.e.,
    $(f(x), f(y))\in \F_\theta^{(2)}$, and in particular $f(x) \neq f(y)$. On the other hand, since $\ga_i^{-1}x_i \notin O_i^{\Ga}(\ga_i^{-1}, e)$ for all $i \ge 1$, the sequence of geodesics $[\ga_i^{-1} x_i, \ga_i^{-1}]$ escapes any large ball centered at $e$. This implies that two sequences $\ga_i^{-1} x_i$ and $\ga_i^{-1}$ must have the same limit, and hence $x = y$ which is a contradiction. Therefore the claim follows.
\qed

\medskip

The  Gromov product in $(\Ga, \d_{\Ga})$ is defined as follows: for $\alpha, \beta, \ga \in \Ga$, $$(\alpha, \beta)_{\ga} = \frac{1}{2} \left( \d_{\Ga}(\alpha, \ga) + \d_{\Ga}(\beta, \ga) - \d_{\Ga}(\alpha, \beta)\right)$$ and for $x, y \in \partial \Ga$, $$(x, y)_{\ga} = \sup \liminf_{i, j \to \infty} (x_i, y_j)_{\ga}$$ where the supremum is taken over all sequences $\{x_i\}, \{y_j\}$ in $\Ga$ such that $\lim_{i \to \infty} x_i = x$ and $\lim_{j \to \infty} y_j = y$. The Gromov product for a pair of a point in $\Ga$ and a point in $\partial \Ga$ is defined similarly. The Gromov product $(x, y)_{\ga}$ is known to measure distance from $\ga$ and to a geodesic $[x, y]$ up to a uniform additive error (see \cite{Bridson1999metric} for basic properties of Gromov hyperbolic spaces).

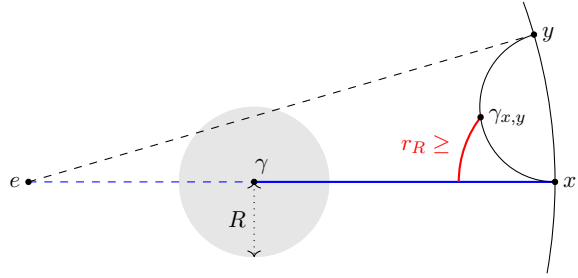
\begin{figure}[h]
\begin{tikzpicture}[scale=1, every node/.style={scale=0.8}]

        \draw[draw=none, fill=gray!20, opacity=0.5] (0, 0) circle(1);
        
        \draw[draw=none, fill=gray!20, opacity=0.5] (0, 0) circle(1);
  
        \draw (4, 0) arc(0:20:7);
        \draw (4, 0) arc(0:-10:7);

        \draw[dashed, blue] (-3, 0) -- (0, 0);
        \draw[blue, thick] (0, 0) -- (4, 0);

        \draw[dashed] (-3, 0) -- (3.72, 1.96);

	\filldraw (0, 0) circle(1pt);
	\draw (0.1, 0) node[above] {$\ga$};
        \draw[dotted, <->] (0, 0) -- (0, -1);
        \draw (0, -0.5) node[left] {$R$};

        \filldraw (3.72, 1.96) circle(1pt);
        \draw (3.72, 1.96) node[right] {$y$};
 	
	\filldraw (-3, 0) circle(1pt);
	\draw (-3, 0) node[left] {$e$};

        \draw (4, 0) arc(-90:-254:1);

        \filldraw (4, 0) circle(1pt);
        \draw (4, 0) node[right] {$x$};

        \draw[thick, red] (2.72, 0) arc(180:141:1.4212670404);
        \draw[red] (2.72, 0.5) node[left] {$r_R \ge$};
        
        \filldraw (3.01, 0.86) circle(1pt);
        \draw (3.01, 0.86) node[right] {$\ga_{x, y}$};

\end{tikzpicture}
\caption{Pictorial description of Lemma \ref{lem.hypgp}.} \label{fig.shadowopplight}
\end{figure}

The following lemma says that the half-space spanned by
the shadow is opposite to the light; more precisely, for any $x\in \partial \Ga$ and  $\ga\in [e,x]$, the half-space spanned by
all geodesics connecting $x$ and $O^\Ga_R(e,\ga)$ lies farther than $\ga$, viewed from $e$:

\begin{lemma} \label{lem.hypgp}
    Given $R > 0$, there exists
    $r = r_R > 0$  such that for any $x\in \partial \Ga$, $\ga\in [e,x]$, and
     $y \in O_R^{\Ga}(e, \ga)$,  we have  
     $$\d_{\Ga}(\ga_{x,y}, [\ga, x] ) \le r$$ where
    $\ga_{x,y}\in [x, y]$ denotes the nearest-point projection of $e$ to a geodesic $[x,y]$.
\end{lemma}

\begin{proof}
    Let $[e, x] = \{\ga_i\}_{i \ge 0}$.
    We fix $\ga := \ga_i$ and $y \in O_R^{\Ga}(e, \ga)$. 
In terms of the Gromov product, we have 
    $(e, y)_{\ga} < R + \delta / 2$ for some uniform $\delta > 0$ depending only on $\Ga$.
    On the other hand, the hyperbolicity of $\Ga$ also implies that we can take $\delta$ large enough so that 
    $$(e, y)_{\ga}
    \ge \min \{ (e, \ga_{x,y})_{\ga}, (\ga_{x,y}, y)_{\ga} \} - \delta/2$$ and that every geodesic triangle in $\Ga \cup \partial \Ga$ is $\delta$-thin. 
    Therefore $$\min \{ (e, \ga_{x,y})_{\ga}, (\ga_{x,y}, y)_{\ga} \}  < R+\delta .$$

 First consider the case when $(\ga_{x,y}, y)_{\ga} < R + \delta$. Then for some constant $\delta_1$ depending on $R + \delta$, there exists $\ga' \in [\ga_{x,y}, y]$ such that $\d_{\Ga}(\ga', \ga) < \delta_1$. Consider the  geodesic triangle with vertices  $x, \ga, \ga'$. Since this triangle is $\delta$-thin and $\ga_{x,y} \in [x, \ga']$, the $\delta$-neighborhood of $\ga_{x,y}$ intersects $[x, \ga] \cup [\ga, \ga']$. Hence it follows from $\d_{\Ga}(\ga, \ga') < \delta_1$ that the $(\delta + \delta_1)$-neighborhood of $\ga_{x,y}$ intersects the geodesic $[x, \ga]$. Namely, $$\d_{\Ga}(\ga_{x,y}, [\ga, x] ) \le \delta+\delta_1.$$

    Now consider the case that $(e, \ga_{x,y})_{\ga} < R + \delta$. Since $\ga_{x,y}$ is the nearest-point projection of $e$ to $[x, y]$, there exists a constant $\delta_2$ depending only on $\Ga$ such that the $\delta_2$-neighborhood of $\ga_{x,y}$ intersects both geodesic rays $[e, x]$ and $[e, y]$. In particular, there exists $\ga_{k}\in [e, x]$ such that $\d_{\Ga}(\ga_{x,y}, \ga_{k}) < \delta_2$. This implies $$(e, \ga_{k})_{\ga} \le (e, \ga_{x,y})_{\ga} + \d_{\Ga}(\ga_{x,y}, \ga_{k}) < R + \delta + \delta_2.$$
    Since both $\ga =\ga_i $ and $ \ga_{k}$ lie on the geodesic $[e, x]$, this implies that ${k} \ge i - (R + \delta + \delta_2)$. 
    Let $j$ be the unique integer such that $k + R + \delta + \delta_2 \le j \le k + R + \delta + \delta_2  + 1$.  Note that since $k \ge i - (R + \delta + \delta_2)$, we have
    $j \ge i$, and hence $\ga_j \in [\ga, x]$. Then 
    \begin{align*}
        \d_{\Ga}(\ga_{x,y}, [\ga, x] ) &\le  \d_{\Ga}(\ga_{x,y}, \ga_j) \\
        &\le \d_{\Ga}(\ga_{x, y}, \ga_k) + \d_{\Ga}(\ga_k, \ga_j)
         \\ &\le 
    \delta_2 + (j-k)\le R + \delta + 2 \delta_2 + 1 . \end{align*}
    Therefore it remains to set  $r = R + \delta + \delta_1 + 2 \delta_2 + 1$.
\end{proof}

Now we are ready to prove:
\subsection*{Proof of Proposition \ref{si}} Let $\xi \in \La_\theta=\partial\Ga$ and $g\in [e,\xi]$ in $\Ga$. Fix $r>0$, and
    let $\eta \in O_r^{\theta}(o, g o) \cap \La_{\theta}$ distinct from $\xi$. We will continue to use the convention of identifying $\La_\theta$ and $\partial \Ga$ in this proof.
    As in Lemma \ref{lem.hypgp}, we let $\ga_{\xi,\eta}$ be the nearest-point projection of $e$ to a bi-infinite geodesic $[\xi, \eta]$ in $(\Ga, \d_\Ga)$. By Proposition \ref{prop.shadowgromovbooundary}, there exists $R>0$, depending only on $r$, such that $\eta \in O_R^{\Ga}(e, g)$. 
    Write the geodesic ray $[e,\xi]$ as a sequence $\{g_k\}_{k \ge 0}$ with $g_0=e$.
  Since  $g\in [e,\xi]$ by the hypothesis,  we have  $g_i = g$ for some $i \ge 0$.
    Then for $r_R > 0$ given in Lemma \ref{lem.hypgp}, there exists $j \ge i$ such that
$$\d_{\Ga}(\ga_{\xi,\eta}, g_j) \le r_R.$$

    Let $n_1 \ge 0$ and $D \ge 0$ be given by Lemma \ref{lp} (uniform progression lemma) and Theorem  \ref{thm.triangle} (coarse triangle inequality) respectively. 
   We then have  $$\begin{aligned}
        \d_{\psi}(o, g_j o) & \ge \d_{\psi}(o, g_{i - n_1} o) + 1 \\
        & \ge \d_{\psi}(o, g o) - \d_{\psi}(g_{i  - n_1} o, g o) - D + 1
    \end{aligned} .$$ Since $g = g_i$, 
    we have  $\d_{\psi}(g_{i - n_1} o, go) \le Q_{\psi}(n_1 + 1)$ where $Q_{\psi}$ is the constant in Proposition \ref{prop.thintriangle}. Hence we deduce by setting $D' := Q_{\psi}(n_1 + 1) + D$ that $$\d_{\psi}(o, g_j o) \ge \d_{\psi}(o, g o ) - D'.$$

    On the other hand, applying the coarse triangle inequality (Theorem \ref{thm.triangle}) again, we have $$\d_{\psi}(o, g_j o) \le \d_{\psi}(o, \ga_{\xi,\eta} o) + \d_{\psi}(\ga_{\xi,\eta} o, g_j o) + D.$$ Since $\d_{\Ga}(\ga_{\xi,\eta}, g_j) \le r_R$,  we have $\d_{\psi}(\ga_{\xi,\eta} o, g_j o) \le Q_{\psi}(r_R + 1)$ by Proposition \ref{prop.thintriangle}, and hence $$\d_{\psi}(o, \ga_{\xi,\eta} o) \ge \d_{\psi}(o, go) - D' - Q_{\psi}(r_R + 1) - D.$$ Since we have $|\psi( \cal G^{\theta}(\xi, \eta)) - \d_{\psi}(o, \ga_{\xi,\eta} o) | < \|\psi\|C_1$
    with $C_1$ given by Lemma \ref{lem.gromovsym}, $$\psi( \cal G^{\theta}(\xi, \eta)) \ge \d_{\psi}(o, go) - D' - Q_{\psi}(r_R + 1) - D - \|\psi\|C_1 .$$ Setting $c' := e^{D' + Q_{\psi}(r_R + 1) + D + \|\psi\|C_1}$, we have $$d_{\psi}(\xi, \eta) \le c' e^{-\d_{\psi}(o, go)}.$$
    Hence $\eta \in B_\psi (\xi, c' e^{{-\d_{\psi}(o, go)}}) $ as desired.
\qed

\subsection*{Comparing Gromov products}

\begin{lem} \label{comG}

    Let $\psi \in \fa_{\theta}^*$ be such that $\psi > 0$ on $\L - \{0\}$. Then there exists $c >0$, depending only on $\psi$, such that for all $x, y\in \partial\Ga$,
    $$ Q_{\bar \psi}^{-1} \cdot (x, y)_e - c \le \psi(\cal G^{\theta}(f(x), f(y))) \le Q_{\bar \psi} \cdot (x, y)_e + c$$
    where $Q_{\bar\psi}$ is as given in Proposition \ref{Qp}.
\end{lem}

\begin{proof}
    Let $x \neq y \in \partial \Ga$ and set $\ga_{x, y} \in \Ga$ the nearest-point projection of $e$ to $[x, y]$ in $(\Ga, \d_{\Ga})$. By Lemma \ref{lem.gromovsym}, we have $$| \bar \psi( \cal G^{\theta}(f(x), f(y))) - \bar \psi(\mu_{\theta}(\ga_{x, y}))| < C_1$$
    where $C_1$ is given in Lemma \ref{lem.gromovsym}.
    As in the proof of Proposition \ref{prop.bilip}, $$| \bar \psi( \cal G^{\theta}(f(x), f(y))) - \psi( \cal G^{\theta}(f(x), f(y)))| < \| \psi \| C_1.$$
    Since $$Q_{\bar \psi}^{-1} \cdot \d_{\Ga}(e, \ga_{x, y}) - Q_{\bar \psi} \le \bar \psi(\mu_{\theta}(\ga_{x, y})) \le Q_{\bar \psi} \cdot \d_{\Ga}(e, \ga_{x, y}) + Q_{\bar \psi}  $$
    by    Proposition \ref{prop.thintriangle} with $Q_{\bar \psi} \ge 1$ therein, we now have that $$Q_{\bar \psi}^{-1} \cdot \d_{\Ga}(e, \ga_{x, y})- c' \le \psi( \cal G^{\theta}(f(x), f(y))) \le Q_{\bar \psi} \cdot \d_{\Ga}(e, \ga_{x, y}) +c'$$
    where $c' :=  Q_{\bar \psi} + C_1 (1 + \|\psi\|)$.
    Since $(\Ga, \d_{\Ga})$ is Gromov hyperbolic, we have that $|(x, y)_e - \d_{\Ga}(e, \ga_{x, y})|$ is uniformly bounded. Hence the claim follows.
\end{proof}

\section{Ahlfors regularity of Patterson-Sullivan measures} \label{sec.localsize}
As before, let $\Ga$ be a  $\theta$-Anosov subgroup of a connected semisimple real algebraic group $G$.
Recall from Theorem \ref{thm.uniquePS} that the space of $\Ga$-Patterson-Sullivan measures on $\La_\theta$ is parameterized by the set
$$  \mathscr T_\Ga=\{\psi\in \fa_\theta^*: \psi \text{ is tangent to } \psi_\Ga^\theta \}. $$ 
We continue to use the notation $\nu_\psi$ for the unique $(\Ga, \psi)$-Patterson-Sullivan measure on $\La_{\theta}$.
Recall that $d_\psi $ is the premetric on $\La_\theta$ defined by $d_\psi(\xi, \eta)=e^{-\psi(\cal G(\xi, \eta))}$ for all $\xi\ne \eta$ in $\La_\theta$ and
$B_{\psi}(\xi, r)=\{\eta\in \La_\theta: d_\psi(\xi, \eta)<r\}$.

The Ahlfors regularity is an important notion in fractal geometry:
\begin{Def}\label{aaa} A premetric space $(Z, d)$  is called Ahlfors $s$-regular
 if there exist a Borel measure $\nu$ on $Z$ and $C\ge 1$ so that for all $z\in Z$ and $r\in [0, \diam Z)$,
 $$C^{-1} r^s \le \nu (B(z,r)) \le C r^s$$
where $B(z, r)=\{w\in Z: d(z,w)<r\}$.
Such a measure $\nu$ is also called Ahlfors $s$-regular. 
\end{Def}

The goal of this section is to deduce the following from Theorem \ref{thm.ballshadowball}:

\begin{theorem} \label{thm.exactdim} For any symmetric  $\psi \in  \mathscr T_\Ga$, 
the measure
$\nu_\psi$ is Ahlfors one-regular on $(\La_\theta, d_\psi)$.
\end{theorem}

\begin{remark}\label{sss} When $\Ga$ is a convex cocompact subgroup of $G=\SO^\circ (n,1)$, $\mathscr T_\Ga$ is a singleton consisting of the critical exponent $\delta_\Ga$ (more precisely, the multiplication by $\delta_\Ga$ on $\R$), and the metric $d_{\delta_\Ga}$ is the $\delta_\Ga$-power of a
$K$-invariant Riemannian metric on $\S^{n-1}$. Hence Theorem \ref{thm.exactdim} is equivalent to Sullivan's theorem \cite[Theorem 7]{Sullivan1979density} that
the Patterson-Sullivan measure of a Riemannian ball of radius $r$ is comparable to $r^{\delta_\Ga}$.
\end{remark}

We use the higher rank version of Sullivan's shadow lemma. The following is a special case of \cite[Lemma 7.2]{KOW_indicators}:
    \begin{lemma}[Shadow lemma] \label{lem.shadow} Let $\Ga < G$ be a non-elementary $\theta$-Anosov
     subgroup. For all large enough $R > 0$, there exists $c_0 = c_0 (\psi, R) \ge 1$ such that for all $\ga \in \Ga$, 
    $$
    c_0^{-1} e^{-\psi(\mu_{\theta}(\ga))} \le  \nu_\psi (O_R^{\theta}(o, \ga o)) \le c_0 e^{-\psi(\mu_{\theta}(\ga))}.
    $$
\end{lemma}

\subsection*{Proof of Theorem \ref{thm.exactdim}} 
By Lemma \ref{lem.cansymm} and Remark \ref{nons}, it suffices to consider the case of $\theta = \i(\theta)$.
Let $c$ and $R_0$ be the constants as in Theorem \ref{thm.ballinshadow}.
    Fix $\xi \in \La_{\theta}$ and $0<r <\op{diam}(\La_\theta, d_\psi)$. Write the geodesic ray
    $[e, \xi] $ as  $\{\ga_k\}_{k \ge 0}$ in $(\Ga, \d_\Ga)$.  Setting \be\label{ir}  i = i_r := \max \{ k : r \le c e^{- \d_{\psi}(o, \ga_k o)} \},\ee  Theorem \ref{thm.ballinshadow} implies that
    for any $R > R_0$, $$B_{\psi}(\xi, r) \subset B_{\psi}(\xi, c e^{-\d_{\psi}(o, \ga_{i} o)}) \subset O_R^{\theta}(o, \ga_{i} o).$$  By Lemma \ref{lem.shadow}, we get
    \be \label{eqn.upperexdim}
    \nu_\psi(B_{\psi}(\xi, r)) \le c_0 e^{- \d_{\psi}(o, \ga_{i} o)}.
    \ee
    By the coarse triangle inequality of $\d_\psi$ (Theorem \ref{thm.triangle}), we have 
    $$\d_{\psi}(o, \ga_{i + 1} o) \le \d_{\psi}(o, \ga_{i} o) + \d_{\psi}(\ga_i o, \ga_{i + 1} o) + D$$ where $D$ is as in loc. cit. Since $\d_{\psi}(\ga_i o, \ga_{i + 1} o) \le 2 Q_{\psi}$ with $Q_{\psi}$ in Proposition \ref{prop.thintriangle}, we have $$\d_{\psi}(o, \ga_{i + 1} o) \le \d_{\psi}(o, \ga_{i} o) + D'$$ where $D' = D + 2 Q_{\psi}$. This implies $$c e^{- D'} e^{-\d_{\psi}(o, \ga_{i} o)} \le c e^{-\d_{\psi}(o, \ga_{i + 1} o)} < r$$ where the last inequality follows from the definition of $i = i_r$ in \eqref{ir}. Hence we deduce from \eqref{eqn.upperexdim} that $$\nu_\psi(B_{\psi}(\xi, r)) \le ( c_0 e^{D'}/c) \cdot r.$$

 Now let $c' = c'_R > 0$ be given by Theorem \ref{thm.shadowinball} and set \be\label{jr} j = j_r := \min\{k : c' e^{-\d_{\psi}(o, \ga_k o)} \le r\}.\ee  By Theorem \ref{thm.shadowinball}, we have $$O_R^{\theta}(o, \ga_{j} o) \cap \La_{\theta} \subset B_{\psi}(\xi, c' e^{- \d_{\psi}(o, \ga_{j} o)}) \subset B_{\psi}(\xi, r),$$ and hence applying Lemma \ref{lem.shadow} yields $$c_0^{-1} e^{-  \d_{\psi}(o, \ga_{j} o)} \le \nu_\psi(B_{\psi}(\xi, r)).$$
    By the minimality of $j = j_r$ as defined in \eqref{jr}  and the coarse triangle inequality of $\d_\psi$
    (Theorem \ref{thm.triangle}), we have
    $$r < c' e^{-\d_{\psi}(o, \ga_{j - 1} o)} \le c' e^{D} e^{-\d_{\psi}(o, \ga_{j} o) + \d_{\psi}(\ga_{j-1} o, \ga_{j} o)}.$$ Recalling that $\d_{\psi}(\ga_{j-1} o, \ga_j o) \le 2 Q_{\psi}$ and $D' = D + 2 Q_{\psi}$, we have $$r < c' e^{D'} e^{-\d_{\psi}(o, \ga_{j} o)}$$ and hence $$ (c_0 c' e^{D'})^{-1} \cdot r \le \nu_\psi(B_{\psi}(\xi, r)).$$ Therefore the theorem is proved with $c_1= \max (c_0 e^{D'} c^{-1}, c_0 c' e^{D'})$. 
\qed

\section{Hausdorff measures on limit sets} \label{sec.hausmeas}
Let $\Ga < G$ be a $\theta$-Anosov subgroup where $G$ is a connected semisimple real algebraic group.
For a linear form $\psi \in \fat$ which is positive on $\L -\{0\}$,
consider the associated conformal premetric $d_\psi$ on $\La_\theta$. For $s>0$, we denote by  $\cal H_{\psi}^s$
the associated Hausdorff measure of dimension $s$, that is, for any subset $B \subset \La_{\theta}$, let 
\be\label{hhhjun19} \cal H_{\psi}^{s}(B) :=  \lim_{\varepsilon \to 0} \inf \left\{ \sum_{i}(\diam_{\psi} U_i)^s : B \subset \bigcup_{i} U_i, \ \sup_i\diam_{\psi} U_i \le \varepsilon\right\}\ee where $\diam_{\psi} U =\sup_{\xi, \eta\in U} d_\psi(\xi, \eta)$.
This is an outer measure which induces a Borel measure on $\La_\theta$ (see \cite{Falconer_fractatl}, \cite[Appendix A]{DK_patterson}). 
For $s=1$, we simply write $\cal H_\psi$ for $\cal H_\psi^1$. 
 Recall that $\T_{\Ga}$ is the space of all linear forms tangent to the growth indicator $\psi_\Ga^\theta$.
In this section, we first deduce the following two theorems from Theorem \ref{thm.exactdim}. 
Together with Theorem \ref{thm.Aregular}, they imply Theorems \ref{main1} and \ref{main0}. We also prove Theorem \ref{main2}.

\begin{theorem} \label{thm.hausmeas}
For any symmetric $\psi\in \mathscr T_\Ga$, the associated Patterson-Sullivan measure $\nu_\psi$
coincides with the  one-dimensional Hausdorff measure $\cal H_{\psi}$, up to a constant multiple.
In other words, $\cal H_\psi$ is the unique  $(\Ga, \psi)$-conformal measure on $\La_\theta$ (up to a constant multiple).
\end{theorem}

We also show that the symmetric hypothesis is necessary: 
\begin{theorem}
     \label{thm.nonsymmhmeas}
    If $\psi\in \mathscr T_\Ga$ is not symmetric and $\Ga$ is Zariski dense, then $\nu_\psi$ is not comparable to $\cal H_{\psi}^s$ for any $s > 0$. 
\end{theorem}

\begin{Rmk}\label{hhh} If $\psi\in \fa_\theta^*$ is positive on $\L -\{0\}$, then $\delta_\psi \psi\in \mathscr T_\Ga$. Since $\cal H_{ \delta_\psi \psi}=\cal H_{ \psi}^{\delta_\psi}$, Theorem \ref{thm.hausmeas} says that if $\psi$ is symmetric in addition,
\be\label{hhh2} \cal H_\psi^{\delta_{\psi}} =\nu_{\delta_{\psi} \psi}\quad
\text{up to a constant multiple.}\ee   
    \end{Rmk}

\begin{Rmk} \label{rmk.DK} For a  special class of symmetric $\psi$ whose gradient lies in the interior of $\fa_\theta^+$,  
Dey-Kapovich \cite[Corollary 4.8]{DK_patterson} showed that $(\Ga o, \d_\psi)$ is a Gromov hyperbolic space and
they proved Theorem \ref{thm.hausmeas} relying upon the work of Coornaert \cite{Coornaert_PS} 
which gives the positivity and finiteness of $\cal H_\psi$ for the Gromov hyperbolic space. In our generality,
$(\Ga o, \d_\psi)$ is not even a metric space, and hence their approach cannot be extended.
\end{Rmk}

The main work is to establish the positivity and the finiteness of $\cal H_\psi$ and the key ingredient is the Ahlfors regular property of $\nu_\psi$ obtained in Theorem \ref{thm.exactdim}. For example, positivity of $\cal H_\psi$ is a standard consequence of the Ahlfors regularity of $(\La_\theta, d_\psi)$. However, we cannot conclude finiteness of ${\mathcal H}_\psi$ directly from Ahlfors regularity due to the lack of the triangle inequality.

\begin{prop} [Positivity] \label{prop.poshaus}
 For any symmetric $\psi\in \mathscr T_\Ga$,  we have $$\cal H_{\psi}(\La_{\theta}) > 0.$$
\end{prop}

\begin{proof}
    Fix $\varepsilon > 0$ and a countable cover $\{U_i\}_{i \in \N}$ such that $\diam_{\psi} U_i \le \varepsilon$ for all $i \in \N$. For each $i \in \N$, we choose $\xi_i \in U_i$. By Theorem \ref{thm.exactdim}, we have $$\sum_{i \in \N} \diam_{\psi} U_i \ge C \sum_{i \in \N} \nu_{\psi}(B_{\psi}(\xi_i, \diam_{\psi} U_i))$$
    where the  constant $C>0$ depends only on $\psi$.
    Since $\La_{\theta} \subset \bigcup_{i \in \N} U_i \subset \bigcup_{i \in \N} B_{\psi}(\xi_i, \diam_{\psi} U_i)$, it follows that $\sum_{i \in \N} \diam_{\psi} U_i \ge C \cdot \nu_{\psi} (\La_{\theta}) = C.$ Since $\{U_i\}_{i \in \N}$ is an
    arbitrary countable cover, it follows that
    $\cal H_{\psi, \varepsilon}(\La_{\theta}) \ge C$.
Since $\e>0$ is arbitrary and $C$ is independent of $\e>0$, we have $\cal H_\psi (\La_\theta)>0$. 
\end{proof}

\begin{prop}[Finiteness] \label{prop.finhaus}
For any symmetric $\psi\in \mathscr T_\Ga$,
    we have $$\cal H_{\psi}(\La_{\theta}) < \infty.$$
\end{prop}

\begin{proof}
Let $N=N(\psi)$ and $N_0=N_0(\psi)$ be the constants  given in Proposition \ref{prop.lotriangle} and Lemma \ref{lem.lovitali} respectively.   Fix $\varepsilon > 0$. Since $\La_{\theta}$ is compact, we have a finite cover $\La_{\theta}$ by $ \bigcup_{i = 1}^n B_{\psi}(\xi_i, \tfrac{\varepsilon}{2NN_0})$  for some finite set $\xi_1, \cdots, \xi_n \in \La_{\theta}$.
Applying the Vitali covering type lemma (Lemma \ref{lem.lovitali}), there exists a disjoint subcollection $B_{\psi}(\xi_{i_1}, \tfrac{\varepsilon}{2NN_0}), \cdots, B_{\psi}(\xi_{i_k}, \tfrac{\varepsilon}{2NN_0})$ such that $$\La_{\theta} \subset \bigcup_{j = 1}^k B_{\psi}(\xi_{i_j}, \tfrac{\varepsilon}{2N}).$$ Since $\diam_{\psi} B_{\psi}(\xi_{i_j}, \tfrac{\varepsilon}{2N}) \le \varepsilon$ for each $1 \le j \le k$ by Proposition \ref{prop.lotriangle}(2), we have  $$\cal H_{\psi,\e}(\La_{\theta}) \le \sum_{j = 1}^k \diam_{\psi} B_{\psi}(\xi_{i_j}, \tfrac{\varepsilon}{2N}) \le  k \cdot \varepsilon.$$
    Applying Theorem \ref{thm.exactdim}, we obtain that for some constant $C>0$ depending only on $\psi$,
    $$
         k \cdot \varepsilon \le C \sum_{j = 1}^k \nu_{\psi}(B_{\psi}(\xi_{i_j}, \tfrac{\varepsilon}{2NN_0}))
         = C \cdot \nu_{\psi} \left( \bigcup_{j = 1}^k B_{\psi}(\xi_{i_j}, \tfrac{\varepsilon}{2NN_0}) \right) \le C \cdot \nu_{\psi}(\La_{\theta}) = C
    $$ where the equality follows from the disjointness. This implies $\cal H_{\psi, \varepsilon}(\La_{\theta}) \le C$. Since $\varepsilon$ is arbitrary, we have $\cal H_{\psi}(\La_{\theta}) \le C$.
\end{proof}

Hence $\cal H_\psi$ is a {\it non-trivial} measure on $\La_\theta$. It is also $(\Ga, \psi)$-conformal:

\begin{lemma}[Conformality] \label{lem.confhaus}
 For any symmetric $\psi\in \mathscr T_\Ga$,
 we have $$\frac{d \ga_* \cal H_{\psi}}{ d \cal H_{\psi}}(\xi) = e^{ \psi(\beta_{\xi}^{\theta}(e, \ga))}$$
  for all $\ga \in \Ga$ and $\xi \in \La_{\theta}$.
\end{lemma}

\begin{proof} 
Since $d_{\psi}$ is invariant under the $\Ga$-equivariant homeomorphism $p : \La_{\theta\cup\i(\theta)} \to \La_\theta$ by the definition of $d_{\psi}$ (Remark \ref{nons}), the measure $(\cal H_\psi,\La_\theta)$ is the push-forward of the Hausdorff measure $(\cal H_\psi, \La_{\theta\cup \i (\theta)})$ via $p$.
Therefore it suffices to prove this lemma 
assuming that $\theta = \i(\theta)$. We simply write $\beta^{\theta} = \beta$ in this proof to ease the notations.
    
    Fix $\ga \in \Ga$ and $\xi \in \La_{\theta}$. Let $U \subset \La_{\theta}$ be a small open neighborhood of $\xi$. To estimate $\ga_* \cal H_{\psi}(U)$ in terms of $\cal H_{\psi}(U)$, we fix $\varepsilon > 0$ and take any cover $\{U_i \}_{i \in \N}$ of $U$ such that $\diam_{\psi} U_i \le \varepsilon$ and that  $U \cap U_i \neq \emptyset$ for all $i \in \N$.

 For simplicity, we write $s_{\xi, R} (\ga) := \sup_{\eta \in B_{\psi}(\xi, R)} e^{\psi(\beta_{\eta}(e, \ga))}$. 
    By Lemma \ref{lem.gromovconformal} and  Proposition \ref{prop.lotriangle} with $N=N(\psi) > 0$ therein,
    we have that for each $i \ge 1$, $$
        \diam_{\psi} \ga^{-1} U_i \le \sup_{\eta \in U_i} e^{ \psi(\beta_{\eta}(e, \ga))} \diam_{\psi} U_i \le  s_{\xi, R_\e}( \ga)  \diam_{\psi} U_i
    $$ where $R_{\varepsilon} =R_\e(U)= N\cdot (\diam_{\psi} U + \varepsilon)$.     We then have for $\tilde{\varepsilon} := s_{\xi, R_{\varepsilon}}(\ga) \varepsilon$, $$\begin{aligned}
        \cal H_{\psi,\tilde{\e}}(\ga^{-1} U) & \le \sum_{i \in \N} \diam_{\psi} \ga^{-1} U_i \le s_{\xi, R_\e}( \ga) \sum_{i \in \N} \diam_{\psi} U_i.
    \end{aligned}$$ Since $\{U_i\}_{i \in \N}$ is an arbitrary countable open cover of $U$, the above inequality implies $$\cal H_{\psi,\tilde{\e}}(\ga^{-1} U ) \le s_{\xi, R_\e} (\ga) \cal H_{\psi,\e }(U).$$ Taking $\varepsilon \to 0$, we have $\tilde{\varepsilon} = s_{\xi, R_\e}( \ga) \e \to 0$ and $R_{\varepsilon} \to R_U := N \cdot \diam_{\psi} U$. Therefore \be \label{eqn.conformalhauss}
    \cal H_{\psi}(\ga^{-1}U) \le s_{\xi, R_U}( \ga)  \cal H_{\psi}(U).
    \ee

    Applying \eqref{eqn.conformalhauss} after replacing $U$ with $\ga^{-1} U$, and $\ga$ by $\ga^{-1}$, we have \be \label{eqn.jan13}
        \cal H_{\psi}(U)  = \cal H_{\psi}(\ga(\ga^{-1}U)) \le 
        s_{\ga^{-1}\xi, R_{\ga^{-1}U}} (\ga^{-1})
        \cal H_{\psi}(\ga^{-1}U).
    \ee 
    If we set $c = \sup_{\zeta \in \La_{\theta}} e^{\psi(\beta_{\zeta}(e, \ga^{-1} ))}$, then
    for any $\eta \in B_{\psi}(\ga^{-1} \xi, R_{\ga^{-1}U})$, it follows from Lemma \ref{lem.gromovconformal} that $$d_{\psi}(\xi, \ga \eta) \le c d_{\psi} (\ga^{-1}\xi, \eta) \le c R_{\ga^{-1} U} .$$ 
    This implies $$
    \begin{aligned}
        s_{\ga^{-1} \xi, R_{\ga^{-1} U}}(\ga^{-1}) & = \sup_{\eta \in B_{\psi}(\ga^{-1} \xi, R_{\ga^{-1} U})} e^{ \psi(\beta_{\eta}(e, \ga^{-1}))} \\
        & \le \sup_{\eta \in B_{\psi}(\xi, c R_{\ga^{-1} U})} e^{\psi(\beta_{\ga^{-1} \eta}(e, \ga^{-1}))} = \sup_{\eta \in B_{\psi}(\xi, c R_{\ga^{-1} U})} e^{\psi(\beta_{\eta}(\ga, e))}
        \end{aligned}$$
    Hence we obtain from \eqref{eqn.jan13} that $$
        \cal H_{\psi}(U) \le \sup_{\eta \in B_{\psi}(\xi, c R_{\ga^{-1} U})} e^{- \psi(\beta_{\eta}(e, \ga))} \cal H_{\psi}(\ga^{-1} U).
    $$ Together with \eqref{eqn.conformalhauss}, we deduce $$\inf_{\eta \in B_{\psi}(\xi, c R_{\ga^{-1} U})} e^{\psi(\beta_{\eta}(e, \ga))} \le \frac{ \ga_* \cal H_{\psi}(U)}{\cal H_{\psi}(U)} \le \sup_{\eta \in B_{\psi}(\xi, R_{U})} e^{  \psi(\beta_{\eta}(e, \ga ))}.$$ Now shrinking $U \to \xi$, we have $R_U, R_{\ga^{-1} U} \to 0$ and hence the both sides in the above inequality converge to $e^{ \psi(\beta_{\xi}(e, \ga))}$, by the continuity of the Busemann map $\beta_{\eta}(e, \ga )$ on the $\eta$-variable. Therefore $$\frac{d \ga_* \cal H_{\psi}}{d \cal H_{\psi}}(\xi) = e^{\psi(\beta_{\xi}(e, \ga ))}$$ as desired.
\end{proof}

\subsection*{Proof of Theorem \ref{thm.hausmeas}} By Propositions \ref{prop.poshaus} and \ref{prop.finhaus}, we have $\cal H_{\psi}(\La_{\theta}) \in (0, \infty)$. Moreover, it follows from Lemma \ref{lem.confhaus} that $\frac{1}{\cal H_{\psi}(\La_{\theta})}\cal H_{\psi}$ is a $(\Ga, \psi)$-Patterson-Sullivan measure. Since there exists a unique $(\Ga, \psi)$-Patterson-Sullivan measure on $\La_{\theta}$ (Theorem \ref{thm.uniquePS}), this completes the proof.
\qed

\subsection*{Proof of Theorem \ref{thm.nonsymmhmeas}}
 By Lemma \ref{lem.cansymm}, we may assume $\theta = \i(\theta)$.
    Since $\psi\ne \psi \circ \i$, two linear forms
    $\psi$ and $\bar \psi$ are not proportional. Since $d_{\psi}$ and $d_{\bar \psi}$ are bi-Lipschitz by Proposition \ref{prop.bilip}, $\cal H_{\psi}^s$ is in the same measure class as $\cal H_{\bar \psi}^s$ for all $s > 0$. Hence it follows from Theorem \ref{thm.hausmeas} (see also Remark \ref{hhh}) that $\cal H_{\psi}^s(\La_{\theta}) = 0$ or $\infty$ if $s \neq \delta_{\bar \psi}$. Now it suffices to show that $\nu_\psi$ is not comparable to $\cal H_{\psi}^{\delta_{\bar \psi}}$.
Since $\psi$ and $\bar \psi$ are not proportional, $\psi$ and $\delta_{\bar \psi} \bar \psi$ are two different forms tangent to $\psi_{\Ga}^{\theta}$.
    By Theorem \ref{thm.uniquePS}, it follows that $\nu_\psi $ is mutually singular to $\nu_{\delta_{\bar \psi} \bar \psi}$.
    Since the latter is proportional to $\cal H_{\bar \psi}^{\delta_{\bar \psi}}$ by Theorem \ref{thm.hausmeas}, $\nu_\psi$
   is singular to $\cal H_{\bar \psi}^{\delta_{\bar \psi}}$ and hence singular to  
   $\cal H_{\psi}^{\delta_{\bar \psi}}$ as well. This finishes the proof.
\qed

\begin{remark} \label{rmk.jan13}
    In fact, without Zariski dense hypothesis, it was shown in \cite[Theorem A]{sambarino2022report} that for $\psi_1, \psi_2 \in \T_{\Ga}$, $\nu_{\psi_1}$ and $\nu_{\psi_2}$ are mutually singular unless $\psi_1 = \psi_2$ on $\L_{\theta}$. Hence
    Theorem \ref{thm.nonsymmhmeas} holds provided that $\psi$ and $\psi \circ \i$ are not identical on  $\L_{\theta}$.
\end{remark}

\subsection*{Critical exponents and Hausdorff dimensions} 
The Hausdorff dimension of $\La_{\theta}$ with respect to $d_{\psi}$ is defined as $$\dim_{\psi} \La_\theta := \inf \{ s > 0 : \cal H_{\psi}^s(\La_\theta) = 0 \} = \sup \{s > 0 : \cal H_{\psi}^s(\La_\theta) = \infty\}.$$
As a corollary of Theorem \ref{thm.hausmeas}, we obtain the following (Theorem \ref{main2}):
\begin{corollary} \label{cor.crithdim}
    For any $\psi \in \fa_{\theta}^*$ positive on $\L -\{0\}$, we have $$\delta_{\bar \psi} = \dim_{\psi} \La_{\theta}$$
    where $ \bar \psi = \frac{\psi + \psi \circ \i}{2}.$
\end{corollary}

\begin{proof}
By Proposition \ref{prop.bilip}, we have $\dim_{\psi} \La_{\theta} = \dim_{\bar \psi} \La_{\theta}$. Applying  Theorem \ref{thm.hausmeas} to $\delta_{\bar\psi} \bar \psi$
(see Remark \ref{hhh}), we have $\cal H_{\bar \psi}^{\delta_{\bar \psi}}(\La_{\theta}) \in (0, \infty)$, which implies $\dim_{\bar \psi} \La_{\theta} = \delta_{\bar \psi}$. This shows the claim.
\end{proof}

For $\psi$ non-symmetric,  $\dim_{\psi} \La_{\theta}$ is not in general equal to $\delta_\psi$:
\begin{prop} \label{prop.mincrit}
  For $\psi \in \fa_{\theta}^*$ positive on $\L - \{0\}$, we have $$\delta_{\bar \psi} \le \delta_{\psi}.$$ 
 If $\Ga$ is Zariski dense, then the equality holds if and only if $\psi = \psi \circ \i$.
    \end{prop}
\begin{proof}

As before, we may assume $\theta = \i(\theta)$. Suppose that $\psi \neq \psi \circ \i$.
Note that $\delta_{\psi}=\delta_{\psi \circ \i}$ and hence 
both $\delta_{\psi}\psi$ and $ \delta_{\psi} (\psi\circ \i)$ are tangent to the $\theta$-growth indicator $\psi_{\Ga}^{\theta}$ (\cite[Theorem 2.5]{KMO_tent}, \cite[Lemma 4.5]{KOW_indicators}). We then have
$\psi_\Ga^\theta \le \delta_{\psi}\psi$ and $\psi_{\Ga}^{\theta} \le\delta_{\psi} (\psi\circ \i) $.
Hence $\psi_\Ga^\theta \le \delta_{\psi} \bar\psi$. Since  $\delta_{\bar\psi}\bar\psi$ is tangent to $\psi_\Ga^\theta$, it follows that $\delta_{\bar \psi}\le \delta_{\psi} $.

Now suppose that $\Ga$ is Zariski dense and that $\psi\ne \psi\circ \i$. By Theorem \ref{strict},
there exists a unique unit vector $u=u_{\delta_\psi \psi} \in \inte \L_\theta$ such that
$\psi_\Ga^{\theta}(u)= \delta_\psi \psi(u)$. Since $\psi_\Ga^\theta$ is $\i$-invariant, it implies that $\psi_\Ga^{\theta}(\i(u))= \delta_\psi (\psi \circ \i) (\i (u))$. On the other hand,
$u \neq \i(u)$ by Theorem \ref{strict}. Hence the inequality $\psi_{\Ga}^{\theta} \le \delta_{\psi} \psi$ and $\psi_{\Ga}^{\theta} \le \delta_{\psi} ( \psi \circ \i)$ cannot become
equalities  simultaneously at the same vector. This implies that $\psi_{\Ga}^{\theta} < \delta_{\psi} \bar{\psi}$ and hence $\delta_{\bar{\psi}} < \delta_{\psi}$.
\end{proof}

By Corollary  \ref{cor.crithdim}  and Proposition \ref{prop.mincrit}, we obtain:
\begin{cor} Let $\Ga$ be Zariski dense in $G$.
    For any $\psi\in \mathscr T_\Ga$,
   we have $\dim_\psi \La_\theta\le 1$ and the equality holds if and only if
     $\psi$ is symmetric.
\end{cor}

We now prove the Ahlfors regularity of $(\La_\theta, d_\psi)$ for general $\psi \in \T_{\Ga}$:
\begin{theorem} \label{thm.Aregular}
    Let $\Ga$ be a $\theta$-Anosov subgroup. For any
    $\psi\in \mathscr T_\Ga$, the premetric space $(\La_\theta, d_\psi)$ is Ahlfors $s$-regular for some $0 < s \le 1$.
   Moreover if $\Ga$ is Zariski dense, we have $s = 1$ if and only if $\psi$ is symmetric.
\end{theorem}
\begin{proof}
Let $\psi \in \T_{\Ga}$. By Proposition \ref{prop.bilip}, the identity map $(\La_{\theta}, d_{\psi}) \to (\La_{\theta}, d_{\bar \psi})$ is bi-Lipschitz. Noting that $\delta_{\bar \psi} \bar \psi \in \T_{\Ga}$ by Lemma \ref{lem.tangentcritexp}, we denote by $\nu := \nu_{\delta_{\bar \psi} \bar \psi}$ the $(\Ga, \delta_{\bar \psi}\bar \psi)$-Patterson-Sullivan measure on $\La_{\theta}$.
By Theorem \ref{thm.exactdim}, for any $\xi \in \La_{\theta}$ and $r \in [0, \diam_{\bar \psi} \La_{\theta})$, we have 
\be \label{eqn.jan222}
 c^{-1} r\le \cal \nu(B_{\delta_{\bar \psi} \bar \psi}(\xi, r)) \le c  r.
\ee
for some constant $c\ge 1$ depending only on $\psi$. 
Since  $B_{\delta_{\bar \psi} \bar \psi}(\xi, r^{\delta_{\bar \psi}}) = B_{\bar \psi}(\xi, r)$ and the identity map $(\La_{\theta}, d_{\psi}) \to (\La_{\theta}, d_{\bar \psi})$ is bi-Lipschitz  by Proposition \ref{prop.bilip}, \eqref{eqn.jan222} implies that for some $C \ge 1$ depending only on $\psi$, we have 
$$C^{-1} r^{\delta_{\bar \psi}} \le \nu(B_{\psi}(\xi, r)) \le C r^{\delta_{\bar \psi}}.$$
Recall  $\delta_{\psi} = 1$ for $\psi \in \T_{\Ga}$ (Lemma \ref{lem.tangentcritexp}). Hence
$\delta_{\bar\psi}<1$ for $\psi$ non-symmetric and $\Ga$ Zariski dense by Proposition \ref{prop.mincrit}. This finishes the proof.
\end{proof}

\subsection*{Analyticity of Hausdorff dimensions}
For a hyperbolic group $\Sigma$, a representation $\sigma : \Sigma \to G$ is called $\theta$-Anosov if $\sigma$ has a finite kernel and its image $\sigma(\Sigma)$ is a $\theta$-Anosov subgroup of $G$.
For a given $\psi \in \fa_{\theta}^*$ which is non-negative on $\fa_{\theta}^+$, the $\psi$-critical exponents $\delta_{\psi}(\sigma(\Sigma))$ vary analytically on analytic families of $\theta$-Anosov representations $\sigma$ in the variety $\op{Hom}(\Sigma, G)$ 
 by Bridgeman-Canary-Labourie-Sambarino \cite[Proposition 8.1]{BCLS_gafa} 
 (see also \cite[Section 4.4]{CS_local}).  Hence the following is an immediate consequence of Corollary \ref{cor.crithdim}:

\begin{corollary}\label{analytic}
    Let $\Sigma$ be a non-elementary\footnote{A hyperbolic group is non-elementary if its Gromov boundary has at least three points.} hyperbolic group and $\psi \in \fa_{\theta}^*$ be non-negative on $\fa_{\theta}^+$. Let $\cal D \subset \op{Hom}(\Sigma, G)$ be an analytic family of $\theta$-Anosov representations. Then $$\sigma \mapsto \dim_{\psi} \La_{\theta}(\sigma  (\Sigma))$$ is analytic on $ \cal D$.
\end{corollary}

\subsection*{$(p,q)$-Hausdorff dimensions} Let $\Sigma$ be a non-elementary convex cocompact subgroup of $\so(n, 1)= \Isom^+(\H^n)$, $n\ge 2$. Let $\op{CC}(\Sigma) $ denote the space
$$ \{ \sigma : \Sigma \to \so(n, 1) : \text{convex cocompact, faithful representation}\}/\sim$$
where the equivalence relation is given by conjugations.
As in the introduction, for $\sigma \in \op{CC}(\Sigma)$, we denote by $\La_{\sigma} \subset \S^{n-1} \times \S^{n-1}$ the limit set of the self-joining subgroup $\Sigma_{\sigma} := (\id \times \sigma)(\Sigma) < \so(n, 1) \times \so(n, 1)$, which is well-defined up to translations.
The Hausdorff dimension of $\La_{\sigma}$  with respect to a Riemannian metric on $\S^{n-1}\times \S^{n-1}$ is 
equal to  $\max \left(\dim \La_\Sigma, \dim \La_{\sigma(\Sigma)} \right)$, 
where $\La_{\Sigma} \subset \S^{n-1}$ and $\La_{\sigma(\Sigma)}\subset \S^{n-1}$ are limit sets of $\Sigma$ and $\sigma(\Sigma)$  respectively and Hausdorff dimensions  are computed with respect to a Riemannian metric on $\S^{n-1}$ \cite[Theorem 1.1]{KMO_HD}.
 
For a pair $(p, q)$ of positive real numbers, let $d_{p, q}$ be the premetric on $\S^{n-1} \times \S^{n-1}$ defined as $$d_{p, q}(\xi, \eta) = d_{\S^{n-1}} (\xi_1, \eta_1)^p d_{\S^{n-1}}(\xi_2, \eta_2)^q$$ where $\xi = (\xi_1, \xi_2), \eta = (\eta_1, \eta_2) \in \S^{n-1} \times \S^{n-1}$ and $d_{\S^{n-1}}$ is a Riemannian metric on $\S^{n-1}$. We also denote by $\dim_{p, q}$ the Hausdorff dimension with respect to $d_{p, q}$. 
Let $\delta_{p,q}(\sigma)$ denote the critical exponent of the series
$$s\mapsto \sum_{\ga \in \Sigma} e ^{ - s(pd_{\H^n}(o, \ga o) + q d_{\H^n}(o, \sigma(\ga) o))}.$$
We deduce the following:
\begin{cor}\label{analytic23}
Let $\Sigma$ be a non-elementary convex cocompact subgroup of $\so(n, 1)$, $n\ge 2$. Let $p,q$ be positive real numbers. 

\begin{enumerate}
    \item 
For any $\sigma \in \op{CC}(\Sigma)$, we have
 $$\dim_{p, q}\La_{\sigma} =\delta_{p,q}(\sigma).$$

\item For any $\sigma \in \op{CC}(\Sigma)$, we have $$\dim_{p, q} \La_{\sigma} \le \left( \frac{p}{\dim \La_{\Sigma}}  + \frac{q}{\dim \La_{\sigma(\Sigma)}}  \right)^{-1}$$ and the equality holds if and only if $\rho = \id$.

\item  Moreover the map $$\sigma \mapsto \dim_{p, q}\La_{\sigma} $$ is an analytic function on any analytic subfamily of $\op{CC}(\Sigma)$.
In particular, for $n=2,3$, it is analytic on $\op{CC}(\Sigma)$.
\end{enumerate}
\end{cor}

\begin{proof}  
Identifying the Cartan subspace $\fa$ of $\so(n, 1) \times \so(n, 1)$ with $\br^2$ and $\fa^+$ with $\R_{\ge 0}^2$,
consider the linear form $\Psi\in \fa^*$ defined by
$\Psi(u_1, u_2)= pu_1+ q u_2$.
Since  $d_{\S^{n-1}}(\xi, \eta)=
e^{-\cal G(\xi, \eta)}$ is a $\op{SO}(n)$-invariant metric and hence a Riemannian metric where
$\cal G$ is the Gromov product on $\S^{n-1}\simeq\partial \bH^n$,
we have
$d_{\Psi} = d_{p, q}$, where $d_{\Psi}$ is defined in \eqref{eqn.defconfmetric}.
Since the opposition involution $\i$ is trivial for $\so(n, 1) \times \so(n, 1)$, the linear form $\Psi$ is symmetric and hence
the claims (1) and (3) respectively follow  from  Corollary \ref{cor.crithdim} and Corollary \ref{analytic} applied to any analytic subfamily of
$ \{ (\id\times \sigma):\Sigma \to \so(n, 1) \times \so(n, 1) : \sigma \in \op{CC}(\Sigma)\}$.
For  $n=2,3$,
$\op{CC}(\Sigma)$ is known to be analytic (cf. \cite{Bers_spaces}, \cite[Theorem 10.8]{Marden_geometry}, \cite{Imayoshi_Teich}). Hence the last claim  of (3) follows.
Claim (2) follows from (1) and the following Theorem \ref{bur}.
\end{proof}

The following theorem is due to Bishop-Steger \cite[Theorem 2]{BS_rigidity} for $n=2$
and to Burger \cite[Theorem 1(a)]{Burger_manhattan} in general. We denote by $\delta_{\Sigma}$ the critical exponent of $\Sigma$, the abscissa of convergence of the Poincar\'e series $s \mapsto \sum_{\ga \in \Sigma} e^{-s d_{\H^n}(o, \ga o)}$.
\begin{theorem} \label{bur} \label{cor.feb5rigidity}
  For each $\sigma \in \op{CC}(\Sigma)$, we have $$\delta_{p, q}(\sigma) \le \left( \frac{p}{\delta_{\Sigma}}  + \frac{q}{\delta_{\sigma(\Sigma)}}  \right)^{-1}$$ and the equality holds if and only if $\sigma = \id$.
\end{theorem}

\begin{proof} We explain how to deduce this from \cite[Theorem 1(a)]{Burger_manhattan}.
 We again identify the Cartan subspace $\fa$ of $\so(n, 1) \times \so(n, 1)$ with $\br^2$. For each $i=1,2$, denote by $\delta_i$ the $\alpha_i$-critical exponent of $\Sigma_{\sigma} = (\id \times \sigma)(\Sigma)$
 where $\alpha_i : \fa \to \R$, $(u_1, u_2) \mapsto u_i$. 
 Then $\delta_1=\delta_\Sigma$ and  $\delta_{2} = \delta_{\sigma(\Sigma)}$. 
If we set $\alpha_i':= \delta_i \alpha_i$, then $\delta_{\alpha_i'}=1$ for each $i=1,2$ and hence
  Burger's theorem \cite[Theorem 1(a)]{Burger_manhattan} implies that
  the critical exponent of any convex combination of $\alpha_1'$ and $\alpha_2'$ is at most one and is equal to one only when $\sigma=\id$.

   Since $$
       p\alpha_1+q\alpha_2 = \left( \frac{p}{\delta_{1}}  + \frac{q}{\delta_{2}}  \right) \frac{\frac{p}{\delta_{1}} \alpha_1' + \frac{q}{\delta_{2}} \alpha_2'}{\frac{p}{\delta_{1}} + \frac{q}{\delta_{2}}}
    =\left( \frac{p}{\delta_{1}}  + \frac{q}{\delta_{2}}  \right) \Psi_0 $$
    where $\Psi_0 := \frac{\frac{p}{\delta_{1}}  \alpha_1' + \frac{q}{\delta_{2}}  \alpha_2'}{\frac{p}{\delta_{1}} + \frac{q}{\delta_{2}}}$  is a convex combination of $\alpha_1'$ and $ \alpha_2'$,
    we get
    \be \label{eqn.feb5convexcomb}
    \delta_{p,q}(\sigma) = \delta_{p\alpha_1+q\alpha_2}=\left( \frac{p}{\delta_{1}}  + \frac{q}{\delta_{2}}  \right)^{-1} \delta_{\Psi_0} \le \left( \frac{p}{\delta_{1}}  + \frac{q}{\delta_{2}}  \right)^{-1}
    \ee
    and the equality holds if and only if $\sigma = \id$.
\end{proof}

\begin{remark}
    We remark that in Corollary \ref{analytic23},
    the hypothesis $p, q>0$ was imposed to be able to consider {\em all} $\sigma \in \op{CC}(\Sigma)$. If we replace $\op{CC}(\Sigma)$ by a subset
    $\cal D \subset \op{CC}(\Sigma)$,  then Corollary \ref{analytic23} holds for
     any $p, q \in \R$ such that $p u_1 + q u_2 > 0$ for all non-zero $(u_1, u_2) \in \L(\Sigma_\sigma)$ and all $\sigma \in \cal D$.
\end{remark}

\section{Hausdorff dimensions with respect to Riemannian metrics} \label{sec.riem}
Let $G$ be a connected semisimple real algebraic group. As before, let
$\theta$ be a non-empty subset of the set $\Pi$ of simple roots of $(\frak g, \frak a)$.
   We denote  by $$\dim \La_{\theta}$$ the Hausdorff dimension of $\La_{\theta}$ with respect to
 a Riemannian distance $d_{\rm Riem}$ on $\F_{\theta}$. As any two Riemannian metrics are bi-Lipschitz, $\dim \La_\theta$ is well-defined independent of the choice of a Riemannian metric.
In this section, we present an estimate on $\dim \La_\theta$.

\subsection*{Tits representations and the sum of Tits weights}
Let $\mathbf G$ be the semisimple algebraic group defined over $\br$ such that $G=\mathbf G(\br)^\circ$.
There exists an exact sequence $\tilde{\mathbf G}\to_{\tilde p} \mathbf G\to_{\bar p} \bar{\mathbf G}$
where $\tilde {\mathbf G}$  and $\bar{\mathbf G}$ are respectively simply connected and adjoint semisimple $\br$-groups and $\tilde p$ and $\bar p$ are central $\br$-isogenies (\cite{BT}, \cite[Proposition 1.4.11]{Mar}).

 Recall that for $\alpha \in \Pi$, $\omega_\alpha$ denotes the (restricted) fundamental weight associated to $\alpha$ as defined in \eqref{fw}.
The first part of the following theorem immediately follows as a special case of a theorem of Tits \cite{Tits_representations}, and the second part is remarked in \cite{Benoist1997proprietes} and proved in \cite{Sm}.
\begin{theorem}[{\cite[Theorem 7.2]{Tits_representations}, \cite[Lemma 2.13]{Sm}}] \label{tits} 
  For each $\alpha\in \Pi$, there exists an irreducible $\br$-representation ${\tilde{\rho}}_{\alpha}: \tilde{\mathbf{G}}\to \op{GL} (\mathbf V_{\alpha}) $ whose  highest (restricted) weight 
    $\chi_{\alpha}$ is equal to $k_\alpha \omega_\alpha$
    for some positive integer $k_\alpha$ and whose highest weight space is one-dimensional.
    Moreover, all weights of $\tilde\rho_{\alpha}$ are $\chi_\alpha$, $\chi_\alpha-\alpha$ and
  weights of the form $\chi_{\alpha} - \alpha - \sum_{\beta \in \Pi} n_{\beta} \beta$ with $n_\beta$ non-negative integers. 
\end{theorem}

For each $\alpha\in \Pi$, we fix once and for all a representation $\tilde\rho_\alpha:\tilde{\mathbf{G}}\to \op{GL} (\mathbf V_\alpha)$ as in Theorem \ref{tits} with {\it minimal} $k_\alpha$.
Since $\tilde p$ and $\bar p$ are central isogenies and $\tilde p (\tilde{\mathbf G}(\br))=G$, the representation $\tilde\rho_\alpha$ induces a projective representation 
\be\label{choice} \rho_\alpha: G\to \op{PGL}(V_\alpha)\ee 
where $V_\alpha=\mathbf V_\alpha(\br)$. Since the restriction of $\tilde \rho_\alpha$ to
$\tilde{\mathbf G}(\br)$ and $\rho_\alpha$ induce the same representation of
the Lie algebra $\mathfrak g$ to $\mathfrak{gl}(V_\alpha)$,  their restricted weights are the same. 
We call \be\label{ttt2}
\rho_\alpha \quad \text{and} \quad \chi_\alpha
\ee the Tits representation and the Tits weight associated to $\alpha$ respectively.

Let $\rho$ denote the half-sum of all positive roots for $(\fg, \fa)$ counted with multiplicity: $2\rho =\sum_{\alpha\in \Phi^+} (\dim \fg^{\alpha} ) \alpha$.
In terms of the restricted fundamental weights $\omega_\alpha$, we then have
\be \label{rhow} \rho=\sum_{\alpha\in \Pi} c_\alpha \omega_\alpha\ee 
where $c_\alpha = \dim \frak g^{\alpha}$ if $2\alpha$ is not a root, and
$c_\alpha = \frac{1}{2}(\dim \fg^{\alpha} + 2 \dim \fg^{2 \alpha}) $ otherwise (cf. \cite{Bourbaki}). If $G$ is split over $\R$, 
we have $\chi_\alpha=\omega_\alpha$ for all $\alpha\in \Pi$ and hence
$ \sum_{\alpha \in \Pi} \chi_{\alpha}=\rho$. In general, we do not have this identity, which motivates the following definition:

\begin{Def}\label{CGG} 
Define ${\mathsf c}_\theta$ to be the minimum number $c\ge 0$
such that  $$\sum_{\alpha\in \theta} (\chi_\alpha + \chi_{\i(\alpha)}) \le  c\cdot \rho \quad \text{on } \fa^+.$$
We also set $\mathsf{c}_G := \mathsf{c}_{\Pi}$.

\end{Def}
It is easy to check that $0<\mathsf{c}_{\theta} \le \mathsf{c}_G$, and moreover if $\theta \cap \i(\theta) = \emptyset$, then $\mathsf{c}_{\theta} \le \frac{\mathsf{c}_G}{2}$.
By our choice of the Tits representation of $G$, note that ${\mathsf c}_G$ depends only on the Lie algebra $\mathfrak g$; hence we sometimes write ${\mathsf c}_G=\mathsf{c}_{\mathfrak g}$.
The proof of the following lemma was provided by I. Smilga:
\begin{lem}\label{smi}
    We have 
    $\sum_{\alpha\in \Pi} \chi_\alpha \le \rho,$
    and hence $$\mathsf c_{\mathfrak g}\le 2.$$
\end{lem}

\begin{proof} Let $\mathfrak g_\mathbb C$ be the complexification of $\frak g=\op{Lie}G$ and $\fh$ be a Cartan subalgebra of $\fg_\mathbb C$
containing $\fa$.   Since $\fa\subset \fh$, we have a natural restriction map $\pi : \mathfrak{h}^* \to \fa^*$.  Recall the restricted fundamental weights $\omega_1, \cdots, \omega_s$ defined in \eqref{fw} where $s=\dim \fa $; they form a basis of $\fa^*$.
We denote by $\bar \omega_1, \cdots, \bar \omega_r$ the fundamental weights of $(\fg_{\mathbb C}, \fh)$ where $r=\dim \fh$, which were chosen compatibly with $\omega_j$'s
 so that $\pi$ sends each $\bar \omega_j$ to some linear combination $\sum_{i} c_{j, i} \omega_i$ where $c_{j, i}$ are non-negative integers. They form a basis of $\fh^*$.

Set $\rho_\mathbb C=\sum_{i=1}^r \bar \omega_i$, which is equal to the half-sum of all positive roots of $(\fg_\mathbb C, \fh)$. Then 
$$\rho=\pi( \rho_{\mathbb C})= \sum_i d_i \omega_i$$
   where  $d_i=\sum_j c_{j, i}\in \mathbb N$. For the Tits weights $\chi_i=\kappa_i \omega_i$
for $i=1,\cdots ,s$, recall that $\kappa_i$
is the smallest positive integer $\kappa$ such that $\kappa \omega_i$ is proximal, that is, its highest weight space is one-dimensional. 
In view of the Killing form, we may consider $\fa^*$ as a subset of $\mathfrak{h}^*$. We have the following facts:
\begin{itemize}
    \item a representation with the highest weight $\chi\in \mathfrak{h}^*$ is proximal if and only if $\chi$ actually lies in $\fa^*$ \cite[Theorem 6.3]{AMS_semigroups};
    \item each  coefficient $\kappa_i$ is either $1$ or $2$ \cite[Section 2.3]{Benoist_automorphism}.
\end{itemize}

We now claim that $\kappa_i\le d_i$ for all $1\le i\le s$; this implies that $\sum_{i} \chi_i
=\sum_i \kappa_i \omega_i\le \rho$. This is clear if $\kappa_i=1$, since $d_i\ge 1$. 
So suppose that $\kappa_i=2$ and let us show that $d_i\ge 2$. Then $\chi_i=2  \omega_i$ lies in $\fa^*$ and is an integral weight; hence it is equal to some linear combination $\sum_j c_j \bar \omega_j$  with non-negative integer coefficients $c_j$. Moreover the sum $\sum_j c_j$ cannot exceed $2$, because $\pi$ has to map $\sum_j  c_j \bar \omega_j$ to $2  \omega_i$, and it maps each $\bar \omega_j$ to some non-zero sum of the $\omega_k$'s. So we are left with three cases:
\begin{enumerate}
    \item[(a)] $2  \omega_i = \bar \omega_j$ for some $j$;
    \item[(b)] $2  \omega_i = \bar \omega_j + \bar \omega_{k}$ for some distinct $j$ and $k$;
    \item[(c)] $2 \omega_i = 2 \bar \omega_j$ for some $j$.
\end{enumerate}

We can rule out case (c), because then $\omega_i = \bar \omega_j$ would be proximal which contradicts $\kappa_i = 2$. In case (a), we get that $c_{j, i}=2$ and hence $d_i$ is at least $2$. In case (b), applying $\pi$ on both sides, we necessarily have $\pi(\bar \omega_j) = \pi(\bar \omega_{k}) = \omega_i$, so $c_{j, i} = c_{k, i} = 1$ and hence $d_i$ is also at least 2.
\end{proof}

The bound on $\mathsf c_{G}$ can be improved in certain cases. For example, for $\mathfrak g=\mathfrak{so}(n,1)$, $n\ge 2$,
we have $\Pi=\{\alpha\}$, $\rho=\frac{n-1}{2} \alpha$ and $\chi_\alpha=\omega_\alpha=\frac{\alpha}{2}$; hence
$\mathsf c_{\mathfrak g}=\frac{2}{n-1}$.

\subsection*{Riemannian metric on $\mathcal F_\theta$} For each $\alpha \in \Pi$, we denote by $V_{\alpha}^+$ the highest weight space of $\rho_{\alpha}$ and by $V_{\alpha}^{<}$ its unique complementary $A$-invariant subspace in $V_{\alpha}$. Then the map $g \in G \mapsto (\rho_{\alpha}(g)V_{\alpha}^+)_{\alpha \in \theta}$ factors through a proper immersion \be\label{pi} \F_{\theta} \to \prod_{\alpha \in \theta} \P(V_{\alpha}).\ee 

Let $\langle \cdot, \cdot \rangle_{\alpha}$ be a $K$-invariant inner product on $V_{\alpha}$ with respect to which $A$ is symmetric, so that $V_{\alpha}^+$ is perpendicular to $V_{\alpha}^{<}$. We denote by $\| \cdot \|_{\alpha}$ the norm on $V_{\alpha}$ induced by $\langle \cdot, \cdot \rangle_{\alpha}$. We also use the notation $\| \cdot \|_{\alpha}$ for a bi-$\rho_{\alpha}(K)$-invariant norm on $\GL(V_{\alpha})$. 
The angle $\angle (E, F)$ between a line $E$ and a subspace $F$ is defined
as minimum of all angles between all non-zero $v\in E$ and non-zero $w\in F$.

We write $g V_{\alpha}^+ := \rho_{\alpha}(g) V_{\alpha}^+$ and $g V_{\alpha}^{<} := \rho_{\alpha}(g) V_{\alpha}^{<}$ for $g \in G$ and $\alpha \in \Pi$. Up to a Lipschitz equivalence, the Riemannian distance $d_{\rm Riem}$ on $\F_\theta=G/P_\theta$
satisfies that for all $g_1, g_2 \in G$, 
$$d_{\rm Riem}(g_1P_{\theta}, g_2 P_{\theta}) = \sqrt{\sum_{\alpha \in \theta} \sin^2 \angle (g_1 V_{\alpha}^+, g_2 V_{\alpha}^+)}.$$ 

The Gromov product $\cal G$ on $\F^{(2)}$ can be expressed in terms of angles between appropriate subspaces as follows:
\begin{lemma}[{\cite[Lemma 6.4]{Quint2002Mesures}, \cite[Lemma 3.11]{LO_invariant}}] \label{lem.gromovangle}
    For  $(\xi, \eta) \in \F^{(2)}$, we have that for any $\alpha\in \Pi$,
    $$2\chi_{\alpha} (\cal G(\xi, \eta)) = -\log \sin \angle (gV_{\alpha}^+, gV_{\alpha}^{<})$$
    where $g\in G$ is such that $\xi=gP$ and $\eta=gw_0P$.
\end{lemma}

We then have the following estimates on the Riemannian distance  using Gromov products and Tits weights:

\begin{lem} \label{lem.riembdr} 
There exists a constant $C>0$ such that for all $g \in G$, $$
d_{\rm Riem}(gP_\theta , gw_0P_\theta) \ge C \left(\sum_{\alpha \in \theta} e^{-4\chi_\alpha (\cal G (g P, g w_0 P))}\right)^{1/2}.$$
\end{lem}

\begin{proof}
    We first note that for each $\alpha\in \Pi$, $w_0 V_{\alpha}^+ \subset V_{\alpha}^{<}$; to see this, recall that $V_\alpha^<$ is the sum of all weight subspaces of $V_\alpha$ whose weight
is not equal to $\chi_{\alpha}$. On the other hand, $w_0 V_{\alpha}^+$ is a weight space with the weight given by $\chi_{\alpha} \circ \op{Ad}_{w_0} = - \chi_{\alpha} \circ \i$. Since $- \chi_{\alpha} \circ \i (\fa^+) \le 0$ while $\chi_{\alpha}(\fa^+) \ge 0$, we have $\chi \circ \op{Ad}_{w_0} \neq \chi_{\alpha}$, which shows $w_0 V_{\alpha}^+ \subset V_{\alpha}^{<}$.

Therefore for all $g\in G$, 
$$ \sin^2 \angle (g V_{\alpha}^+, gV_{\alpha}^<)  \le \sin^2 \angle (g V_{\alpha}^+, gw_0  V_{\alpha}^+) $$
Hence, up to a Lipschitz constant, we have that for  all $g \in G$,
$$
\begin{aligned}
d_{\rm Riem}(g P_{\theta}, g w_0 P_{\theta}) & = \sqrt{\sum_{\alpha \in \theta} \sin^2 \angle (g V_{\alpha}^+, gw_0  V_{\alpha}^+)} \\
& \ge \sqrt{\sum_{\alpha \in \theta} \sin^2 \angle (g V_{\alpha}^+, gV_{\alpha}^<)}  
 \\ &= \left(\sum_{\alpha \in \theta} e^{-4\chi_\alpha (\cal G (g P, g w_0 P))}\right)^{1/2}
\end{aligned}
$$ 
where the last equality follows from Lemma \ref{lem.gromovangle}.
\end{proof}

\subsection*{Lower bounds}
 In the rest of this section, we assume that
 $$\text{$\Ga$ is  a $\theta$-Anosov subgroup of $G$.}$$

Since the Tits weights $ \{ \chi_{\alpha} : \alpha \in \theta \}$ form a basis
of $\fa_{\theta}^*$,
 each linear form $\psi \in \fa_{\theta}^*$ can be uniquely written as
 $\psi = \sum_{\alpha \in \theta} \kappa_{\psi, \alpha} \chi_{\alpha}$ with $\kappa_{\psi, \alpha}\in \br$.
 We consider the following height of $\psi$:
 $$\kappa_{\psi} := \sum_{\alpha \in \theta} \kappa_{\psi, \alpha} \in \br $$ 
 
Denote by $\ess_{\theta}$ the collection of all linear forms which are non-negative linear combinations of $\{\chi_{\alpha} : \alpha \in \theta\}$. That is, 
\be\label{esst}\ess_{\theta} := \{ \psi \in \fa_{\theta}^* :  \kappa_{\psi, \alpha} \ge 0 \text{ for all } \alpha \in \theta\}.\ee 
Since $\chi_\alpha > 0$ on $\inte \fa_\theta^+$ for all $\alpha\in \theta$, each non-zero $\psi\in \ess_\theta$ is positive on $\inte \fa_{\theta}^+$.
Since $\L_\theta - \{0\} \subset \inte\fa_\theta^+$ by Theorem \ref{thm.anosovbasic}(2), each non-zero $\psi\in \ess_\theta$ is positive on $\L_\theta-\{0\}$ and hence  we have the corresponding conformal premetric $d_\psi$ on $\La_\theta$ discussed in section \ref{sec.confmetric}:

\begin{lem}\label{RG}
   For any non-zero $\psi \in  \ess_{\theta}$, 
   the identity map $(\La_{\theta}, d_{\rm Riem}) \to (\La_{\theta}, d_{\psi})$ is bi-H\"older. More precisely, we have for some $c_1, c_2 > 0$ so that
   $$c_1 \cdot d_{\op{Riem}}(\xi, \eta)^{r_{\Ga, \psi}} \le d_{\psi}(\xi, \eta) \le c_2 \cdot d_{\op{Riem}}(\xi, \eta)^{\kappa_{\psi}/2} \quad \text{for all } \xi, \eta  \in \La_{\theta}$$
   where $r_{\Ga, \psi} > 0$ is defined in \eqref{eqn.holdersep11}.
\end{lem}

\begin{proof}    By \cite[Theorem 6.1]{BCLS_gafa}, there exists $c, h_{\Ga} > 0$ such that 
    $d_{\op{Riem}}(\xi, \eta) \le c  e^{-h_{\Ga} ( \xi, \eta )_e}$ for all $\xi\ne  \eta\in \La_\theta\simeq \partial \Ga$.
Together with Lemma \ref{comG}, this implies the first inequality with \be \label{eqn.holdersep11}
    r_{\Ga, \psi} := h_{\Ga}^{-1}Q_{\bar \psi}.
    \ee
    For each $\xi \neq \eta \in \La_{\theta}$, there exists $g \in G$ such that $\xi = g P_{\theta}$ and $\eta = g w_0 P_{\theta}$ (Theorem \ref{thm.anosovbasic}(4)).
 By Lemma \ref{lem.riembdr}, we have that for each $\alpha \in \theta$, up to a Lipschitz constant, \be \label{eqn.jan11}
 d_{\rm Riem}(\xi, \eta) \ge \left(\sum_{\alpha \in \theta} e^{-4\chi_\alpha (\cal G (g P, g w_0 P))}\right)^{1/2} \ge e^{-2\chi_\alpha (\cal G (g P, g w_0 P))}.
 \ee
Recalling $$d_{2 \chi_{\alpha}}(\xi, \eta) = e^{-2\chi_\alpha (\cal G (g P, g w_0 P))}$$
and writing $\psi=\sum_{\alpha\in \theta} \kappa_{\psi, \alpha}\chi_\alpha\in \ess_\theta$,  since all $\kappa_{\psi, \alpha}$ are non-negative, \eqref{eqn.jan11} implies 
   $$
   d_{\psi}(\xi, \eta) = \prod_{\alpha \in \theta} d_{2 \chi_{\alpha}}(\xi, \eta)^{\frac{\kappa_{\psi,\alpha}}{2}}
    \le \prod_{\alpha \in \theta} d_{\rm Riem} (\xi, \eta)^{\frac{\kappa_{\psi, \alpha}}{2}}=
    d_{\rm Riem}(\xi, \eta) ^{\frac{\kappa_{\psi}} 2}$$  up to a Lipschitz constant. Hence the second inequality follows. 
\end{proof}

\begin{Rmk}
Since $d_{\psi}$ and $d_{\bar \psi}$ are bi-Lipschitz (Proposition \ref{prop.bilip}),
 Proposition \ref{prop.ballinshadowpart} and the above lemma imply that
 there exist $c, R > 0$ such that for any $\xi \in \La_{\theta}$ and  $g  \in [e,\xi]$ in $\Ga$, the shadow  $O_{R}^{\theta}(o, g o) \cap \La_{\theta}$ contains the Riemannian ball of center $\xi$ and of radius
$c e^{-\frac{2}{\kappa_{\psi}}\d_{\psi}(o, go)}$.  
\end{Rmk}

\begin{theorem} \label{thm.riemlower}
    For any non-zero $\psi \in  \ess_{\theta}$, we have $$r_{\Ga, \psi} \cdot \dim_{\psi} \La_{\theta} \ge \dim \La_{\theta} \ge 
     \frac{\kappa_{\psi}}{2} \cdot \dim_{\psi}\La_\theta.$$
     In particular,
     $$r_{\Ga, \psi} \cdot \delta_{\bar \psi} \ge \dim \La_{\theta} \ge \frac{\kappa_{\psi}}{2} \cdot \delta_{\bar \psi}.$$

\end{theorem}

\begin{proof} It follows from Lemma \ref{RG} and a standard property of Hausdorff dimension
that  we get $$r_{\Ga, \psi} \cdot \dim_{\psi} \La_{\theta} \ge \dim \La_{\theta} \ge   \frac{\kappa_{\psi}}{2} \cdot \dim_{\psi}\La_\theta.$$
   Since $\dim_{\psi}\La_\theta=  \delta_{\bar \psi}$ by Corollary \ref{cor.crithdim}, the claim follows.
\end{proof}

Applying Theorem \ref{thm.riemlower} to each $\chi_{\alpha}$, $\alpha\in \theta$,
we obtain the following uniform lower bound on the Hausdorff dimension of all non-elementary $\theta$-Anosov subgroups:

\begin{corollary} \label{cor.riemlowermax}
    We have $$\dim \La_{\theta} \ge \max_{\alpha \in \theta}  \delta_{\chi_{\alpha} + \chi_{\i(\alpha)}}.$$
\end{corollary}

\begin{example}
 For $G = \PSL_n(\br)$, we have  $\Pi = \{ \alpha_1, \cdots, \alpha_{n-1} \}$ where $$\alpha_i : \op{diag}(a_1, \cdots, a_n) \mapsto a_i - a_{i+1} .$$
 Let $1\le p\le n-1$.
Since  $\chi_{\alpha_p}$ is equal to the fundamental weight $\omega_p$
which is given by $\omega_p(\op{diag}(a_1, \cdots, a_n))=  a_1 + \cdots + a_p$, we deduce from  Corollary \ref{cor.riemlowermax} that
for all non-elementary $\alpha_p$-Anosov subgroups of $\PSL_n(\br)$,
we  have
$$\dim \La_{\alpha_p} \ge \delta_{\omega_p+ \omega_{n-p}}.
$$
When $p = 1$, this lower bound is obtained in \cite[Theorem 10.1]{DK_patterson}.
\end{example}

 The following upper bound in Proposition \ref{prop.psllower} was obtained in (\cite[Theorem B]{PSW_Lipschitz}, \cite[Theorem 1.2]{CZZ_entropy}) for $G = \PSL_n(\br)$ and $\theta$ is a singleton. 
 \begin{prop} \label{prop.psllower}
We  have
$$ \dim \La_{\theta} \le  \max_{\alpha\in \theta} 
\delta_{\alpha} .$$
\end{prop}

\begin{proof} 
Via the proper immersion of $\F_{\theta}$ into $\prod_{\alpha \in \theta} \P(V_{\alpha})$ as discussed in \eqref{pi},
 we may consider the following metric on $\F_{\theta}$: for $g_1, g_2 \in G$, $$d_{\F_\theta} (g_1 P_{\theta}, g_2 P_{\theta}) = \max_{\alpha \in \theta} d_{\P(V_{\alpha})}(g_1 V_{\alpha}^+, g_2 V_{\alpha}^+)$$ where $d_{\P(V_{\alpha})}$ is the metric on $\P(V_{\alpha})$ given by $d_{\P(V_{\alpha})}(v_1, v_2) = \sin \angle (v_1, v_2)$. Then $d_{\F_{\theta}}$ is Lipschitz equivalent to the Riemannian distance on $\F_\theta$ and hence we can use $d_{\F_{\theta}}$ to compute $\dim \La_{\theta}$.

Fix $\alpha\in \theta$ and consider the Tits representation $(\rho_\alpha, V_\alpha)$. We write  $V_\alpha=\br^{n_\alpha}$ and $\PGL(V_{\alpha})=\PGL_{n_\alpha}(\br)$ by fixing a basis.
  We denote by $\beta_{1,\alpha}$ the  simple root of $\PGL_{n_\alpha}(\br)$ given by $\beta_{1,\alpha}(\op{diag}(u_1, \cdots, u_{n_\alpha}))=u_1-u_2.$
     Since  the highest weight of $\rho_\alpha$ is $\chi_{\alpha}$ and the second highest weight is  $\chi_{\alpha}-\alpha $ by Theorem \ref{tits}, we have that for all $\ga\in \Ga$,
\be \label{eqn.jan17}
\beta_{1,\alpha}(\mu(\rho_{\alpha}(\ga))) = \alpha(\mu(\ga)).
\ee
 Since $\Ga$ is an $\{\alpha\}$-Anosov subgroup of $G$, 
there exists $C>1$ such that  for all $\ga \in \G$, $\alpha(\mu(\ga)) \ge C^{-1}|\ga|-C $, and hence 
$\beta_{1,\alpha}(\mu(\rho_{\alpha}(\ga))) \ge C^{-1}|\ga|-C$. Therefore
$\rho_\alpha (\Ga)$ is a
$\{\beta_{1, \rho}\}$-Anosov subgroup of $\PGL_{n_\alpha}(\br)$.

We denote by $f_{\alpha} : \partial \Ga \to \P(V_{\alpha})$ the $\rho_{\alpha}(\Ga)$-equivariant embedding obtained as the extension of the orbit map of $\rho_{\alpha}(\Ga)$ (Theorem \ref{thm.anosovbasic}(4)).
It is shown in \cite[Proposition 3.5, Proposition 3.8]{PSW_Lipschitz} that there exists a constant $C_{\alpha} > 0$ such that for each $\ga \in \Ga$, there exists a ball $\B_{\alpha}(\ga)$ of radius $C_{\alpha} e^{-\beta_{1, \rho}(\mu(\rho_{\alpha}(\ga)))}$ in $\P(V_{\alpha})$ so that for any $x \in \partial \Ga$ such that $\ga \in [e, x]$ in $\Ga$, we have $f_{\alpha}(x) \in \B_{\alpha}(\ga)$. In particular, for every $k \ge 1$, the collection $$\{ \B_{\alpha}(\ga) : \ga \in \Ga, \ |\ga| = k \}$$ covers the limit set of $\rho_{\alpha}(\Ga)$ in $\P(V_{\alpha})$. Hence $\La_{\theta}$ is covered by the collection $$\left \{ \prod_{\alpha \in \theta} \B_{\alpha}(\ga) : \ga \in \Ga, |\ga| = k \right\}$$ via the immersion $\F_{\theta} \to \prod_{\alpha \in \theta} \P(V_{\alpha})$. Since  $\prod_{\alpha \in \theta} \B_{\alpha}(\ga)$ has $d_{\F_\theta}$-diameter at most 
$$\max_{\alpha \in \theta} C_{\alpha} e^{-\beta_{1, \rho}(\mu(\rho_{\alpha}(\ga)))} \le C e^{-\min_{\alpha \in \theta} \alpha(\mu(\ga))}$$ where $C = \max_{\alpha \in \theta} C_{\alpha}$ by \eqref{eqn.jan17}, we have that for each $s > 0$, the $s$-dimensional Hausdorff measure $\cal H^s(\La_\theta)$ with respect to 
$d_{\cal F_\theta}$ satisfies $$ \cal H^s(\La_\theta) \le \limsup_{k \to \infty} C^s \sum_{\ga \in \Ga, |\ga| = k} e^{-s \min_{\alpha \in \theta} \alpha(\mu(\ga))}.$$ 
Therefore, denoting by $\delta_{\min_{\alpha \in \theta} \alpha}$ the abscissa of convergence of the series $s \mapsto \sum_{\ga \in \Ga} e^{-s \min_{\alpha \in \theta} \alpha(\mu(\ga))}$,
if $s > \delta_{\min_{\alpha \in \theta} \alpha}$,
we have $ \cal H^s(\La_\theta)=0$ and hence 
 $$\dim \La_{\theta} \le \delta_{\min_{\alpha \in \theta} \alpha}.$$
On the other hand, we have $$ \frac{1}{\# \theta} \sum_{\alpha \in \theta} \sum_{\ga \in \Ga} e^{- s \alpha(\mu(\ga))} \le \sum_{\ga \in \Ga} e^{-s \min_{\alpha \in \theta} \alpha(\mu(\ga))} \le \sum_{\alpha \in \theta}  \sum_{\ga \in \Ga} e^{-s \alpha(\mu(\ga))}.$$ The first inequality implies $\max_{\alpha \in \theta} \delta_{\alpha} \le \delta_{\min_{\alpha \in \theta} \alpha}$ and the second gives $\delta_{\min_{\alpha \in \theta} \alpha} \le \max_{\alpha \in \theta} \delta_{\alpha}$. Hence $\delta_{\min_{\alpha \in \theta} \alpha} = \max_{\alpha \in \theta} \delta_{\alpha}$, which completes the proof.
\end{proof}

Theorem \ref{m4} is a combination of Corollary \ref{cor.riemlowermax} and Proposition \ref{prop.psllower}.

\section{Growth indicator bounds and applications to the $L^2$-spectrum}\label{fff}
As before, let $\Ga<G$ be a  $\theta$-Anosov subgroup where $G$ is a connected semisimple real algebraic group. In this final section, we deduce bounds on the growth indicator $\psi_\Ga^\theta:\fa_\theta\to [0, \infty)\cup\{-\infty\}$ of $\Ga$ (see Definition \ref{def.growthindicator}) from Corollary \ref{cor.riemlowermax}.
Recall Tits weights $\chi_\alpha$, $\alpha\in \Pi$, of $G$ from \eqref{ttt2}. We have the following (Corollary \ref{maincor5}):

\begin{corollary} \label{cor.growthbysumchi}\label{simple} 
    We have
    \be\label{upper1}
     \psi_{\Ga}^{\theta \cup \i(\theta)} \le \dim \La_{\theta} \cdot \min_{\alpha \in \theta} (\chi_{\alpha} + \chi_{\i(\alpha)}).
     \ee 
Moreover,
\be\label{p10} 
     \psi_{\Ga} \le \dim \La_{\theta} \cdot \min_{\alpha \in \theta} (\chi_{\alpha} + \chi_{\i(\alpha)}).
\ee 
\end{corollary}
\begin{proof}
  For any linear form $\psi\in \fa_{\theta \cup \i(\theta)}^*$ positive on $\L_{\theta \cup \i(\theta)} - \{0\}$, the scaled linear from $\delta_{\psi} \psi$ is tangent to the growth indicator (Lemma \ref{lem.tangentcritexp}).
  Hence it follows from Corollary \ref{cor.riemlowermax} that for each $\alpha \in \theta$, we have
  $$
  \psi_{\Ga}^{\theta \cup \i(\theta)} \le \delta_{\chi_{\alpha} + \chi_{\i(\alpha)}} \cdot (\chi_{\alpha} + \chi_{\i(\alpha)}) \le \dim \La_{\theta} \cdot (\chi_{\alpha} + \chi_{\i(\alpha)}) \quad \text{on } \fa_{\theta \cup \i(\theta)}.
  $$
  Therefore taking minimum among $\alpha \in \theta$ finishes the proof of \eqref{upper1}.

By \cite[Lemma 3.12]{KOW_indicators}, we have 
$$\psi_{\Ga} \le \psi_{\Ga}^{\theta \cup \i(\theta)} \circ p_{\theta \cup \i(\theta)}\quad\text{on } \fa.$$
Hence by \eqref{upper1}, we have $$\psi_{\Ga} \le \dim \La_{\theta} \cdot \min_{\alpha \in \theta} (\chi_{\alpha} + \chi_{\i(\alpha)}) \circ p_{\theta \cup \i(\theta)}.$$
Since the linear form $\chi_{\alpha} + \chi_{\i(\alpha)} \in \fa_{\theta \cup \i(\theta)}^*$ is $p_{\theta \cup \i(\theta)}$-invariant for each $\alpha \in \theta$, \eqref{p10} follows.
\end{proof}

Observing 
$$\min_{\alpha \in \theta} (\chi_{\alpha} + \chi_{\i(\alpha)} ) \le \frac{1}{\# \theta} \sum_{\alpha \in \theta} \chi_{\alpha} + \chi_{\i(\alpha)},$$
Corollary \ref{cor.growthbysumchi} implies the following:

\begin{cor} 
For any 
$\theta$-Anosov subgroup of $G$, we have
\be \label{ppp0}
\psi_{\Ga} \le \frac{ {\mathsf c}_\theta \dim \La_{\theta}}{\# \theta }\cdot \rho .
\ee
\end{cor}

\begin{remark}
    We remark that our proof shows that $\frac{{\mathsf c}_\theta }{ \# \theta }$ in the above corollary
can be replaced by the minimum $c\ge 0$ such that
$\min_{\alpha \in \theta} ( \chi_{\alpha} + \chi_{\i(\alpha)} ) \le  c \cdot \rho  $ on the limit cone $\L$. \end{remark}

 Define the real number $\lambda_0(\Ga\ba X) \in [0,\infty)$ as follows: \be\label{ll} \lambda_0(\Ga\ba X):= \inf\left\lbrace \frac{\int_{\Gamma\ba X}\|\text{grad} \, f \|^2\,d\vol}{\int_{\Gamma\ba X}|f|^2\,d\vol}\,:\,f\in C^\infty_c(\Gamma\ba X),\; f\neq 0 \right\rbrace .\ee 
This number is equal to  the bottom of the $L^2$-spectrum of $\Ga \ba X$ of the Laplace-Beltrami operator \cite[Theorem 2.2]{Sullivan_Riemannian}. 
The following was proved in \cite[Theorem 1.6]{EO_temperedness} for $\Pi$-Anosov subgroups and in \cite[Corollary 3]{LWW} in general:
 \begin{theorem} \label{eo}
 If $\Ga<G$ is a torsion-free discrete subgroup of $G$ with $\psi_\Ga \le \rho$, then  $L^2(\Ga \ba G)$ is tempered and $\la_0(\Ga\ba X) = \| \rho\|^2$.
 \end{theorem}

Applying Theorem \ref{eo},
we obtain the following (Corollary \ref{cortempered}) from \eqref{ppp0}.

\begin{corollary} \label{cor.feb2}\label{final} Let  $\Ga$ be a  torsion-free $\theta$-Anosov subgroup. 
  If $ \dim \La_{\theta}  \le \frac{\# \theta}{ {\mathsf c}_\theta}$,
then $L^2(\Ga \ba G)$ is tempered and $\la_0(\Ga\ba X) = \| \rho\|^2$.
\end{corollary}

Moreover $\la_0$ is not an $L^2$-eigenvalue (\cite{EO_temperedness}, \cite{EFLO}, see also \cite[Corollary 5.2]{Weich_Wolf} for the absence of any principal joint $L^2$-eigenvalues as well).

\begin{remark}\label{lpint}
   Indeed, it is shown in \cite[Theorem 11]{LWW} that if $\psi_\Ga \le (2-\frac{2}{p}) \rho$ for some $p \ge 1$,
   then $L^2(\Ga\ba G)$ is strongly $L^{p+\e}$-integrable for all $\e>0$, that is,
   for a dense subset of vectors, the associated matrix coefficients belong to $L^{p+\e}(G)$. Hence if $ \dim \La_{\theta}  \le (2- \frac{2}p) \frac{\# \theta}{ {\mathsf c}_\theta}$, we obtain that
   $L^2(\Gamma\ba G)$ is strongly $L^{p+\e}$-integrable for all $\e>0$.
\end{remark}

\begin{remark}\label{ffinal}
Using that $\mathsf c_G = \frac{2}{n-1}$ for $G = \so(n, 1)$, Corollary \ref{cor.feb2} says that for a Zariski dense convex cocompact $\Ga < \so(n, 1)$, if $\dim \La \le \frac{n-1}{2}$, then $L^2(\Ga \ba \so(n, 1))$ is tempered and $\la_0(\Ga \ba \H^n) = \frac{(n-1)^2}{4}$, as shown by Sullivan \cite[Theorem 2.21]{Sullivan_Riemannian}.
\end{remark}


\end{document}